\DeclareSymbolFont{AMSb}{U}{msb}{m}{n}
\documentclass[11pt,a4paper,reqno, noamsfonts]{amsart}
\usepackage[utopia,  cal = cmcal]{mathdesign}

\def\ignore#1{\relax}

\ignore{
\usepackage{etoolbox}
\makeatletter
\def\new@mathgroup{\alloc@8\mathgroup\mathchardef\@cclvi}
\patchcmd{\document@select@group}{\sixt@@n}{\@cclvi}{}{}
\patchcmd{\select@group}{\sixt@@n}{\@cclvi}{}{}
\makeatother
}

\usepackage{bm}

\usepackage{youngtab}  

\usepackage{color}

\usepackage{amsthm}
\usepackage{enumitem}
\usepackage{fullpage}

\usepackage[svgnames]{xcolor}
\usepackage[plainpages=false, pdfpagelabels,  colorlinks=true, pdfstartview=FitV, linkcolor=DarkBlue, citecolor=DarkRed, urlcolor=blue]{hyperref}

\usepackage{tikz}
\usepackage{graphicx}

\def\deltabold{{\bm \delta}}
\def\qbold{{\bm q}}
\def\zbold{{\bm z}}
\newcommand\boldq{\qbold}
\def\p #1{ \bm {#1}}
\def\pbar #1{\overline{\p {#1}}}

\ignore{
\renewcommand{\ge}{\geqslant}
\renewcommand{\ge}{\geqslant}

\renewcommand{\le}{\leqslant}

\renewcommand{\unrhd}{\trianglerighteqslant}

\renewcommand{\preceq}{\preccurlyeq}
}

\theoremstyle{plain}
\newtheorem{theorem}{Theorem}[section]
\newtheorem{proposition}[theorem]{Proposition}
\newtheorem{lemma}[theorem]{Lemma}
\newtheorem{corollary}[theorem]{Corollary}

\newcounter{saveenumi}

\theoremstyle{definition}
\newtheorem{definition}[theorem]{Definition}

\newtheorem{remark}[theorem]{Remark}
\newtheorem{notation}[theorem]{Notation}
\numberwithin{equation}{section}

 \DeclareMathOperator{\A}{A}
 
 \DeclareMathOperator{\Z}{\mathbb{Z}}
 \DeclareMathOperator{\Hom}{Hom}

\DeclareMathOperator{\Shape}{Shape}
\DeclareMathOperator{\Res}{Res}
\DeclareMathOperator{\Ind}{Ind}\DeclareMathOperator{\Span}{span}
\DeclareMathOperator{\id}{id}
\DeclareMathOperator{\row}{row}
\DeclareMathOperator{\col}{col}
\DeclareMathOperator{\node}{node}
\DeclareMathOperator{\shape}{shape}
\DeclareMathOperator{\spn}{span}

\DeclareMathOperator{\conn}{conn}

\newcommand{\Std}{\mathrm{Std}}
\newcommand{\SStd}{\mathrm{SStd}}
\newcommand\semistd[2]{\mathcal{T}^\SStd_{#1}(#2)}
\newcommand\std[1]{\mathcal{T}^\Std(#1)}

\newcommand{\cell}[3]{\Delta_{{#1}_{#2}}^{#3}}
\newcommand{\el}{l}

\newcommand\generator[2]{\delta_{#1}^{#2}}
\newcommand\inv{^{-1}}
\newcommand{\dd}[3]{d_{{#1} \to {#2}}^{\power {#3}}}
\newcommand{\uu}[3]{u_{{#1} \to {#2}}^{\power {#3}}}
\newcommand{\ddbar}[3]{\bar d_{{#1} \to {#2}}^{\power {#3}}}
\newcommand{\uubar}[3]{\bar u_{{#1} \to {#2}}^{\power {#3}}}

\input xy
\xyoption{all}

\def\mf{\mathfrak}
\def\la{\lambda}
\def\mft{\mathfrak t}
\def\mfs{\mathfrak s}

\def\mfv{\mathfrak v}

\def\power #1{^{(#1)}}
 \def\spp{\power}
\def\leftbrace{\big\{}
\def\rightbrace{\big\}}
\def\ignore#1{\relax}

\def\Q{\mathbb Q}

\def\bmw #1{W_{#1}}

\begin{document}\usetikzlibrary{matrix,arrows,decorations.pathreplacing,backgrounds,decorations.markings}

\title[Cellular bases]
{Cellular bases for algebras with a Jones basic construction}

\author[J. Enyang]{John Enyang}
\address{Department of Mathematics, City University London,   London, United Kingdom}
\email{john.enyang.1@city.ac.uk}

\author[F.M.Goodman]{Frederick M. Goodman}
\address{Department of Mathematics, University of Iowa, Iowa City, IA, USA}
\email{frederick-goodman@uiowa.edu}

\begin{abstract}
We define a method which produces explicit cellular bases for algebras obtained via a Jones basic construction. For the class of algebras in question, our method gives formulas for generic Murphy--type cellular bases indexed by paths on branching diagrams and compatible with restriction and induction on cell modules. The construction given here allows for a uniform combinatorial treatment of cellular bases and representations of the  Brauer, Birman--Murakami--Wenzl, Jones--Temperley--Lieb, and partition algebras, among others.




\end{abstract}

\thanks{The first author is grateful to Arun Ram and Andrew Mathas for many stimulating conversations during the course of this research and for detailed comments on previous versions of this paper. The second author thanks Andrew Mathas for many helpful conversations and the School of Mathematics and Statistics at the University of Sydney for its hospitality.   We thank Steffen K\"onig and the Institute for Algebra and Number Theory at the Univerity of Stuttgart for hospitality.
This research was supported by the Australian Research Council (grant ARC DP--0986774) at the University of Melbourne and (grants ARC DP--0986349, ARC DP--110103451) at the University of Sydney.}


\maketitle

\setcounter{tocdepth}{1}
\tableofcontents
\setcounter{tocdepth}{3}

\section{Introduction}
The notion of cellularity was introduced by Graham and Lehrer~\cite{MR1376244} as a tool for studying the modular representation theory  of Hecke algebras and other algebras with geometric connections.
 Cellular algebras are defined by the existence of a \textit{cellular basis} with combinatorial properties that reflect the Robinson--Schensted correspondence in the Iwahori--Hecke algebra of the symmetric group. 
 From the cellular basis, one obtains a family of modules known as cell modules;  all simple modules of a cellular algebra occur  as quotients of the cell modules.    
 Important examples of cellular algebras include the Iwahori--Hecke algebras of the symmetric groups, 
 Brauer algebras, Birman--Murakami--Wenzl algebras, Jones--Temperley--Lieb algebras and  
 partition algebras~\cite{MR1376244, MR1327362, MR1711582, MR1784677}.  
 
 A cellular algebra always has many different cellular bases,  and basis--free characterizations of cellularity ~\cite{MR1753809, MR2794027}  are also helpful for some purposes.  However, particular cellular bases with special properties play an important role in applications of cellularity. In particular, the Murphy basis ~\cite{MR1327362}  of the Iwahori--Hecke algebra $H_n(q^2)$ of $\mathfrak S_n$ is a cellular basis with many remarkable properties.  The triangular action of  the set of Jucys--Murphy elements of the Hecke algebra on the 
 Murphy basis   allows the construction of the seminormal representations and the classification of simple modules and blocks,  see ~\cite{MR1711316}, Chapter 3.  Several papers in the literature have aimed at generalizations or axiomatizations of the Murphy basis, the seminormal basis,  and the set of  Jucys--Murphy elements, for example ~\cite{MR2414949, MR2774622, MR2348099, MR2542221}.  The present paper is also a contribution to this theme.

 Several fundamental  examples of cellular algebras actually occur in towers, that is increasing sequences  $(A_n)_{n\ge 0}$  of algebras with a common identity, with coherent cellular structures.  Coherence means that a cell module of $A_n$, induced to $A_{n+1}$ or restricted to $A_{n-1}$ has a  {\em cell filtration},  that is, a filtration with cell modules as subquotients.  The prototypical example of a coherent tower of cellular algebras is the sequence of Hecke algebras $H_n(q^2)$.   
 The idea of coherence of cellular structures  was introduced in  ~\cite{MR2794027, MR2774622}, where it was used to study cellularity of a tower of algebras 
 $(A_n)_{n\ge 0}$ which is obtained from another tower of algebras $(H_n)_{n\ge 0}$ by repeated Jones basic constructions.  An example of such a pair of towers of algebras is the following:  $(A_n)_{n\ge 0}$ is the sequence of Birman-Wenzl-Murakami algebras, and $(H_n)_{n\ge 0}$ is the sequence of Hecke algebras, $H_n = H_n(q^2)$.  
 
 An innovation in this paper is to use a variant of the notion of cellularity:  a cellular algebra $A$ is called {\em cyclic cellular}  if all of its cell modules are cyclic $A$--modules. 
 Although cyclic cellularity is nominally stronger than cellularity, in fact most important classes of cellular algebras appearing in representation theory are cyclic cellular.  In particular,  the Hecke algebras  $H_n(q^2)$  are cyclic cellular.  

In this paper, we study coherent towers $(A_n)_{n\ge 0}$ of cyclic cellular algebras.  We first obtain some rather simple general results about cellular bases in such towers, in \hyperref[section: bases in coherent towers]{Section~\ref*{section: bases in coherent towers}}.    First we observe that there exists a system of ``branching factors''  associated to each edge of the generic branching diagram for the tower.  Then we note that an ordered product of branching factors along paths on the generic branching diagram determine bases of each cell module of each $A_n$ as well as a cellular basis of each $A_n$.   The bases obtained are ``families of path bases,'' in the sense of \hyperref[definition: path bases]{Definition~\ref*{definition: path bases}}.  
Consequently, when the tower $(A_n)_{n\ge 0}$ has a family of Jucys--Murphy elements in the sense of 
~\cite{MR2774622}, these elements act triangularly on the path bases, by 
~\cite{MR2774622}, Propositions 3.6 and 3.7.  Hence,   Mathas' theory of cellular algebras with Jucys--Murphy elements  and seminormal representations ~\cite{MR2414949} can be applied.   

	In \hyperref[section: Hecke algebra example]{Section~\ref*{section: Hecke algebra example}}, we recall that the sequence of Iwahori--Hecke algebras of the symmetric groups is a coherent tower of cyclic cellular algebras.   We compute  branching factors for reduced and  induced cell modules.  We show that the path bases obtained via ordered products of branching factors coincide with the Murphy bases ~\cite{MR1327362}. 

In \hyperref[framework axioms]{Section~\ref*{framework axioms}}, we return to the study of pairs of towers of algebras $(A_n)_{n\ge 0}$ and $(H_n)_{n\ge 0}$, 
where the algebras $A_n$ are obtained by repeated Jones basic constructions from the algebras $H_n$.  We augment the framework which was established in  ~\cite{MR2794027, MR2774622} for such pairs of towers by the assumption that the algebras $H_n$ are cyclic cellular.  It follows easily from the previous work in  ~\cite{MR2794027, MR2774622} that the tower $(A_n)_{n\ge 0}$ is a coherent tower of cyclic cellular algebras.   We show here that  branching factors, and therefore path bases for the tower $(A_n)_{n\ge 0}$ can be obtained by explicit formulas from branching factors for the tower $(H_n)_{n\ge 0}$.  

Finally, in \hyperref[section: applications]{Section~\ref*{section: applications}}, we apply our results to the Brauer algebras,  Birman--Murakami--Wenzl (BMW) algebras, Jones--Temperley--Lieb algebras, and partition algebras.  Thus, we obtain explicit path bases for these algebra which are in every respect analogues of Murphy's cellular basis of the Hecke algebras $H_n(q^2)$.   Jucys--Murphy elements are known for each of these examples (see ~\cite{MR2774622, MR3092697} and further references in these papers), so the theory of ~\cite{MR2414949} is applicable. 

A complication in our approach to the Murphy type bases is that the results of ~\cite{MR2794027, MR2774622}  do not apply to the basic construction algebras  defined over their generic ground ring, say $R_0$, but only to the algebras defined over $R_0[\deltabold\inv]$,  where $\deltabold$ is the ``loop parameter'';  see \hyperref[subsection:  cellularity and jones setting and correction]{Section~\ref*{subsection:  cellularity and jones setting and correction}}, where a mistake in   ~\cite{MR2794027, MR2774622} is discussed and corrected.   
Therefore, the Murphy type bases appear {\em a priori} to be bases only for the algebras defined over $R_0[\deltabold\inv]$.  However, as the bases are explicit, we can check for each of our examples that the Murphy type basis is actually a basis for the algebras defined over the generic ground ring $R_0$.

An illustration of the utility of the explicit ``path basis'' approach to cellularity in this paper is provided in~\cite{MR3092697}, where the Murphy-type bases of \hyperref[partition algebra basis]{Theorem~\ref*{partition algebra basis}} have been used to obtain an analogue of the Young seminormal form for partition algebras.
\ignore{\footnote{The reference on page 1746 of \cite{MR3092697} to Lemma 2.27 in this paper should be to Lemma 3.13 instead, and the reference on the same page to Theorem 5.29 in this paper should be to Theorem 6.10.} 
}
In the cases of the Brauer and BMW algebras, our results recover the Murphy type bases  obtained in~\cite{MR2348099}; however, the construction here is simpler, and does not involve computations in the braid group.  Rui and Si ~\cite{MR2542221}  used the path bases from ~\cite{MR2348099} to compute Gram determinants for cell modules of the BMW algebras, and to obtain definitive semisimplicity results.

Finally, we note that the results of Ariki and Mathas~\cite{MR1750939} and Mathas~\cite{MR2531227} on restriction and induction on cell modules of the cyclotomic hecke algebras, imply that the construction of cellular bases given here applies equally well to the  cyclotomic BMW algebras with admissible parameters. In this setting, our construction would recover the generalisation of~\cite{MR2348099} to the cyclotomic case given by Rui and Si in~\cite{MR2912472}.

\section{Preliminaries}
\subsection{Cellular algebras}
Cellular algebras were defined by Graham and Lehrer~\cite{MR1376244}. In this paper we use a slightly weaker version of cellularity which was introduced in ~\cite{MR2510050, MR2794027}.
\begin{definition}\label{c-d}
Let $R$ be an integral domain. A \emph{cellular algebra} is a tuple $(A,*,\hat{A},\unrhd,\mathscr{A})$ where
\begin{enumerate}[label=(\arabic{*}), ref=\arabic{*},leftmargin=0pt,itemindent=1.5em]
\item $A$ is a unital $R$--algebra and $*:A\to A$ is an algebra  involution, that is an $R$--linear anti--automorphism of $A$ such that $(x^*)^* = x$ for $x \in A$;
\item $(\hat{A},\unrhd)$ is a finite partially ordered set, and $\hat{A}^\lambda$, for $\lambda\in\hat{A}$, is a finite indexing set;
\item The set
\begin{align*}
\mathscr{A}=\big\{c_\mathfrak{st}^\lambda  \ \big | \  \text{$\lambda\in\hat{A}$ and $\mathfrak{s},\mathfrak{t}\in\hat{A}^\lambda$}\big\},
\end{align*}
is an $R$--basis for $A$, for which the following conditions hold:
\begin{enumerate}[label=(\alph{*}), ref=\alph{*},leftmargin=0pt,itemindent=1.5em]
\item\label{c-d-1} Given $\lambda\in\hat{A}$, $\mathfrak{t}\in\hat{A}^\lambda$, and $a\in A$, there exist 
coefficients $r_\mfv(\mft, a) \in R$, for $\mathfrak{v}\in\hat{A}^\lambda$, such that, for all $\mathfrak{s}\in\hat{A}^\lambda$, 
\begin{align}\label{r-act}
c_\mathfrak{st}^\lambda a\equiv 
\sum_{\mathfrak{v}\in\hat{A}^\lambda}
r_\mathfrak{v}(\mft, a) c_{\mathfrak{sv}}^\lambda \mod{A^{\rhd\lambda}},
\end{align}
where $A^{\rhd\lambda}$ is the $R$--module generated by
\begin{align*}
\big\{
c^\mu_\mathfrak{st}\ \big \vert \  \text{$\mu\in\hat{A}$, $\mathfrak{s},\mathfrak{t}\in\hat{A}^\mu$ and $\mu\rhd\lambda$}
\big\}.
\end{align*}
\item\label{c-d-2} If $\lambda\in\hat{A}$ and $\mathfrak{s},\mathfrak{t}\in\hat{A}^\lambda$, then $(c_\mathfrak{st}^\lambda)^*\equiv (c_\mathfrak{ts}^\lambda)\mod{A^{\rhd\lambda}}$.
\end{enumerate}
\end{enumerate}
The tuple $(A,*,\hat{A},\unrhd,\mathscr{A})$ is a \textit{cell datum} for $A$. 
\end{definition}
If $A$ is an algebra with cell datum $(A,*,\hat{A},\unrhd,\mathscr{A})$  we will frequently omit reference to the cell datum for $A$ and simply refer to $A$ as a \textit{cellular algebra}.   The basis $\mathscr{A}$ is called a {\em cellular basis} of $A$.  

\smallskip
From points 3(a) and 3(b) of the definition of cellularity, we have for $a \in A$ and $\mfs, \mft \in \hat A^\la$,
$$
a c^\la_{\mfs \mft}  \equiv \sum_{\mfv \in \hat A^\la}  r_\mfv(\mfs, a^*) c^\la_{\mfv \mft}  \mod{A^{\rhd\la}}.
$$
An order ideal $\Gamma \subset \hat A$ is a subset with the property that if $\la \in \Gamma$ and 
$\mu \unrhd \la$, then $\mu \in \Gamma$.  It follows from the axioms of a cellular algebra that for any order ideal $\Gamma$ in $\hat A$, 
$$
A^\Gamma = \Span \big\{c^\la_{\mfs \mft}  \ \big \vert \   \la \in \Gamma, \mfs, \mft \in \hat A^\la\big\}
$$
is a two sided ideal of $A$.  In particular $A^{\rhd \la}$ and
$$
A^{\unrhd \la} = \Span\ \big\{
c^\mu_\mathfrak{st}\ \big \vert \  \text{$\mu\in\hat{A}$, $\mathfrak{s},\mathfrak{t}\in\hat{A}^\mu$ and $\mu\unrhd\lambda$}
\big\}
$$
are two sided ideals.

\begin{definition} \label{definition: cell module}
Let $A$ be a cellular algebra, and $\lambda\in\hat{A}$. The \textit{cell module} $\cell A {} \lambda$  is the right $A$--module defined as follows.  As an $R$--module, $\cell A {} \lambda$ is free with basis indexed by $\hat A^\lambda$,  say $\{c^\la_\mft  \mid \mft \in \hat A^\lambda \}$.  
The right $A$--action is given by 
$$
c_\mft^\la a =  \sum_{{\mathfrak{v}\in\hat{A}^\lambda}} r_\mfv(\mft, a)  c^\la_\mfv,
$$
where the coefficients $r_\mfv(\mft, a)$  are those of Equation~\eqref{r-act}.
\end{definition}

Thus, for any $\mfs \in \hat A^\la$,  the map
$$
c^\la_\mft \mapsto c^\la_{\mfs \mft} + A^{\rhd\la}
$$
is an injective $A$--module homomorphism of the cell module $\cell A {} \la$ into $A^{\unrhd \la}/A^{\rhd\la}$.

We now mention some generalities regarding bimodules over algebras with involution.
If $A$ and $B$ are $R$--algebras with involutions denoted by $*$,  then we have a functor
$M \mapsto M^*$   from
$A$--$B$ bimodules to $B$--$A$ bimodules, as follows.  As an $R$--module, $M^*$ is just a copy of $M$ with elements marked by $*$.  The $B$--$A$ bimodule structure of $M^*$ is determined by 
$b x^* a = (a^* x b^*)^*$.    We have a natural isomorphism $M^{**}  \cong M$, via $x^{**} \mapsto x$.   In particular, taking $B$ to be $R$ with the trivial involution,  we get a functor from left $A$--modules to right $A$ modules.  Similarly, we have a functor from right $A$--modules to left $A$--modules.   (If $\Delta \subset A$ is a left or right ideal, we have two meanings for $\Delta^*$, namely application of the functor $*$, or application of the involution in $A$,  but these agree as right or left $A$--modules.)
If ${}_A M$ is a left $A$--module and  $N_A$ is a right  $A$--module,  then
$$
(M \otimes_R N)^* \cong  N^* \otimes_R M^*,  
$$
as $A$--$A$ bimodules, with the isomorphism determined by $(m \otimes n)^* \mapsto n^* \otimes m^*$. 
In particular if $M_A$ is a right $A$--module and we identify $M^{**}$ with $M$  and
$(M^* \otimes M)^*$ with $M^{*} \otimes M^{**}  = M^* \otimes M$,   then we have $(x^* \otimes y)^* = y^* \otimes x$.    

Now we apply these observations with $A$ a cellular algebra and $\cell A {} \la$ a cell module.  The assignment $$\alpha_\la : c^\la_{\mfs \mft}  + A^{\rhd \la}  \mapsto  (c^\la_\mfs)^* \otimes (c^\la_\mft)$$ determines an
$A$--$A$ bimodule isomorphism from $A^{\unrhd\la}/\A^{\rhd\la}$ to $(\cell A {} \la)^* \otimes_R \cell A  {} \la$. Moreover,
we have $*\circ \alpha_\la =  \alpha_\la \circ *$, which reflects the cellular algebra axiom
$(c^\la_{\mfs \mft})^* \equiv c^\la_{\mft \mfs} \mod  A^{\rhd \la}$.   When it is necessary to identify the algebra we are working with, we will write $\alpha_\la^A$  instead of
$\alpha_\la$.  

The importance of the maps $\alpha_\la$ for the structure of cellular algebras was stressed by K\"onig and Xi 
~\cite{{MR1753809}, MR1648638}.

\subsection{Generic ground rings}
The most important examples of cellular algebras are actually families $A^S$  of algebras defined over various integral ground rings $S$,  possibly containing distinguished elements (parameters) which enter into the definition of the algebras.   The prototypical example is the Iwahori--Hecke algebra of the symmetric group $\mathcal{H}_k=\mathcal{H}_k(q^2)$, which can be defined over any integral domain $S$ with a distinguished invertible element $q$;  see \hyperref[e-i-h]{Section~\ref*{e-i-h}} for the detailed description.

Again in the most important examples, there is a ``generic ground ring''  $R$ for $A$  with the following properties:
\begin{enumerate}
\item For any integral ground ring $S$ there is a ring homomorphism from $R$ to $S$, and the algebra over $S$ is the specialization of the algebra over $R$,  that is $A^S \cong A^R \otimes_R S$.  Likewise, the cell modules of $A^S$ are specializations of those of $A^R$, that is  $\cell  {A^S} {} \la \cong \cell  {A^R} {} \la \otimes_R S$.  
\item $R$ has characteristic zero, and if $F$ denotes the field of fractions of $R$, then
$A^F$ is split semisimple;  and the cell modules $\cell  {A^F} {} \la$ are the simple $A^F$ modules.
\end{enumerate}
For example, the generic ground ring for the Iwahori--Hecke algebra is $\Z[\boldq, \boldq\inv]$, where $\boldq$ is an indeterminant over $\Z$.  

Indeed, the entire point of the theory of cellular algebras is to provide a setting for a modular representation theory of important classes of algebras such as the Iwahori--Hecke algebras, Brauer algebras, Birman Murakami Wenzl algebras, etc.   The cell modules of $A^R$ are integrally defined versions of the simple modules of $A^F$ which specialize to $A^k$--modules  for any field $k$ (with appropriate parameters).   The simple $A^k$  modules are found as quotients of the cell modules $\cell  {A^k} {} \la$.   See ~\cite{MR1376244, MR1711316} for details.

\subsection{Equivalent cellular bases}
A cellular algebra $A$ with cell datum $(A,*,\hat{A},\unrhd,\mathscr{A})$  always admits different cellular bases $\mathscr B$. In fact,  any choice of an $R$--basis in each cell module of $A$ can be globalized to a cellular basis of $A$, see \hyperref[lemma:  globalizing bases of cell modules]{Lemma~\ref*{lemma:  globalizing bases of cell modules}}.   We say that a cellular basis 
$$\mathscr B = \leftbrace b^\la_{\mfs \mft} \ \big \vert \   \la \in \hat A \text{ and }  \mfs, \mft \in \hat A^\la \rightbrace
$$ 
is {\em equivalent} to the original cellular basis $\mathscr A$ if it determines the same ideals $A^{\unrhd \la}$ and the same cell modules as does $\mathscr A$.  More precisely, the requirement is that
\begin{enumerate}
\item for all $\la \in \hat A$, 
$$
A^{\unrhd \la} = \Span \big\{b^\mu_{\mfs \mft}  \ \big \vert \  \mu \unrhd \la \text{ and }  \mfs, \mft \in \hat A^\mu     \big\},  \text{ and}
$$
\item  for all $\la \in \hat A$ and all $\mfs \in \hat A^\la$, 
$$
\Span \big\{b^\la_{\mfs \mft} + A^{\rhd \la} \ \big \vert \  \mft \in \hat A^\la \big\} \cong \cell A {} \la,
$$
as right $A$--modules.
\end{enumerate}

\begin{lemma} \label{an isomorphism for cell modules 1}  Let $A$ be a cellular algebra and let $\la \in \hat A$.   Let $b \in \cell A {} \la$ be non--zero.  Then $x \mapsto b^* \otimes x$ is an $A$--module isomorphism of $\cell A {} \la$ onto $b^* \otimes \cell A {} \la \subseteq (\cell A {} \la)^* \otimes_R \cell A {} \la$.
\end{lemma}

\begin{proof}  Since $(\cell A {} \la)^*$ is a free $R$ module, it is torsion free; hence
$$
\cell A {} \la \cong R \otimes_R \cell A {} \la \cong  R b^* \otimes_R \cell A {} \la  = b^* \otimes \cell A {} \la.
$$
Explicitly, the isomorphism is $x \mapsto b^* \otimes x$.  
\end{proof}

\begin{lemma}[\cite{MR3065998}, Lemma 2.3]  \label{lemma:  globalizing bases of cell modules}
Let $A$ be a cellular algebra with cell datum $(A,*,\hat{A},\unrhd,\mathscr{A})$.  
For each $\la \in \hat A$,  fix  an $A$--$A$--bimodule isomorphism $\beta_\la : A^{\unrhd \la}/A^{\rhd \la} \to (\cell A {} \la)^* \otimes_R \cell A {} \la$ satisfying $*\circ \beta_\la = \beta_\la\circ *$.  
For each
$\la \in \hat A$,  let $\{ b_\mft \,\mid\, \mft \in \hat A^\la \}$  be an $R$--basis of the cell module $\cell A {} \la$.    For each $\la \in \hat A$ and each $\mfs, \mft \in \hat A^\la$,  let
$b^\la_{\mfs \mft}$ be a lifting of $\beta_\la\inv(b_\mfs^* \otimes b_\mft)$ in $A^{\unrhd \la}$.  Then
$$
\mathscr B = \leftbrace b^\la_{\mfs \mft}   \ \big \vert \  \la \in \hat A \text{ and }  \mfs, \mft \in \hat A^la \rightbrace
$$
is a cellular basis of $A$ equivalent to the original cellular basis $\mathscr A$.  
\end{lemma}

\subsection{Extensions of cellular algebras}
\begin{definition}  Suppose $A$ is a unital $R$--algebra with involution $*$, and $J$ is an $*$--invariant ideal. 
Let us say that $J$ is a {\em cellular ideal} in $A$ if it satisfies the axioms for a  cellular algebra (except for being unital) with cellular basis 
$ \{ c_{\mfs \mft}^\la  \,\mid\,  \la \in  \hat J \text{ and }  \mfs, \mft \in  \hat J^\la \} \subseteq J$
and we have, as in point (3a) of the definition of cellularity, 
\begin{align} \label{equation: condition for cellular ideal}
c_{\mfs  \mft}^\la a  \equiv \sum_v  r_\mathfrak{v}(\mft, a)  c_{\mfs \mfv}^\la  \mod  J^{\rhd \la}
\text{ for all } a \in A,
\end{align}
not only for $a \in J$.
\end{definition}

\begin{lemma} \label{lemma: extensions of cellular algebras} (Extensions of cellular algebras)  Let $A$ be an algebra with involution over an integral domain $R$.  Suppose that $J$ is a cellular ideal in $A$ and $A/J$ is a cellular algebra with the involution induced from $A$.   Then $A$ is a cellular algebra.
\end{lemma}

\begin{proof}  Write $H = A/J$.  Denote the cell datum of $J$ by $(J, *, \unrhd, \hat J, \mathscr J)$ and that of $H$ by $(H, *, \unrhd, \widehat H, \mathscr H)$.    We define a cell datum for $A$:  The partially ordered set $\hat A$ is $\hat J \cup \widehat H$ with the original partial orders on $\hat J$ and $\widehat H$ and
$\mu \rhd \la$ for all $\mu \in \hat J$ and $\la \in \widehat H$.  For $\mu \in \hat J$, we take $\hat A^\mu = \hat J^\mu$ and for $\la \in \widehat H$, we take $\hat A^\la = \widehat H^\la$.
For each  $\la \in \widehat H$ and each pair
$\mfs, \mft \in \widehat H^\la$,  let $h^\la_{\mfs \mft}$ be the corresponding element of $\mathscr H$, and let 
$\bar h^\la_{\mfs \mft}$ be any lifting of $h^\la_{\mfs \mft}$ in $A$.  Let $\overline{\mathscr H}$ be the set of all such elements $\bar h^\la_{\mfs \mft}$.  Then it is straightforward  to check that
$\mathscr A = \mathscr J \cup \overline{\mathscr H}$ is a cellular basis of $A$.    
\end{proof}

\begin{remark} \label{remark on extensions of cellular algebras}
We make some useful observations regarding the situation of \hyperref[lemma: extensions of cellular algebras]{Lemma~\ref*{lemma: extensions of cellular algebras}}:  Let $\pi: A \to A/J = H$ be the canonical map.  

For $\la \in \widehat H$,  the following statements hold:  $A^{\unrhd \la} = \pi\inv(H^{\unrhd \la})$, and likewise $A^{\rhd \la} = \pi\inv(H^{\rhd \la})$. 
Consequently, $J \subseteq  A^{\rhd \la} $ for all $\la \in \widehat H$.
We have
$A^{\unrhd \la}/A^{\rhd \la} \cong H^{\unrhd \la}/H^{\rhd \la}$  via $a + A^{\rhd \la}  \mapsto 
\pi(a) + H^{\rhd \la}$.   The cell modules $\cell A {} \la$ and $\cell H {} \la$ can be identified (by
$x a = x \pi(a)$ for $x \in \cell H {} \la$ and $a \in A$.)  The map
$\alpha_\la^A : A^{\unrhd \la}/A^{\rhd \la} \to (\cell A {} \la)^* \otimes_R \cell A {} \la$ is
\begin{align} \label{equation: alpha la A versus alpha lambda H}
\alpha_\la^A: a + A^{\rhd \la} \mapsto  \alpha_\la^H(\pi(a) + H^{\rhd \la}).
\end{align}
The ideals
$A^{\unrhd \la}$ and $A^{\rhd \la}$ and the maps $\alpha_\la^A$ are independent of the choice of the liftings $\bar h^\la_{\mfs \mft}$ of the cellular basis $\mathscr H$ of $H$.  

For $\mu \in \hat J$,  the cell modules $\cell A {} \mu$ and $\cell J {} \mu$ can be identified;  this is because of condition~\eqref{equation: condition for cellular ideal} in the definition of cellular ideals. 
We have $A^{\unrhd \mu} = J^{\unrhd \mu}  \subseteq J$, and similarly for $A^{\rhd \mu}$.
\end{remark}

\subsection{Cellular algebras with cyclic cell modules}
\label{subsection:  cellular algebras with cyclic cell modules}

\begin{definition} \label{definition:  cyclic cellular}
 A  cellular algebra is said to be {\em cyclic cellular} if every cell module is cyclic.
\end{definition}

\begin{remark}  For examples of cyclic cellular algebras, see \hyperref[section: applications]{Section~\ref*{section: applications}}.  Cyclic cellularity was also introduced in ~\cite{MR3065998}, and some additional examples, beyond those studied here are presented in that paper.
\end{remark}

\begin{lemma}[~\cite{MR3065998}, Lemma 2.5] \label{lemma: equivalent conditions for cyclic cellular algebra}
Let $A$ be a  cellular algebra with cell datum $(A,*,\hat{A},\unrhd,\mathscr{A})$.  The following are equivalent:
\begin{enumerate}
\item  $A$ is cyclic cellular.
\item  For each $\la \in \hat A$,  there exists an element $c_\la \in A^{\unrhd \la}$ with the properties:
\begin{enumerate}
\item $c_\la \equiv c_\la^* \mod{A^{\rhd \la}}$.
\item $A^{\unrhd \la} =  A c_\la A +  A^{\rhd \la}$.
\item    $(c_\la A + A^{\rhd \la})/A^{\rhd \la} \cong \cell A {} \la$, as right $A$--modules. 
\end{enumerate}
\end{enumerate}
\end{lemma}

For the remainder of this section (and commonly in the rest of the paper as well) 
 we will  adopt the following notation.   For a cyclic cellular algebra $A$ and $\la \in \hat A$,  we let   $\generator A \la$  denote a generator of the cell module $\cell A {} \la$, and
$c_\la$  a lifting to $A^{\unrhd \la}$ of $\alpha\inv((\generator A \la)^* \otimes \generator A \la)$.  
When it is necessary to identify the algebra we are working in, we write $c_\la^A$.

Let $A$ be a cyclic cellular algebra with cell datum $(A,*,\hat{A},\unrhd,\mathscr{A})$.  For each 
$\la \in \hat A$,  let $\big\{c_\mft^\la   \,\mid\,  \mft \in \hat A^\la\big\}$ be the standard basis of the cell module
$\cell A {} \la$  derived from the cellular basis $\mathscr A$ of $A$.  Since $\cell A {} \la$ is cyclic, for each
$\mft \in \hat A^\la$, there exists $v_\mft \in  A$ such that $c_\mft^\la = \generator A \la v_\mft$.  

\begin{lemma} \mbox{}  \label{lemma:  cellular basis for cyclic cellular algebra}
\begin{enumerate}
\item  For  $\la \in \hat A$ and $\mfs, \mft \in \hat A^\la$, we have 
$$
v_\mfs^* c_\la v_\mft \equiv c_{\mfs \mft}^\la \mod{A^{\rhd \la}}.
$$
\item The set $ \{ v_\mfs^* c_\la v_\mft   \,\mid\,   \la \in \hat A, \mfs, \mft \in \hat A^\la   \}$ is a cellular basis of $A$ equivalent to the original cellular basis $\mathscr A$ of $A$.  
\end{enumerate}
\end{lemma}

\begin{proof}  Point (1) holds because both $v_\mfs^* c_\la v_\mft$ and $c_{\mfs \mft}^\la$ are liftings to 
$A^{\unrhd \la}$ of $\alpha_\la\inv((c_\mfs^\la)^* \otimes c_\mft^\la)$.   Point (2) follows from (1).
\end{proof}

We record a version of \hyperref[lemma:  globalizing bases of cell modules]{Lemma~\ref*{lemma:  globalizing bases of cell modules}} that is adapted to the context of cyclic cellular algebras:

\begin{lemma}  \label{lemma: globalizing bases of cell modules -- cyclic case}
Let $A$ be a cyclic cellular algebra with cell datum $(A,*,\hat{A},\unrhd,\mathscr{A})$.
For each $\la \in \hat A$, let $\{b_\mft  \,\mid\, \mft \in \hat A^\la\}$ be an  $R$--basis of $\cell A {} \la$.
For $\mft \in \hat A^\la$,  choose $w_\mft \in A$ such that $b_\mft = \generator A \la w_\mft$.  For $\mfs, \mft \in \hat A^\la$,  let  $$b^\la_{\mfs \mft} =  (w_\mfs)^* c_\la w_\mft.$$  Then
$\mathscr B = \{ b^\la_{\mfs \mft}   \,\mid\,  \la \in \hat A \text{ and }  \mfs, \mft \in \hat A^\la  \}$ is a cellular basis of $A$ equivalent to  the original cellular basis $\mathscr A$.
\end{lemma}

\begin{proof}   For each $ \la \in \hat A$  and  $\mfs, \mft \in \hat A^\la$, 
$b^\la_{\mfs \mft}$ is a lifting in $A^{\unrhd \la}$ of $\alpha_\la\inv(b_\mfs^* \otimes b_\mft)$,   so this follows immediately from \hyperref[lemma:  globalizing bases of cell modules]{Lemma~\ref*{lemma:  globalizing bases of cell modules}}.
\end{proof}

\begin{remark}  \label{remark: extensions of cyclic cellular algebras} (Extensions of cyclic cellular algebras)  Let $A$ be an algebra with involution over $R$,  let $J$ be a cellular ideal in $A$ and suppose that $H = A/J$ is cellular.  If both $J$ and $H$ are cyclic cellular, then so is $A$.  This is evident from \hyperref[lemma: extensions of cellular algebras]{Lemma~\ref*{lemma: extensions of cellular algebras}}
and \hyperref[remark on extensions of cellular algebras]{Remark~\ref*{remark on extensions of cellular algebras}}.   

Let $\pi: A \to A/J = H$ denote the quotient map.
For each $\la \in \widehat H$, let $\generator H \la$ be a generator of the cell module $\cell H {} \la = \cell A {} \la$.  
Let $c_\la^H \in H^{\unrhd\la}$ satisfy $\alpha_\la^H(c_\la^H + H^{\rhd\la}) = ( \generator H \la)^* \otimes \generator H \la$.

Let $c_\la^A \in \pi\inv(c_\la^H)$.   Then $c_\la^A \in 
\pi\inv(H^{\unrhd\la}) = A^{\unrhd\la}$. Moreover, it
 follows from the description of 
$\alpha_\la^A$ in Equation~\eqref{equation: alpha la A versus alpha lambda H}, that 
$\alpha_\la^A(c_\la^A + A^{\rhd \la}) =  ( \generator H \la)^* \otimes \generator H \la$.
 Thus, by \hyperref[lemma: equivalent conditions for cyclic cellular algebra]{Lemma~\ref*{lemma: equivalent conditions for cyclic cellular algebra}}, 
$A^{\unrhd \la} =  A c_\la^A A + A^{\rhd \la}$, and $c_\la^A a + A^{\rhd \la} \mapsto \generator H \la a$
is an isomorphism of $(c_\la^A A + A^{\rhd \la})/A^{\rhd \la}$ to $\cell H {} \la = \cell A {} \la$.  
\end{remark}

\section{Bases in towers of cellular algebras}  \label{section: bases in coherent towers}

In this section we obtain some elementary results on bases in towers of cellular algebras.  The main results can be summarized as follows.  Consider an increasing sequence of cellular algebras $(H_n)_{n\ge 0}$ over an integral domain $R$ with field of fractions $F$.  Suppose that
 \begin{enumerate}
 \item \label{semicoherent + cyclic 1} $H_0 = R$, and $H_n^F = H_n \otimes_R F$ is split semisimple for all $n$.
  \item \label{semicoherent + cyclic 2}  For each $n \ge 0$ and each cell module $\Delta$ of $ H_{n+1}$,  the $H_n$--module  $\Res^{H_{n+1}}_{H_{n}} (\Delta)$ has an order preserving cell filtration, see 
  \hyperref[definition: cell filtration]{Definition~\ref*{definition: cell filtration}}
 \item \label{semicoherent + cyclic 3}  $H_n$ is cyclic cellular for all $n$.   
 \end{enumerate}
Then one can associate to each edge $\la \to \mu$ in the 
 branching diagram $\widehat H$ for the tower $(H_N^F)_{n \ge 0}$  of split semisimple algebras a ``branching factor" $d_{\la \to \mu}$.  The ordered product of these branching factors along paths in $\widehat H$  determines a basis of each cell module of each algebra $H_n$.   The collection of these bases is a ``family of path bases,"   which means  that the bases behave well with respect to restriction to smaller algebras in the tower, see \hyperref[definition: path bases]{Definition~\ref*{definition: path bases}}. The existence of these special bases of the cell modules depends on the existence of cell filtrations for the restricted modules $\Res^{H_{n+1}}_{H_{n}} (\Delta)$;  conversely, any family of path bases determines cell filtrations of each restricted module  $\Res^{H_{n+1}}_{H_{n}} (\Delta)$.

\subsection{Coherence conditions for towers of cellular algebras} 
\label{subsection:  coherence conditions for towers}

If $A$ is a cellular algebra over $R$, $\lambda\in\hat{A}$, and $N\subseteq M$ is an inclusion of right $A$--modules, write  $N\stackrel{\lambda}{\subseteq}M$ if $ M/N\cong  \cell A {} \la $  as right $A$--modules.

\begin{definition} \label{definition: cell filtration}  Let $A$ be a cellular algebra with cell datum
$(A, *, \unrhd, \hat A, \mathscr A)$.  
If $M$ is a right $A$--module, a  {\em cell filtration}  of  $M$  is a filtration by right $A$--modules
\begin{align*}
\{0\}= M_0\stackrel{\lambda^{(1)}}{\subseteq}  M_1\stackrel{\lambda^{(2)}}{\subseteq}\cdots\stackrel{\lambda^{(r)}}{\subseteq}   M_r=M,
\end{align*}
with subquotients isomorphic to cell modules.  Say that the filtration is {\em order preserving} if  $\lambda^{(s)}\rhd\lambda^{(t)}$ in $\hat{A}$ whenever $s < t $. 
\end{definition}

Observe  that all the modules occurring in a  cell filtration are necessarily free as $R$--modules.
Evidently, a given cell module occurs at most once as a subquotient in an order preserving cell filtration.

\medskip
Here and in the remainder of the paper, we will consider  increasing sequences
$$
H_0 \subseteq H_1 \subseteq H_2 \cdots
$$
of cellular algebras over an integral domain $R$. Whenever we have such a sequence of algebras, we 
 assume that all the inclusions are unital and that the involutions are consistent; that is the involution on $H_{i+1}$, restricted to $H_i$, agrees with the involution on $H_i$.    

\begin{definition}[\cite{MR2794027,MR2774622}]  The tower of cellular algebras $(H_i)_{i \ge 0}$ is {\em coherent}  if the following conditions are satisfied:
\begin{enumerate} 
\item  For each $i\ge 0$ and each cell module $\Delta$ of  $H_i$,  the induced module $\Ind_{H_i}^{H_{i+1}}(\Delta)$   has  cell filtration. 
\item  For each $i \ge 0$ and each  cell module $\Delta$ of $H_{i+1}$ the restricted module $\Res_{H_i}^{H_{i+1}}(\Delta)$  has a cell filtration. 
\end{enumerate}
The tower is called {\em strongly coherent} if the cell filtrations can be chosen to be order preserving.
\end{definition}

In the examples of interest to us, we will also have {\em uniqueness of the multiplicities} of the cell modules appearing as subquotients of the cell filtrations, and {\em Frobenius reciprocity} connecting the multiplicities in the two types of filtrations, see 
\hyperref[corollary:  multiplicities in cell filtrations 3]{Corollary~\ref*{corollary:  multiplicities in cell filtrations 3}}.

Only the filtrations of restricted modules $\Res_{H_i}^{H_{i+1}}(\Delta)$ play a role in this section,  but the filtrations of induced modules $\Ind_{H_i}^{H_{i+1}}(\Delta)$   also play an essential role in the study of towers of algebras with a Jones basic construction in \hyperref[section:  algebras with basic construction]{Sections~\ref*{section:  algebras with basic construction}} and \ref{section: applications}.

\subsection{Inclusions matrices, branching diagrams, and cell filtrations}  We recall the notion of an {\em inclusion matrix} for an inclusion of split semisimple algebras over a field.
Suppose
$A \subseteq B$ are finite dimensional split semisimple algebras over  a field $F$ (with the same identity element).  Let   $\{V_\la  \,\mid\,   \la \in \hat A  \}$, be the set of isomorphism classes of simple $A$--modules  and  $ \{W_\mu  \,\mid\,   \mu \in \widehat B \}$   the set of isomorphism classes of simple $B$--modules.  
We associate a  $\widehat B \times \hat A$   {\em inclusion matrix}
$\omega$ to the inclusion $A \subseteq B$, as follows.  For each $\mu \in \widehat B$, the $A$--module $\Res_A^B(W_\mu)$   has a direct sum decomposition in simple $A$--modules, with uniquely determined multiplicities, and
$\omega(\mu,\la)$ is defined to be the multiplicity of $V_\la$ in the decomposition of $\Res_A^B(W_\mu)$.   Say that the inclusion $A \subseteq B$ is {\em multiplicity--free} if the inclusion matrix has entries in $\{0, 1\}$.  

Now consider an increasing sequence $(B_n)_{n\ge 0}$  of split semisimple algebras over a field $F$.  The {\em branching diagram $\widehat B$} of the sequence $(B_n)_{n \ge 0}$  is a graph with vertex set
$\coprod_{n \ge 0}  \widehat B_n$,  where $\widehat B_n$ indexes the set of isomorphism  classes simple $B_n$--modules.  Fix $n \ge 0$ and let $\omega_n$ denote the inclusion matrix of $B_n \subseteq B_{n+1}$.  For $\la \in \widehat B_n$ and $\mu \in \widehat B_{n+1}$, the branching diagram has $\omega_n(\mu, \la)$ edges connecting $\la $ and $\mu$.  We write $\la \to \mu$ if
$\omega_n(\mu, \la) \ne 0$.  
In our examples, all the inclusions are multiplicity--free, so a single edge connects $\la$  to $\mu$ when $\la \to \mu$.  Also, in our examples we have $B_0 = F$, so that $\widehat B_0$ is a singleton.

\begin{notation}
Let $R$ be an integral domain with field of fractions $F$.  Let $A$ be a cellular algebra over $R$ and $\Delta$ an $A$--module.  Write $A^F$  for $A \otimes_R F$  and
$\Delta^F$ for $\Delta \otimes_R F$.  
\end{notation}

Let $A$ be a cellular algebra over an integral domain $R$ with field of fractions $F$, and suppose
$A^F$ is split semisimple.  Then $ \{(\Delta_A^\la)^F  \,\mid\, \la \in \hat A   \}$  is  a complete family of simple $A^F$--modules.

 \begin{lemma}[{\cite[Lemma 2.22]{MR2794027} and  \cite[Sect.~2.5]{MR2774622}}] \label{lemma: multiplicities in cell filtrations}
Let $R$ be an integral domain with field of fractions $F$.
 Suppose that $A \subseteq B$ are cellular algebras over $R$ and that
 $A^F$ and $B^F$  are split semisimple.    Let $\omega$ denote the inclusion matrix for $A^F \subseteq B^F$.
 \begin{enumerate}
  \item For any $\la \in \hat A$ and $\mu \in \widehat B$, 
  and any cell filtration of $\Res_{A}^{B}(\cell B {} \mu)$,  the number of subquotients of the filtration isomorphic to  $\cell A {}  \la$ is $\omega(\mu, \la)$.
  
  \item  Likewise, for any $\la \in \hat A$ and $\mu \in \widehat B$, 
 and any cell filtration of $\Ind_{A}^{B}(\cell A {} \la)$,  the number of subquotients of the filtration isomorphic to  $\cell B {}\mu$ is $\omega(\mu, \la)$.

 \end{enumerate}
  \end{lemma}
  
  \begin{corollary} \label{corollary:  multiplicities in cell filtrations 3}
Let $R$ be an integral domain with field of fractions $F$.
 Let $(H_n)_{n \ge 0}$  be a strongly coherent tower of cellular algebras over $R$, and suppose that
 $H_n^F$ is split semisimple for all $n$.  Then for all $n$ and for $\la \in \widehat H_n$ and $\mu \in \widehat H_{n+1}$, 
 the 
 multiplicity of $\cell {H} {n} \la$ in any cell filtration of $\Res_{H_n}^{H_{n+1}}(\cell {H} {n+1} \mu)$ equals the multiplicity of 
$\cell {H} {n+1}\mu$ in any cell filtration of $\Ind_{H_n}^{H_{n+1}}(\cell {H} {n} \la)$.  The multiplicities are independent of the choice of the filtrations.
\end{corollary}
  
  \subsection{Path bases and cell filtrations}
 We consider an increasing sequence  $(H_n)_{n \ge 0}$ of cellular algebras over an integral domain $R$ with field of fractions $F$.   We assume the following conditions are satisfied:
 \begin{enumerate}
  \item   $H_0 = R$, and $H_n^F $ is split semisimple for all $n$.
 \item  The branching diagram $\widehat H$  for the tower $(H_n^F)_{n \ge 0}$ is multiplicity free.
  \end{enumerate}
   We let $(H_n, *, \unrhd, \widehat H_n, \mathscr H_n)$ denote a cell datum for $H_n$.     Denote the unique element of $\widehat H_0$ by $\emptyset$.

 \begin{definition}
A {\em path} on $\widehat H$   from $\la \in \widehat H_\el$  to
$\mu \in  \widehat H_m$  ($\el < m$)  is a sequence  
$$(\la = \la\spp \el, \la\spp {\el + 1}, \dots,  \la\spp {m} = \mu)$$  with $\la \spp i \in \widehat H_i$ and $\la \spp i \to \la \spp {i+1}$  for all $i$.  A path $\mathfrak s$  from $\la$ to 
$\mu$  and a path $\mathfrak t$  from $\mu$ to $\nu$  can be concatenated in the obvious way;  denote the concatenation by $\mathfrak s \circ \mathfrak t$.   If
$\mathfrak t = (\emptyset = \la\spp 0, \la\spp 1, \dots, \la\spp n = \la)$ is a path from 
$\emptyset$ to $\la \in  \widehat H_n$, and $0 \le k < \el \le n$,  write $\mf t(k) = \la\spp k$, 
$\mathfrak t_{[k, \el]}$ for
the path  $(\la\spp k,  \dots, \la\spp \el)$, and  write $\mathfrak t'$ for $\mathfrak t_{[0,n-1]}$.
\end{definition}

   Since for all $n \ge 0$ and all $\mu \in \widehat H_n$,  the rank of the cell module $\cell H n \mu$  equals the dimension over $R$ of $(\cell H n \mu)^F$, and the latter is the number of paths on the branching diagram $\widehat H$ from $\emptyset$ to $\mu$,   we can take $\widehat H_n^\mu$ to be the set of such paths.  

We define a total order on paths on $\widehat H$ as follows:
 
 \begin{definition}  Let $\mfs = (\la\spp \el, \la\spp{\el +1}, \dots, \la\spp m)$   and 
 $\mft = (\mu\spp \el, \mu\spp{\el +1}, \dots, \mu\spp m)$  be two paths from  $\widehat H_\el$ to   $\widehat H_m$.  
 Say that $\mathfrak s$  precedes $\mathfrak t$ in {\em reverse lexicographic  order}  (denoted $\mathfrak s \preceq \mathfrak t$)  if $\mathfrak s = \mathfrak t$, or if  for the last index $j$ such that $\la\spp j \ne \mu \spp j$, we have $\la\spp j < \mu \spp j$ in
$\widehat H_{j}$. 
\end{definition}

\begin{definition}[\cite{MR2774622}]  \label{definition: path bases} For each $n \ge 0$ and each $\la \in \widehat H_n$,  let
$ \{b^\la_\mft  \,\mid\, \mft \in \widehat H_n^\la  \}$  be a basis of the cell module $\cell H n \la$.    The family of bases is called a {\em family of path bases} if the following condition holds:  Let $\la \in \widehat H_n$ and let $\mft \in \widehat H_n^\la$.    Fix $k < n$ and write
  $\mathfrak t_1 = \mathfrak t_{[0,k]}$, and $\mathfrak t_2 = \mathfrak t_{[k,n]}$, and $\mu = \mft(k)$.   Let $x \in H_k$, and let $ b_{\mathfrak t_1}^\mu x = \sum_{\mathfrak s} r(x; \mathfrak s, \mathfrak t_1)  b_{\mathfrak s}^\mu$.  Then
$$
 b_\mft^\la x \equiv  \sum_{\mathfrak s} r(x; \mathfrak s, \mathfrak t_1)  b^\la_{\mathfrak s \circ \mathfrak t_2},
$$ 
modulo  $\spn  \{ b^\la_\mathfrak v    \,\mid\,  \mathfrak v_{[k,n]} \succ \mathfrak t_{[k,n]}         \}$.
\end{definition}

\begin{lemma}  \label{lemma:  path bases and cell filtrations}  \mbox{}
\begin{enumerate}
\item  Suppose that for all  $n \ge 1$ and for all $\mu \in \widehat H_n$,  $\Res^{H_n}_{H_{n-1}}(\cell H n \mu)$ has an order preserving cell filtration.   Then the cell modules of the tower $(H_n)_{n \ge 0}$ have a family of path bases.
\item   Conversely, suppose we are given a family  of path bases  of the cell modules of the tower $(H_n)_{n \ge 0}$.   Then for all  $n \ge 1$ and for all $\mu \in \widehat H_n$,  $\Res^{H_n}_{H_{n-1}}(\cell H n \mu)$ has a cell filtration.   Moreover,  if  \break 
$ \{\la \in \widehat H_{n-1}  \,\mid\, \la \to \mu  \}$ is totally ordered in $\widehat H_{n-1}$,  then $\Res^{H_n}_{H_{n-1}}(\cell H n \mu)$ has an order preserving cell filtration.
\end{enumerate}
\end{lemma}

\begin{proof}  The first statement is proved in ~\cite[Proposition 2.18]{MR2774622};  we will give more concrete construction of path bases in 
\hyperref[subsection: Bases of cell modules in towers of cyclic cellular algebras]{Section~\ref*{subsection: Bases of cell modules in towers of cyclic cellular algebras}},
in the case that all the algebras $H_n$ are cyclic cellular.

  For the converse,  suppose we are given a family of path bases $ \{b_\mft  \,\mid\, \mft \in \widehat H_n^\mu  \}$  for $n \ge 0$ and for 
$\mu \in \widehat H_n$.    For $n \ge 1$ and $\mu \in \widehat H_n$,  let $\la\power 1, \dots, \la\power s$ be a list of $ \{\la \in \widehat H_{n-1}  \,\mid\, \la \to \mu  \}$, ordered so that $i < j$ if  $\la \power i \rhd \la \power j$.
Let 
$$
N_j = \spn\left\{b_\mft\ \big \vert \   \mft \in \widehat H_n^\mu \text{ and }  \mft(n-1) =   \la \power i \text{ for some } i \le j  \right\}.
$$  
It follows from the definition of a path basis that $N_j$ is an $H_{n-1}$ submodule of 
$\Res^{H_n}_{H_{n-1}}(\cell H n \mu)$, and that $N_j/N_{j-1} \cong  \cell H {n-1} {\la \power j}$.     If 
$ \{\la \in \widehat H_{n-1}  \,\mid\, \la \to \mu \}$ is totally ordered, then clearly this cell filtration is order preserving.
\end{proof}

\subsection{Cyclic cellularity and branching factors}  \label{subsection: Cyclic cellularity and branching factors}

Suppose that $A \subseteq B$ are cyclic cellular algebras over an integral domain $R$.  We have the following observations regarding cell filtrations of restricted and induced modules:

\begin{enumerate}[label=(\arabic{*}), ref=\arabic{*},leftmargin=0pt,itemindent=1.5em]  
\item \label{observation 1 regarding branching factors}
Let
$\mu \in \widehat B$ and suppose that $\Res^B_A(\cell B {} \mu)$ has a cell filtration:
\begin{align} \label{equation: ordered filtration of restricted module 0}
\{0\}=M_0\stackrel{\lambda^{(1)}}{\subseteq} M_1\stackrel{\lambda^{(2)}}{\subseteq}\cdots\stackrel{\lambda^{(r)}}{\subseteq} M_r= \Res^B_A(\cell B {} \mu).
\end{align}
Let $\generator B \mu$ be a generator of the $B$--module $\cell B {} \mu$.   Since $M_j/M_{j-1} \cong \cell A {} {\la \power j}$ is a cyclic $A$ module,  there exists an element $d_{\la \power j \to \mu}^B \in B$   such that 
$\generator B \mu  d_{\la \power j \to \mu}^B + M_{j-1}$  is a generator of $M_j/M_{j-1}$.  

\item \label{observation 2 regarding branching factors}
 Let $\la \in \hat A$ and suppose $\Ind_A^B(\cell A {} \la)$ has a cell filtration:  
\begin{align} \label{equation: ordered filtration of induced module}
\{0\}=N_0\stackrel{\mu^{(1)}}{\subseteq} N_1\stackrel{\mu^{(2)}}{\subseteq}\cdots\stackrel{\mu^{(p)}}{\subseteq} N_p= \Ind_A^B(\cell A {} \la).
\end{align}
Let $\generator A \la$ be a generator of the $A$--module $\cell A {} \la$;  then $\generator A \la \otimes 1$
is a generator of the $B$--module $\Ind_A^B(\cell A {} \la)$.   Since $N_j/N_{j-1} \cong \cell B {} {\mu \power j}$ is a cyclic $B$--module,  there exists an element $u_{\la  \to \mu \power j }^B  \in B$  such that \break
 $\generator A \la \otimes u_{\la \to \mu \power j }^B + N_{j-1}$ is a generator of $N_j/N_{j-1}$.

\end{enumerate}

We call the elements $d_{\la  \to \mu}^B$  and $u_{\la  \to \mu}^B$  {\em branching factors}.   They are not canonical, but in the examples 
in \hyperref[section: Hecke algebra example]{Section~\ref*{section: Hecke algebra example}} and  \hyperref[section: applications]{Section~\ref*{section: applications}},  it will be possible to make natural choices for these elements.  

\subsection{Bases of cell modules in towers of cyclic cellular algebras}
\label{subsection: Bases of cell modules in towers of cyclic cellular algebras}
\label{subsection: bases of cell modules for coherent cyclic towers}
Consider a tower   $(H_n)_{n \ge 0}$ of cellular algebras over an integral domain $R$ with field of fractions $F,$ satisfying conditions ~\eqref{semicoherent + cyclic 1} --~\eqref{semicoherent + cyclic 3} listed at the beginning of \hyperref[section: bases in coherent towers]{Section~\ref*{section: bases in coherent towers}}  
\ignore{
satisfying the following conditions:
 \begin{enumerate}
 \item \label{semicoherent + cyclic 1} $H_0 = R$.
  \item \label{semicoherent + cyclic 2}  For each $n \ge 0$ and each $\mu \in \widehat H_{n+1}$,   $\Res^{H_{n+1}}_{H_{n}} (\cell H {n+1} \mu)$ has an order preserving cell filtration.
 \item \label{semicoherent + cyclic 3}  $H_n^F$ is split semisimple for all $n$,  where $F$ is the field of fractions of $R$.
 \item \label{semicoherent + cyclic 4}  For all $n$,  $H_n$ is cyclic cellular.
 \end{enumerate}
 }
 We let $(H_n, *, \unrhd, \widehat H_n, \mathscr H_n)$ denote a cell datum for $H_n$.   Denote by $\emptyset$ the unique element of $\widehat H_0$.  

 Because of  assumptions~\eqref{semicoherent + cyclic 1} --~\eqref{semicoherent + cyclic 3}, and 
 ~\hyperref[lemma: multiplicities in cell filtrations]{Lemma~\ref*{lemma: multiplicities in cell filtrations}},
 \ignore{
 and \hyperref[corollary:  multiplicities in cell filtrations 3]{Corollary~\ref*{corollary:  multiplicities in cell filtrations 3}}
 } 
 there is a multiplicity--free branching diagram  $\widehat H$  associated with the tower, namely the branching diagram for the tower $(H_n^F)_{n \ge 0}$ of  split semisimple algebras over $F$.   
 The edges in the branching diagram are determined as follows: For $\la \in \widehat H_n$ and $\mu \in \widehat H_{n+1}$,  $\la \to \mu$ if and only if  $ \cell {H} {n} \la$ appears a subquotient in a cell filtration of 
 $  \Res_{H_n}^{H_{n+1}}(\cell {H} {n+1}  \mu)$.  
 
 Fix once and for all an order preserving cell filtration of 
 $\Res^{H_{n+1}}_{H_{n}} (\cell H {n+1} \mu)$  for each $n \ge 0$ and each $\mu \in \widehat H_{n+1}$:
 \begin{align} \label{equation: ordered filtration of restricted module}
\{0\}=M_0\stackrel{\lambda^{(1)}}{\subseteq} M_1\stackrel{\lambda^{(2)}}{\subseteq}\cdots\stackrel{\lambda^{(r)}}{\subseteq} M_r= \Res^{H_{n+1}}_{H_{n}}(\cell H {n+1} \mu).
\end{align}  
 Let $\generator {H_{n+1}} \mu$ be a generator of $\cell H {n+1} \mu$.     Following observation \eqref{observation 1 regarding branching factors} in 
 \hyperref[subsection: Cyclic cellularity and branching factors]{Section~\ref*{subsection: Cyclic cellularity and branching factors}},  for each edge $\la \to \mu$ in $\widehat H$,  fix an element $\dd {\la } {\mu} {n+1} \in H_{n + 1}$   such that $\generator {H_{n+1}} \mu  \dd {\la \power j } {\mu} {n+1} + M_{j-1}$   is a generator of 
 $M_j/M_{j-1}$.  
 Note that the cell modules of $H_1$ have rank 1, and we can choose all the elements $\dd \emptyset \mu 1$ for $\mu \in \widehat H_1$  to be $1$. 
 
Now fix $n \ge 1$ and $\la \in \widehat H_n$.  For each path
$\mft = (\emptyset = \la \power 0, \la \power 1, \dots, \la \power n = \la) \in \widehat H_n^\la$,  define
\begin{equation}  \label{equation: ordered product of d coefficients}
d_\mft =    {\dd {\la \power {n-1}} {\la \power n} {n}} \,
{\dd {\la \power {n-2}} {\la \power {n-1}} {n-1}}
\cdots
{\dd {\emptyset} {\la \power 1} {1}}
.\end{equation}

\begin{proposition}  \label{proposition:  basis for cell modules of coherent cyclic tower}
Let $n \ge 1$ and let $\mu \in \widehat H_n$.  Consider our chosen cell filtration of 
$\Res^{H_n}_{H_{n-1}}(\cell H n \mu)$,
\begin{align} \{0\}=M_0\stackrel{\lambda^{(1)}}{\subseteq} M_1\stackrel{\lambda^{(2)}}{\subseteq}\cdots\stackrel{\lambda^{(r)}}{\subseteq} M_r=\Res_{H_{n-1}}^{H_{n}}(\cell H {n} \mu).
\end{align}
  
\begin{enumerate}
\item For $1 \le j \le r$,   
$$\left\{{\generator {H_n} \mu} d_\mft  \ \big \vert \ \mft \in \widehat H_n^\mu  \text{ and }  \mft(n-1) \in  \{\la \power 1, \la \power 2, \dots, \la \power j  \}   \right\}
$$ 
is a basis of $M_j$. 
\item  In particular, $ \{{\generator {H_n} \mu} d_\mft  \mid   \mft \in \widehat H_n^\mu \}$ is a basis of $\cell H n \mu$.
\end{enumerate}
\end{proposition}

\begin{proof}  Evidently, statement (1) implies statement (2).  We prove both statements by induction on $n$, the case $n = 1$ being evident.  Fix $n > 1$ and suppose the statements hold for cell modules of 
$H_k$  for $1 \le k \le n-1$.  For each $i$ we have an isomorphism of $H_{n-1}$--modules
$$
\varphi_i: {\generator {H_{n}} \mu}  {\dd {\lambda \power i} \mu {n}} h + M_{i-1} \mapsto
\generator {H_{n-1}} {\la \power i} h
$$
from $M_i/M_{i-1}$ to   $\cell H {n-1}  {\la \power i}$.  By the induction hypothesis,
$\{ {\generator {H_{n-1}} {\la \power i}} d_\mfs  \,\mid\, \mfs \in \widehat H_{n-1}^{\lambda \power i}  \}$  is a basis of the cell module $\cell H {n-1} {\la \power i}$.   Pulling back this basis via $\varphi_i$, we get that
$$
\leftbrace {\generator {H_{n}} \mu}  {\dd {\lambda \power i} \mu {n}}  d_\mfs  + M_{i-1}   \ \big \vert\    \mfs \in  \widehat H_{n-1}^{\la \power i} \rightbrace
$$
is a basis of $M_i/M_{i-1}$.  It follows that for each $j$,
$$\left\{{\generator {H_{n}} \mu}  {\dd {\lambda \power i} \mu {n}}  d_\mfs    \ \big \vert\  1 \le i \le j \text{ and }  \mfs \in  \widehat H_{n-1}^{\la \power i} \right\}
$$
is a basis of $M_j$.  But this basis is equal to 
\[
\left\{{\generator {H_n} \mu} d_\mft  \ \big \vert\  \mft \in \widehat H_n^\mu  \text{ and }  \mft(n-1) \in \big\{\la \power 1, \la \power 2, \dots, \la \power j  \big\}   \right\}.
\]
 This proves statement (1),   and statement (2) follows.
\end{proof}

\begin{corollary}\label{label:1} For each $n$ and $\la \in \widehat H_n$,   let $c_\la$ be a lifting in $H_n^{\unrhd \la}$ of
$\alpha_\la\inv((\generator {H_n} \la)^* \otimes \generator {H_n} \la)$. 
Then
$$
\leftbrace d_\mfs^*  c_\la d_\mft    \ \big \vert\   \la \in \widehat H_n \text{ and }  \mfs, \mft \in \widehat H_n^\la \rightbrace
$$
is a cellular basis of $H_n$ which is equivalent to the original cellular basis  $\mathscr H_n$.
\end{corollary}

\begin{proof} Follows from \hyperref[proposition:  basis for cell modules of coherent cyclic tower]{Proposition~\ref*{proposition:  basis for cell modules of coherent cyclic tower}} and \hyperref[lemma: globalizing bases of cell modules -- cyclic case]{Lemma~\ref*{lemma: globalizing bases of cell modules -- cyclic case}}.
\end{proof}

\begin{lemma}  \label{lemma: path basis condition for delta-d t basis}
 The family of bases $ \{\generator {H_n} \la d_\mft   \,\mid\,  \mft \in \widehat H_n^\la  \}$ of the cell modules $\cell H n \la$ is a family of path bases.
\end{lemma}

\begin{proof}  This is a special case of ~\cite[Proposition 2.18]{MR2774622}.
\end{proof}

\begin{remark}  Existence of path bases in coherent towers of cellular algebras (without the cyclic condition)  was already shown in ~\cite[Proposition 2.18]{MR2774622}, but the construction there is not explicit.

\end{remark}

\section{Example:  The Iwahori--Hecke algebra of the symmetric groups} \label{section: Hecke algebra example}

In this section,  we apply the theory of \hyperref[section: bases in coherent towers]{Section~\ref*{section: bases in coherent towers}} to the 
Iwahori--Hecke algebra of the symmetric groups.  In particular, we recall that the sequence of Hecke algebras is a coherent tower of cyclic cellular algebras, and we compute the branching factors for reduced and  induced cell modules.  We show that the path bases obtained via ordered products of branching factors coincide with the Murphy bases.  

\subsection{Combinatorics} \label{subsection: combinatorics}
Let $n$ denote a non--negative integer and $\mathfrak{S}_n$ be the symmetric group acting on $\{1,\dots,n\}$ on the right. For $i$ an integer, $1\le i<n$, let $s_i$ denote the transposition $(i,i+1)$. Then $\mathfrak{S}_n$ is presented as a Coxeter group by generators $s_1,s_2,\dots,s_{n-1}$, with the relations 
\begin{align*}
&s_i^2=1,&&\text{for $i=1,\ldots,n-1,$}\\
&s_is_j=s_js_i,&&\text{for $|i-j| > 1$.}\\
&s_is_{i+1}s_i=s_{i+1}s_is_{i+1},&&\text{for $i=1,\ldots,n-2$.}
\end{align*}
A product $w=s_{i_1}s_{i_2}\cdots s_{i_j}$ in which $j$ is
minimal is called a \emph{reduced expression} for $w$ and $j=\el(w)$ is the \textit{length} of $w$.

If $n\ge0$, a {\em composition} of $n$ is sequence $\lambda=(\lambda_1,\lambda_2,\dots)$ of non-negative integers such that $\sum_{i\ge 1}\lambda_i=n$.  A {\em partition} of $n$ is a composition of $n$ with weakly decreasing entries.  We denote the unique partition of zero by $\emptyset$.  
The notation $\lambda\vdash n$ indicates that $\lambda$ is a partition of $n$.
If $\lambda$ is a composition, its {\em size} $|\la|$ is  $|\lambda|=\sum_{i\ge 1}\lambda_i$. 
 If $\la$ is a partition, its non-zero entries are called its {\em parts}.
 
 The {\em diagram} of a composition  of $\lambda$ is the set
\begin{align*}
[\lambda]=\left\{(i,j)\,|\,\text{$\lambda_i\ge j\ge1$ and $i\ge
1$}\,\right\}\subseteq \mathbb{N}\times\mathbb{N}.
\end{align*}
The elements of $[\lambda]$ are the \emph{nodes} of $\lambda$ and more generally a node is a pair $(i,j)\in\mathbb{N}\times\mathbb{N}$. The diagram $[\lambda]$ is traditionally represented as an array of boxes with $\lambda_i$ boxes on the $i$--th row. For example, if $\lambda=(3,2)$, then $[\lambda]=\text{\tiny\Yvcentermath1$\yng(3,2)$}$\,.   Usually, we will identify the partition $\lambda$ with its diagram and write $\lambda$ in place of $[\lambda]$.   The diagram of a partition is commonly called a {\em Young diagram}.
We denote the set of Young diagrams of size $n$ by  $\mathcal Y_n$.  

An {\em addable node} of a  Young diagram $\mu$  is a node $\alpha$  not contained in $\mu$ such that appending the node gives another Young diagram;   we write $\mu \cup \alpha$ for the Young diagram  obtained by appending $\alpha$.   For example,
$\alpha = (2, 3)$ is an addable node of $\mu = (3,2)$, and $\mu \cup \alpha = (3,3)$.   Similarly, a {\em removable node} of $\lambda$ is a node  $\alpha$ contained  in $\lambda$ such that removing the node gives another Young diagram; we write $\lambda \setminus \alpha$ for the Young diagram obtained by removing the node.   For example, $\alpha = (2,3)$ is a removable node of $\lambda = (3,3)$  and
$\lambda \setminus \alpha = (3, 2)$.    We write $\mu \to \lambda$  if $\lambda$ is obtained from $\mu$ by adding a node.

The dominance partial order $\unrhd$ on compositions of $n$ is defined as follows:   if $\lambda$ and $\mu$ are compositions of $n$, then $\lambda\unrhd\mu$ if
\begin{align*}
\textstyle\sum_{i=1}^j\lambda_i\ge\sum_{i=1}^j\mu_i \quad \text{for all $j\ge 1$.}
\end{align*}
We write $\lambda\rhd\mu$ to mean that $\lambda\unrhd\mu$ and $\lambda\ne\mu$. 

Let  $\la$ be a composition of $n$  . A $\lambda$--tableau $\mathfrak{t}$ is a map from the nodes of the diagram $[\lambda]$ to the integers $\{1,2,\dots,n\}$.
A  $\lambda$--tableau  can be represented by labelling the nodes of the diagram $[\lambda]$ with the integers $1,2,\dots,n$. For example, if $n=6$ and $\lambda=(3,2,1)$,
\begin{align}\label{tabex0.0}
\mathfrak{t}=\text{\tiny\Yvcentermath1$\young(146,23,5)$}
\end{align}
represents a $\lambda$--tableau. If $\mft$ is a $\lambda$--tableau, we say that $\lambda$ is the {\em shape} of $\mft$, and we write $\la = \shape(\mft)$ or $\la = [\mft]$.  
Write $\mathcal T(\lambda)$ for the set of all  $\lambda$--tableaux and $\mathcal T_0(\lambda)$ for the set of all injective $\lambda$--tableaux, i.e. those  in which each number from $1$ to $n$ appears exactly once.  If $\mft \in \mathcal T_0(\la)$ and $1 \le k \le n$,  we write $\node_\mft(k)$ for the node in $\la$ containing the entry $k$,   $\row_\mft(k)$ for the row coordinate of $k$ in $\mft$ and $\col_\mft(k)$ for the column coordinate of $k$ in $\mft$.  

The symmetric group $\mathfrak{S}_n$ acts freely and transitively on $\mathcal T_0(\lambda)$,  on the right,  by acting on 
the integer labels of the nodes of $[\lambda]$. For example,
\begin{align*}
\text{\tiny\Yvcentermath1$\young(123,45,6)$}\,\text{\small$(2,4)(3,6,5)$}\,
=\text{\tiny\Yvcentermath1$\young(146,23,5)$} \,.
\end{align*}

If $\la$ is a composition of $n$,  a  {\em row standard} $\lambda$--tableau  is an injective $\la$--tableau in in which the entries strictly increase from left to right along rows.  If $\la$ is a partition of $n$ a {\em standard} $\la$--tableau is a row standard $\la$--tableau in which the entries also increase strictly from top to bottom along columns. 
  Let $\mathcal{T}^{\Std}(\lambda)$ denote the set of standard $\lambda$--tableaux.
  
  If $\la$ is a composition of $n$ and $\mft$ is a row standard $\la$--tableau, then for $k \le n$, $\mft \downarrow_k$ is the tableau obtained by deleting from $\mft$ the nodes containing $k+1, \dots, n$.  Since $\mft$ is row standard, it follows that the remaining set of nodes, namely the shape of $\mft \downarrow_k$,  is the diagram of a composition of $k$.    If $\mfs$ and $\mft$ are both row standard tableaux of size $n$, we say $\mfs$ dominates $\mft$ and write $\mfs \unrhd \mft$ if 
 $[\mfs\downarrow_k] \unrhd [\mft\downarrow_k]$ for all $k \le n$.

Let $\mathfrak{t}^\lambda$ denote the standard
$\lambda$--tableau in which $1,2,\dots,n$ are
entered in increasing order from left to right along the rows of
$[\lambda]$. Thus in the previous example where $n=6$ and
$\lambda=(3,2,1)$,
\begin{align}\label{tabex1}
\mathfrak{t}^\lambda=\text{\tiny\Yvcentermath1$\young(123,45,6)$}\,.
\end{align}
For each $\mft \in \mathcal T_0(\lambda)$,  let  $w(\mft)$ denote the unique permutation such that
$\mft = \mft^\lambda w(\mft)$.

 The \emph{Young subgroup} $\mathfrak{S}_\lambda$ is defined to be the row stabiliser of $\mathfrak{t}^\lambda$ in $\mathfrak{S}_{n}$.
For instance, when $n=6$ and $\lambda=(3,2,1)$, as
in~\eqref{tabex1}, then $\mathfrak{S}_\lambda=\langle
s_1,s_2,s_4\rangle$.

Let $\lambda \vdash n$ and let $\mft \in \mathcal T_0(\lambda)$.  Let $\alpha$ be an addable node of $\lambda$.  Then we write $\mft \cup \alpha$ for the tableau of shape $\lambda \cup \alpha$ which agrees with $\mft$ on the nodes of  $\lambda$ and which has the entry $n+1$ in node $\alpha$. 
If $\mft$ is a standard $\lambda$--tableau, then the node of $\mft$ containing the entry $n$ is a removable node $\beta$  of $\lambda$.  Write $\mft' = \mft \downarrow_{n-1}$  for the 
standard tableau of shape $\lambda \setminus \beta$ obtained by removing the node $\beta$.

\subsection{Iwahori--Hecke algebras of the symmetric group}\label{e-i-h}
Let $R$ be an integral domain and $q$ be a unit in $R$. Let $\mathcal{H}_n=\mathcal{H}_n(q^2)$ denote the Iwahori--Hecke algebra of the symmetric group,  which is presented by the generators $T_1,\ldots, T_{n-1}$, and the relations
\begin{align*}
&T_iT_j=T_jT_i,&&\text{if $|i-j|>1$,}\\
&T_iT_{i+1}T_{i}=T_{i+1}T_iT_{i+1},&&\text{for $i=1,\ldots,n-2$,}\\
&(T_i-q)(T_i+q^{-1})=0,&&\text{for $i=1,\ldots,n-1$.}
\end{align*}
If we need to refer explicitly to the ground ring $R$,  we write $\mathcal H_n(R; q^2)$.  
If $v\in\mathfrak{S}_n$, and $v=s_{i_1}s_{i_2}\cdots s_{i_l}$ is a reduced expression for $v$ in $\mathfrak{S}_n$, then $T_v=T_{i_1}T_{i_2}\cdots T_{i_l}$ is well defined in $\mathcal{H}_n(q^2)$ and $\{T_v\mid v\in \mathfrak{S}_n\}$ freely generates $\mathcal{H}_n(q^2)$ as an $R$--module. 
It follows from this that $\mathcal H_n$ imbeds in $\mathcal H_{n+1}$ for all $n \ge 0$. 
The $R$--module map $*:T_v\mapsto T_{v^{-1}}$ is an algebra anti--automorphism of $\mathcal{H}_n(q^2)$. If $i,j=1,\ldots,n,$ let
\begin{align*}
T_{i,j}=
\begin{cases}
T_iT_{i+1}\cdots T_{j-1},&\text{if $j\ge i$,}\\
T_{i-1}T_{i-2}\cdots T_{j},&\text{if $i>j$.}
\end{cases}
\end{align*}

If $R$ is a field of characteristic zero and $q$ is not a proper root of unity, then it is known that each of the algebras $\mathcal H_n$ is split semisimple with simple modules labeled by the set $\mathcal Y_n$  of Young diagrams of size $n$;  moreover the branching diagram $\widehat{\mathcal H}$ of the tower $(\mathcal H_n)_{n \ge 0}$ is Young's lattice;  namely for Young diagrams $\la$ and $\mu$  with $|\mu| = |\la| + 1$,  we have $\la \to \mu$ if and only if $\mu$ is obtained from $\la$ by adjoining one node.

If $\mu\in\widehat{\mathcal{H}}_n$, define $\widehat{\mathcal{H}}_n^\mu$ to be the set of paths 
$(\mu\power 0 = \emptyset, \mu\power 1, \dots, \mu\power n = \mu)$ 
on Young's lattice $\widehat{\mathcal H}$ from $\emptyset$ to $\mu$. 
If $\mathfrak{t}$ is such a path, we regard $\mathfrak{t}$ as a map $\mathfrak{t}:\mu\to\{1,\ldots,n\}$ where, for  a node $a\in\mu$ and $1\le i \le n$,
\begin{align*}
\mathfrak{t}(a)=i\quad\text{if $\mu^{(i)}=\mu^{(i-1)}\cup\{a\}$.}
\end{align*}
In this way we obtain an identification of  the set of paths $\widehat{\mathcal{H}}_n^{\mu}$ with the set of standard tableaux $\mathcal{T}^{\Std}(\mu)$.

If $\mu\in\widehat{\mathcal{H}}_n$, let
\begin{align} \label{equation:  element m mu in Hecke algebra}
m_\mu=\sum_{v\in\mathfrak{S}_\mu}q^{\el(v)}T_v.
\end{align}

In the following statement, recall that for $\la \in \widehat{\mathcal H}_i$ and $\mf t \in \mathcal{T}^{\Std}(\la)$,  $w(\mf t)$ denotes the unique permutation  in $\mf S_i$ such that
$\mft^\la w(\mft) = \mft$.  

\begin{theorem}[Murphy~\cite{MR1327362}]\label{m-b}
For $i \ge 1$,
\begin{align*}
\mathscr{H}_i=
\left\{
m_\mathfrak{st}^\lambda=T_{w(\mathfrak{s})}^*m_\lambda T_{w(\mathfrak{t})} \ \big \vert\ 
\text{$\mathfrak{s},\mathfrak{t}\in\mathcal{T}^{\Std}(\lambda)$, $\lambda\in\widehat{\mathcal{H}}_i$}
\right\}
\end{align*}
is an $R$--basis for $\mathcal{H}_i$, and $(\mathcal{H}_i,*,\widehat{\mathcal{H}}_i,\unrhd,\mathscr{H}_i)$ is a cell datum for $\mathcal{H}_i$. 
\end{theorem}

\begin{remark}  The basis elements defined here actually differ by a power of $q$ from those defined by Murphy.  Murphy and other authors use  generators (call them $\tilde T_i$)  for $\mathcal H_n$  satisfying
\break $(\tilde T_i - q^2)(\tilde T_i + 1) = 0$.   These are related to our generators by $\tilde T_i = q  T_i$
Thus Murphy's basis elements would be $q^{\el(\mfs) + \el(\mft)} m^\lambda_{\mfs \mft}$.  
\end{remark}

We let  $ \{m^\lambda_\mft  \,\mid\, \mft \in   \std \lambda \}$  denote the basis of the cell module
$\cell {\mathcal H} n \lambda$ derived from the Murphy basis.  Then we have $m^\lambda_\mft =
m^\lambda_{\mft^\lambda}  T_{w(\mft)}$.   In particular, we see that the Hecke algebra is a cyclic cellular algebra, with $\cell {\mathcal H} n \lambda$ generated by $m^\lambda_{\mft^\lambda}$.  The bimodule isomorphism $\alpha_\la: \mathcal H^{\unrhd \lambda}/\mathcal H^{\rhd \lambda} \to 
(\cell {\mathcal H} n \lambda)^* \otimes \cell {\mathcal H} n \lambda$ is \break
$\alpha_\lambda: m^\lambda_{\mfs \mft} + \mathcal H^{\rhd \lambda}  \mapsto  T_{w(\mfs)}^* (m^\lambda_{\mft^\lambda})^* \otimes m^\lambda_{\mft^\lambda} T_{w(\mft)}$.   In particular  $m_\lambda$ is a lift
in $\mathcal H^{\unrhd \lambda}$ of $\alpha_\lambda\inv((m^\lambda_{\mft^\lambda})^* \otimes m^\lambda_{\mft^\lambda} )$,  so plays the role of the element $c_\lambda$  in \hyperref[subsection:  cellular algebras with cyclic cell modules]{Section~\ref*{subsection:  cellular algebras with cyclic cell modules}}.

We record this as a corollary:

\begin{corollary}  \label{corollary Hecke algebras cyclic cellular}
The Hecke algebras $\mathcal H_n$ are cyclic cellular algebras.
\end{corollary}

\subsection{Cell filtrations and branching factors}
Our next task is to recall that the sequence of Hecke algebras $(\mathcal H_n)_{n \ge 0}$ is a strongly coherent tower of cellular algebras,  and to determine the branching factors
$d_{\mu \to \lambda}\power n$   and $u_{\mu \to \lambda}\power n$ when $\mu \to \lambda$.  
First we discuss the cell filtrations of restrictions of cell modules and the branching factors
$d_{\mu \to \lambda}\power n$.  

\begin{theorem}[Jost, Murphy]  \label{theorem cell filtration of restricted cell modules}
Let $n \ge 1$ and $\lambda \in \widehat{\mathcal H}_n$.  Let $\cell {\mathcal H} n \lambda$ be the corresponding cell module of $\mathcal H_n$.  Then 
$\Res^{\mathcal H_n}_{\mathcal H_{n-1}}(\cell {\mathcal H} n \lambda)$    has an order preserving filtration by cell modules of $\mathcal H_{n-1}$.  
\end{theorem}

Jost ~\cite{MR1461487}  has shown, using the Dipper--James description of Specht modules of the Hecke algebras ~\cite{MR812444},
that the restriction of a Specht module has a filtration by Specht modules.  Together with Murphy's result that the cell modules of the Hecke algebras can be identified with the Specht modules~\cite[Theorem 5.3]{MR1327362}, this shows that the restriction of a cell module has a cell filtration.   A  direct proof of 
\hyperref[theorem cell filtration of restricted cell modules]{Theorem 
\ref*{theorem cell filtration of restricted cell modules}} 
 using Murphy's description of the cellular structure   is given in ~\cite{GKT-2014}.

We now give a more precise description of the cell filtration in \hyperref[theorem cell filtration of restricted cell modules]{Theorem~\ref*{theorem cell filtration of restricted cell modules}}.   Let $\alpha_1, \dots, \alpha_p$ be the list of removable nodes of $\lambda$, listed from bottom to top and let $\mu\power j = \lambda \setminus \alpha_j$.   Thus $i \le j$ if and only if $\mu\power i \unrhd \mu \power j$.  Let $N_0 = (0)$ and for $1 \le j \le p$, let 
$N_j$ be the $R$--submodule of $\cell {\mathcal H} n \lambda$ spanned by by the basis elements
$m^\lambda_\mft$ such that the node containing $n$ in $\mft$ is one of $\alpha_1, \dots, \alpha_j$.  
Then we have
$$
(0) = N_0 \subseteq N_1 \cdots \subseteq N_p = \Res^{\mathcal H_n}_{\mathcal H_{n-1}}(\cell {\mathcal H} n \lambda).
$$
The explicit form of the assertion of \hyperref[theorem cell filtration of restricted cell modules]{Theorem~\ref*{theorem cell filtration of restricted cell modules}} is that the $N_j$  are $\mathcal H_{n-1}$--submodules of $ \Res^{\mathcal H_n}_{\mathcal H_{n-1}}(\cell {\mathcal H} n \lambda)$ and $N_j/N_{j-1} \cong \cell {\mathcal H} {n-1} {\mu \power j}$  for
$1 \le j \le p$.   The isomorphism is determined by 
\begin{equation} m^{\mu \power j}_\mfs \mapsto   m^\lambda_{\mfs \cup \alpha_j} + N_{j-1}.
\end{equation}

We can now determine the branching factors $d_{\mu \to \lambda}\power n$:
\begin{corollary} \label{corollary:  d branching factors for the Hecke algebra}
  The branching factors  $d_{\mu \to \lambda}\power n$ can be chosen as follows:
Let $\lambda \in \widehat{\mathcal H}_n$ and $\mu \in \widehat{\mathcal H}_{n-1}$ with $\mu \to \lambda$.  Let $\alpha = \lambda \setminus \mu$.   Then 
\begin{equation} 
d_{\mu \to \lambda}\power n = T_{w(\mft^\mu \cup \alpha)}.
\end{equation}
More explicitly,  let $a(\alpha)$ be the entry of $\mft^\la$ in the node $\alpha$.  Then
${w(\mft^\mu \cup \alpha)} = (n, n-1, \dots, a(\alpha))$, so 
\begin{equation} 
d_{\mu \to \lambda}\power n = T_{(n, n-1, \dots, a(\alpha))} = T_{a(\alpha), n}.
\end{equation}
\end{corollary}

\begin{proof}  Under the isomorphism $\cell {\mathcal H} {n-1} {\mu \power j} \to N_j/N_{j-1}$, the generator $m^{\mu \power j}_{\mft^{\mu \power j}}$ is sent to 
$$m^\lambda_{\mft^{\mu \power j} \cup \alpha_j}  + N_{j-1} = m^\la_{\mft^\lambda}  T_{w(\mft^{\mu \power j} \cup \alpha_j)} + N_{j-1}.$$
This means that  we can chose $d_{\mu \to \lambda}\power n = T_{w(\mft^\mu \cup \alpha)}$.   Now it is straightforward to check that ${w(\mft^\mu \cup \alpha)} = (n, n-1, \dots, a(\alpha))$, so  that
$d_{\mu \to \lambda}\power n = T_{(n, n-1, \dots, a(\alpha))} = T_{a(\alpha), n}$. 
\end{proof}

Let  $\lambda\in\widehat{H}_n$ and let $\mft$ be a standard $\lambda$--tableau. We identify $\mft$ with a path on the branching diagram $\widehat{\mathcal H}$,  $\mathfrak{t}=(\emptyset = \lambda^{(0)},\ldots,\lambda^{(n)} = \lambda)$.   Define
\begin{align}
d_\mathfrak{t}=d_{\lambda^{(n-1)}\to\lambda^{(n)}}^{(n)}
d_{\lambda^{(n-2)}\to\lambda^{(n-1)}}^{(n-1)}\cdots d_{\lambda^{(0)}\to\lambda^{(1)}}^{(1)}.\label{lfdtt}
\end{align}

\begin{lemma}  \label{permutation of  tableau adjoined addable node}
Let $\lambda$ be a partition of $n$,  let $\alpha$ be a removable node of $\lambda$, and let $\mu = \lambda \setminus \alpha$.   Let $a(\alpha)$ be the entry of  $t^\lambda$ in the node $\alpha$.   
Let $\mfs \in \mathcal T_0(\mu)$ be a $\mu$--tableau.
Then
$$
w(\mfs \cup \alpha) = (n, n-1, \dots, a(\alpha))  \,w(\mfs), 
$$
and
$$
T_{w(\mfs \cup \alpha)} = T_{(n, n-1, \dots, a(\alpha))} T_{w(\mfs)} = T_{a(\alpha), n} T_{w(\mfs)}.
$$
\end{lemma}

\begin{proof}  We have
$$
\mfs \cup \alpha =  (\mft^\mu \cup \alpha) w(\mfs) = \mft^\lambda  (n, n-1, \dots, a(\alpha)) \, w(\mfs).
$$
Therefore, $$w(\mfs \cup \alpha) = (n, n-1, \dots, a(\alpha))  w(\mfs).$$
Now one can check that $(n, n-1, \dots, a(\alpha))$ is a distinguished left coset representative of 
$\mathfrak S_{n-1}$ in $\mathfrak S_n$.  Therefore, 
$$
T_{w(\mfs \cup \alpha)} =  T_{(n, n-1, \dots, a(\alpha))}  T_{w(\mfs)} = T_{a(\alpha), n} T_{w(\mfs)}.
$$
\end{proof}

\begin{lemma}  Let $\lambda$ be a partition of $n$ and let  $\mft$ be a standard $\lambda$--tableau.  Then 
$T_{w(\mathfrak{t})}=d_\mathfrak{t}$. 
\end{lemma}
\begin{proof}  Let $\alpha$ be the node of $\lambda$ containing the entry $n$ in $\mft$ and let $\mu = \lambda \setminus \alpha$.  Let $\mft'$ be the standard $\mu$ tableau obtained from $\mft$  by removing the node $\alpha$.  Let $a(\alpha)$ be the entry of $\mft^\lambda$  in the node $\alpha$.    
Then $\mft = \mft' \cup \alpha$, so by the previous lemma,
 $$
 T_{w(\mft)} = T_{a(\alpha), n}  T_{w(\mft')} = d_{\mu \to \lambda}\power n  T_{w(\mft')}.
 $$
By induction, we obtain the desired formula $T_{w(\mathfrak{t})}=d_\mathfrak{t}$. 
\end{proof}

\begin{corollary}
The bases of the cell modules and the cellular basis of the Hecke algebra
$\mathcal H_n$  given in \hyperref[proposition:  basis for cell modules of coherent cyclic tower]{Proposition~\ref*{proposition:  basis for cell modules of coherent cyclic tower}} and \hyperref[label:1]{Corollary~\ref*{label:1}}  coincide with the Murphy bases:
\begin{equation}  \label{equation:  formula for Murphy basis 1}
m^\lambda_\mft = m^\la_{\mft^\lambda} \, d_{\mft}, \qquad\text{ and}\qquad
m^\lambda_{\mfs \mft} = d_{\mfs}^* \,m_\lambda \,d_{\mft}.
\end{equation}
\end{corollary}

Next we turn to the cell filtration of induced cell modules and the branching factors $u_{\mu \to \nu}\power n$.

\begin{theorem}[Dipper--James, Murphy, Mathas]  \label{theorem:  cell filtration of induced cell modules for Hecke algebra}
Let $\mu \in \widehat{\mathcal H}_n$ and let $\cell {\mathcal H} n \mu$ be the corresponding cell module of $\mathcal H_n$.  Then $\Ind_{\mathcal H_n}^{\mathcal H_{n+1}}(\cell {\mathcal H} n \mu)$  has an order preserving filtration by cell modules of $\mathcal H_{n+1}$.
\end{theorem}

\begin{corollary}  \label{corollary:  Hecke algebras strongly coherent}
The sequence of Hecke algebras $(\mathcal H_n)_{n \ge 0}$ is  a strongly coherent tower of cyclic cellular algebras.
\end{corollary}

\begin{proof}  Combine \hyperref[theorem cell filtration of restricted cell modules]{Theorem~\ref*{theorem cell filtration of restricted cell modules}},  \hyperref[theorem:  cell filtration of induced cell modules for Hecke algebra]{Theorem~\ref*{theorem:  cell filtration of induced cell modules for Hecke algebra}} and \hyperref[corollary Hecke algebras cyclic cellular]{Corollary~\ref*{corollary Hecke algebras cyclic cellular}}.
\end{proof}

Let $\alpha = \alpha_1, \alpha_2, \dots, \alpha_p = \omega$ be the list of addable nodes of $\mu$, listed from top to bottom.  Let $\nu\power i = \mu \cup \alpha_i$.    Note that $i \le  j$ if and only if $\nu\power i \unrhd \nu\power j$.    The cell modules of $\mathcal H_{n+1}$  occurring as subquotients in the cell filtration of 
 $\Ind_{\mathcal H_n}^{\mathcal H_{n+1}}(\cell {\mathcal H} n \mu)$ are $\cell {\mathcal H} {n+1} {\nu\power i}$  for $1 \le i \le p$.

One proof of  \hyperref[theorem:  cell filtration of induced cell modules for Hecke algebra]{Theorem~\ref*{theorem:  cell filtration of induced cell modules for Hecke algebra}} is obtained by combining   ~\cite[Sect.~7]{MR812444}  with ~\cite[Theorem 5.3]{MR1327362}.  A different proof was recently given by Mathas ~\cite{MR2531227};  this proof is based on Murphy's \hyperref[theorem: Murphy's theorem cell filtration of permutation module]{Theorem~\ref*{theorem: Murphy's theorem cell filtration of permutation module}} on the existence of a cell filtration of permutation modules of $\mathcal H_n$.   We are going to sketch Mathas' proof in order to point out how the branching factors
$u_{\mu \to \nu}^{(n+1)}$ can be extracted from it.

\begin{definition}
Let $\lambda,\mu\vdash n$ and $\mathsf{T}:\lambda\to\mathbb{N}$ be a $\lambda$--tableau. Then:
\begin{enumerate}
\item $\mathsf{T}$ is a tableau of \emph{type} $\mu$ if for all $i \ge 1$, $\mu_i=\sharp\{a\in\lambda\mid \mathsf{T}(a) = i \}$.
\item $\mathsf{T}$ is \emph{semistandard} if the entries of $\mathsf{T}$ are weakly increasing along each row from left to right and strictly increasing  along each column from top to bottom. 
\end{enumerate}
Let $\mathcal{T}^{\SStd}_\mu(\lambda)$ be the set of semistandard $\lambda$-tableaux of type $\mu$ and $\mathcal{T}^{\SStd}_\mu(\widehat{\mathcal{H}}_n) =\bigcup_{\lambda\in\widehat{\mathcal{H}}_n}\mathcal{T}^{\SStd}_\mu(\lambda)$ be the set of all semistandard tableaux of type $\mu$. 
\end{definition}

Let $\lambda,\mu\vdash n$ and $\mathfrak{t}\in\mathcal{T}^{\Std}(\lambda)$. Define $\mu(\mathfrak{t})$ to be the tableau obtained from $\mathfrak{t}$ by replacing each entry $j$ in $\mathfrak{t}$ with $i$ if $j$ appears in the $i^{\textup{th}}$ row of $\mathfrak{t}^\mu$.   If
$\semistd \mu \lambda \ne \emptyset$,  then $\lambda \unrhd \mu$.  Note that there is a unique element $\mathsf T^\mu \in \semistd \mu \mu$, namely $\mathsf T^\mu = \mu(\mft^\mu)$.

If $\mathsf{S}\in \mathcal{T}^{\SStd}_\mu(\lambda)$ and $\mathfrak{t}\in\mathcal{T}^{\Std}(\lambda)$, let 
\begin{align} \label{equation: definition of c S t}
m_{\mathsf{S}\mathfrak{t}}=\sum_{\substack{\mathfrak{s}\in\mathcal{T}^{\Std}(\lambda)\\ \mu(\mathfrak{s})=\mathsf{S}}} q^{\el(w(\mathfrak{s}))} m^\lambda_{\mfs \mft} 
\end{align}

Let $\mu\in\widehat{\mathcal{H}}_n$. Define the \textit{permutation module} 
\begin{align*}
M^\mu=m_\mu \mathcal{H}_n.
\end{align*}
\begin{theorem}[\mbox{See~\cite[Theorem 7.2]{MR1327362}}] \label{theorem: Murphy's theorem cell filtration of permutation module}
If $\mu\in\widehat{\mathcal{H}}_n$, then:
\begin{enumerate}[label=(\arabic{*}), ref=\arabic{*},leftmargin=0pt,itemindent=1.5em]
\item $M^\mu$ is a free as an $R$-module, with basis
\begin{align*}
\left\{m_{\mathsf{S}\mathfrak{t}}\ \big \vert \ \mathsf{S}\in \mathcal{T}^{\SStd}_\mu(\lambda), \mathfrak{t}\in  \mathcal{T}^{\Std}(\lambda)\text{ for }\lambda\in\widehat{\mathcal{H}}_n
\right\}.
\end{align*}
\item Suppose that $\mathcal{T}^{\SStd}_\mu(\widehat{\mathcal{H}}_n)= \{\mathsf{S}_1,\ldots,\mathsf{S}_k\}$ ordered so that $i\le j$ whenever $\lambda^{(i)}\unrhd\lambda^{(j)}$, where $\lambda^{(i)}=\Shape(\mathsf{S}_i)$. Let $M_i$ be the $R$-submodule of $M^\mu$ spanned by the elements $\{m_{\mathsf{S}_j\mathfrak{t}}\mid j\le i\text{ and }\mathfrak{t}\in\mathcal{T}^{\Std}(\lambda^{(j)})\}$. Then 
\begin{align}\label{zlkfo}
\{0\}=M_0\stackrel{\lambda^{(1)}}{\subseteq}M_{1}\stackrel{\lambda^{(2)}}{\subseteq}\cdots \stackrel{\lambda^{(m)}}{\subseteq} M_m=M^\mu
\end{align}
is a cell module filtration of $M^\mu$.
The isomorphism $M_j/M_{j-1} \cong \Delta_{\mathcal{H}_n}^{\lambda^{(j)}}$ is determined by
\begin{align}\label{zlkfo2}
m_{\mathsf{S}_j\mathfrak{t}}+M_{j-1}\mapsto  m^{\lambda \power j}_\mft, \quad \text{for $\mathfrak{t}\in\mathcal{T}^\Std(\lambda^{(j)})$}.
\end{align}
\end{enumerate}
\end{theorem}

\begin{remark}
In \hyperref[theorem: Murphy's theorem cell filtration of permutation module]{Theorem~\ref*{theorem: Murphy's theorem cell filtration of permutation module}},  we have
$\mathsf S_m = \mathsf T^\mu$ and $\la\power m = \mu$.
\end{remark}

Since $\mathcal H_{n+1}$ is free of rank $n+1$ as a left $\mathcal H_n$--module, it follows that the induction functor \break $\Ind_{\mathcal H_n}^{\mathcal H_{n+1}}(\,\underline{\phantom{xx}}\,) =  \,\underline{\phantom{xx}}\,\otimes_{\mathcal H_n} \mathcal H_{n+1}$ is exact.  We will write $\Ind$ for this functor in the following discussion.   Because of exactness, we have
\begin{equation}
\Ind(M_j)/\Ind(M_{j-1}) \cong \Ind(M_j/M_{j-1}) \cong \Ind(\cell {\mathcal H} n {\lambda\power j}).
\end{equation}
In particular
\begin{equation}
\Ind(M^\mu)/\Ind(M_{m-1}) \cong \Ind(\cell {\mathcal H} n \mu).
\end{equation}
Mathas' proof of \hyperref[theorem:  cell filtration of induced cell modules for Hecke algebra]{Theorem~\ref*{theorem:  cell filtration of induced cell modules for Hecke algebra}} proceeds by exhibiting a cell filtration of $\Ind(M^\mu)/\Ind(M_{m-1}) $.

Another consequence of the freeness of  $\mathcal H_{n+1}$ as a left $\mathcal H_n$--module is the following:  if $M$ is a right ideal in $\mathcal H_n$, then
\begin{equation}
\Ind(M) = M \otimes_{\mathcal H_n} \mathcal H_{n+1} \cong M \mathcal H_{n+1},
\end{equation}
via $x \otimes h \mapsto xh$.  We will simply identify $\Ind(M)$ with $M \mathcal H_{n+1}$.  
  Recall that $\omega$ denotes the lowest addable node of $\mu$, and
note that $m_\mu = m_{\mu \cup \omega}$.  Hence,
\begin{equation}
\Ind(M^\mu) = M^\mu \mathcal H_{n+1} = m_\mu \mathcal H_{n+1} = m_{\mu \cup \omega} \mathcal H_{n+1} = M^{\mu \cup \omega}.
\end{equation}

To proceed, we need to relate semistandard tableaux of size $n$ and type $\mu$ and semistandard tableaux of size $n+1$ and type $\mu \cup \omega$.   Let $\el$ denote the number of non--zero parts of $\mu$,  so that $\omega = (\el + 1, 1)$.    If $\mathsf S$ is a semistandard tableau of shape $\lambda$ and type $\mu$,  and $\beta$ is an addable node of $\la$, then we define the semistandard tableau
$S \cup \beta$ of shape $\lambda \cup \beta$ and type $\mu \cup \omega$  by 
$S\cup \beta (x) = S(x)$ if $ x \in [\lambda]$  and $S(\beta) = \el + 1$.    We write $\semistd {\mu \cup \omega} S$ for the set of semistandard tableaux $S \cup \beta$ as $\beta$ ranges over addable nodes of $\lambda$. 
It is easy to see that
every $\mathsf U \in \semistd {\mu \cup \omega}{\widehat{\mathcal H}_{n+1}}$ is obtained as $\mathsf S \cup \beta$ for some $\mathsf S$ and some $\beta$.   

Recall that $S_1, \dots, S_m = T^\mu$ is the list of all semistandard tableaux of size $n$ and type $\mu$, listed so that $\Shape(S_i) \unrhd \Shape(S_j)$ implies $i \le j$.
Mathas defines the following $R$--submodules of $M^{\mu \cup \omega}$:
\begin{equation}
N_i = \spn \{m_{\mathsf U \mfv} \ \big \vert \   \mathsf U \in \semistd {\mu \cup \omega} {S_j},  \mfv \in \std {\Shape(\mathsf U)} \text{ for } 1 \le j \le i  \}
\end{equation}

\begin{lemma}[{\cite[Lemma~3.5]{MR2531227}}]  \label{mathas lemma 3.5}
Let $\mathsf S \in \semistd \mu \lambda$, 
$\mathsf U \in \semistd {\mu \cup \omega} {\mathsf S}$, and $\nu = \Shape(\mathsf U)$.  Then
$m_{\mathsf U \mft^\nu} \in  m_{\mathsf S \mft^\lambda} \mathcal H_{n+1}$.  
\end{lemma}

\begin{proposition}[{cf.~\cite[Theorem 3.6]{MR2531227}}]  $N_{m-1} = \Ind(M_{m-1})$. 
\end{proposition}

\begin{proof}  We have $N_m = \Ind(M_m) = M^{\mu \cup \omega}$.  
By \hyperref[mathas lemma 3.5]{Lemma~\ref*{mathas lemma 3.5}}, we have $N_{m-1} \subseteq \Ind(M_{m-1})$.  Note that
\break $M^{\mu \cup \omega}/\Ind(M_{m-1}) \cong \Ind(\cell {\mathcal H} {n} \mu)$ is free as an $R$--module of rank $(n+1) f_\mu$,  where $f_\mu$ denotes the number of standard tableaux of shape $\mu$.
On the other hand,  by applying \hyperref[theorem: Murphy's theorem cell filtration of permutation module]{Theorem~\ref*{theorem: Murphy's theorem cell filtration of permutation module}}  to $M^{\mu \cup \omega}$,  we see that $M^{\mu \cup \omega}/N_{m-1}$ is free with basis
\begin{equation} \label{eqn: basis of induced perm module mod N}
\begin{aligned}
& \left\{m_{\mathsf U \mfv} + N_{m-1} \ \big \vert \   \mathsf U \in \semistd {\mu \cup \omega} {\mathsf T^\mu},  \mfv \in \std {\Shape(\mathsf U)}  \right\}  \\
&=  \left\{m_{\mathsf T^\mu \cup \beta,  \mfv} + N_{m-1} \ \big \vert \   \beta \text{ is an addable node of $\mu$}, \mfv \in \std{\mu \cup \beta} \right\}.
 \end{aligned}
 \end{equation}
The cardinality of this basis is $\sum f_{\mu \cup \beta}$,  where the sum is over addable nodes $\beta$ of $\mu$.  Using the representation theory of the symmetric groups over the complex numbers, we have that
$$
\sum f_{\mu \cup \beta} = \dim \Ind(V^\mu) = (n+1) f_\mu,
$$
where $V^\mu$ is the simple $\mathbb C \mathfrak S_n$ module labelled by $\mu$.   Thus
$M^{\mu \cup \omega}/\Ind(M_{m-1})$ and $M^{\mu \cup \omega}/N_{m-1}$ are both free $R$ modules of the same rank.  Since $N_{m-1} \subseteq \Ind(M_{m-1}) \subseteq M^{\mu \cup \omega}$, \hyperref[linear algebra lemma]{Lemma~\ref*{linear algebra lemma}} below shows that $N_{m-1} = \Ind(M_{m-1})$.
\end{proof}

\begin{lemma} \label{linear algebra lemma}  Let $R$ be an integral domain and let
$A \subseteq B \subseteq C$ be $R$--modules with $C/A$ and $C/B$ both free of the same rank.
Then $A = B$.
\end{lemma}

\ignore{
\begin{proof}  Consider 
$$
0 \to B/A \to C/A \stackrel{\pi}{\to}  C/B \to 0.
$$
If $\{x_1, \dots, x_s\}$ is a basis of $C/A$  then $\{\pi(x_i)\}$ spans $C/B$.  Since $C/B$ has a basis of the same cardinality and $R$ is an integral domain, it follows that $\{\pi(x_i)\}$ is a basis of $C/B$ and
$\pi$ is an isomorphism.  Hence $B/A = 0$.  
\end{proof}
}

We can now exhibit an order preserving cell filtration of $M^{\mu \cup \omega}/N_{m-1} \cong \Ind(\cell {\mathcal H} n \mu)$.   In the following, we write $N = N_{m-1}$.   Recall that
$\alpha = \alpha_1, \alpha_2, \dots, \alpha_p = \omega$ is the list of addable nodes of $\mu$ 
listed from top to bottom and $\nu\power j = \mu \cup \alpha_j$. 
Let $J^0 = (0)$ and for $1 \le i \le p$, define  $J^i \subseteq M^{\mu \cup \omega}/N$ by
$$
J^i  = \spn\left\{m_{{\mathsf T}^\mu \cup \alpha_j,  \mfv}  + N  \ \big \vert \   j \le i \text{ and }  \mfv \in \std {\mu \cup \alpha_j } \right\}.
$$
\begin{theorem}[{\cite[Corollary 3.7]{MR2531227}}] \label{theorem: explicit form of cell filtration for induced cell module of Hecke algebra}
Each $J^i$ is an $\mathcal H_{n+1}$ submodule of $M^{\mu \cup \omega}/N$, 
\begin{equation} \label{eqn: explicit cell filtration of induced perm module mod N}
(0) = J^0 \subseteq J^1 \subseteq \cdots \subseteq J^p = M^{\mu \cup \omega}/N,
\end{equation}
and $J^i/J^{i-1} \cong \cell {\mathcal H} {n+1} {\nu \power i}$.  
\end{theorem}

This completes the sketch of Mathas' proof of \hyperref[theorem:  cell filtration of induced cell modules for Hecke algebra]{Theorem~\ref*{theorem:  cell filtration of induced cell modules for Hecke algebra}}.  It remains to see how the cell filtration
\eqref{eqn: explicit cell filtration of induced perm module mod N} carries over to $\Ind(\cell {\mathcal H} n \mu)$, and to identify the branching factors $u_{\mu \to \nu}^{(n+1)}$. 
The isomorphism  $\varphi: M^{\mu \cup \omega}/N  \to  \Ind(\cell {\mathcal H} n \mu)$ is the composite of the isomorphism $M^{\mu \cup \omega}/N  \cong  \Ind(M^\mu/M_{m-1})$, given by
$$
m_\mu h + N  \mapsto (m_\mu + M_{m-1}) \otimes h,
$$
and the isomorphism  $\Ind(M^\mu/M_{m-1}) \cong \Ind(\cell {\mathcal H} n \mu)$ given by
$$
(m_{\mathsf T^\mu \mft^\mu} + M_{m-1}) \otimes h \mapsto  m^\mu_{\mft^\mu} \otimes h.
$$
Since $m_{\mathsf T^\mu \mft^\mu} = m_\mu$, the composite isomorphism is given by
\begin{equation} \label{equation: description of isomorphism induced permutation module  mod N to induced cell module}
\varphi: m_\mu h + N  \mapsto m^\mu_{\mft^\mu} \otimes h.
\end{equation}
We need to examine how this isomorphism acts on the basis~\eqref{eqn: basis of induced perm module mod N} of $M^{\mu \cup \omega}/N$.  

Let $\beta$ be an addable node of $\mu$ and let $\nu = \mu \cup \beta$.   Suppose that $\beta$ is in row $r$, and let $a = \sum_{j = 1}^r \nu_j = 1+ \sum_{j = 1}^r \mu_j$.  Recall that $T_{i, i} = 1$ and  if $i > j$,  then
 $T_{i, j} = T_{(j, j+1, \dots, i)} = T_{i-1} T_{i-2} \cdots T_j$.  Define 
  \begin{equation}\begin{aligned}
D(\beta) &=  \sum_{k = 0}^{\mu_r}  q^k T_{a, a-k} \\
 &= 1 + q \,T_{a-1} + q^2 \,T_{a-1} T_{a-2} + \cdots + q^{\mu_r} \,T_{a-1} T_{a-2} \cdots T_{a - \mu_r}.
\end{aligned}
 \end{equation}
In particular, $D(\omega) = 1$.  

\ignore{
Let $\beta$ be an addable node of $\mu$ and let $\nu = \mu \cup \beta$.   Suppose that $\beta$ is in row $r$, and let 
 $b = \sum_{j = 1}^{r-1} \mu_j + 1$ and $a = \sum_{j = 1}^r \mu_j$.    (If $\beta \ne \omega$, then $b$ and $a$  are the first and last entries in row $r$ of $t^\mu$;   if $\beta = \omega$, then $b = n+1$ and $a = n$.)    Recall that $T_{i, i} = 1$ and  if $i > j$,  then
 $T_{i, j} = T_{(j, j+1, \dots, i)} = T_{i-1} T_{i-2} \cdots T_j$.  Define 
 \begin{equation}\begin{aligned}
 D(\beta) &= \sum_{j = 0}^{a+1 - b}  q^{j} \,T_{a+1, a+1 - j} \\
 &= 1 + q \,T_a + q^2 \,T_a T_{a-1} + \cdots + q^{a+1 -b} \,T_a T_{a-1} \cdots T_b.
 \end{aligned}
 \end{equation}
In particular, $D(\omega) = 1$.  
}
 
  The following lemma is a special case of ~\cite{MR2531227},  Lemmas 3.4 and 3.5.

 \begin{lemma} \mbox{}  \label{lemma:   explicit isomorphism induced perm module mod N with induced cell module}
 \begin{enumerate}
 \item  $m_\nu =  T_{n+1, a}\inv  m_\mu  T_{n+1, a}  D(\beta)$.  
 \item  $w(\mft^\mu \cup \beta)  =   (n+1, n, \dots, a)$.  Thus $m^\nu_{\mft^\nu, \mft^\mu \cup \beta}  = m_\nu (T_{n+1, a})^*$.  
 \item  $m_{\mathsf T^\mu \cup \beta , \mft^\nu} = q^{n +1 -a}\, m_\mu \, T_{n+1, a}  D(\beta)$. 
 \item  The isomorphism  $\varphi: M^{\mu \cup \omega}/N \to \Ind(\cell {\mathcal H} n \mu)$ 
  satisfies
  $$\varphi( m_{\mathsf T^\mu \cup \beta , \mft^\nu}   + N) =  m^\mu_{\mft^\mu} \otimes  q^{n+1 -a}\, T_{n+1, a} \, D(\beta).$$ 
\end{enumerate}
 \end{lemma}

 \begin{proof} If $\beta = \omega$, then  $T_{n+1, a} = D(\beta) = 1$, and all the statements are evident.  Suppose that $\beta \ne \omega$.
 Let $\nu'$ be the composition $\nu' = (\mu_1, \dots, \mu_r, 1, \mu_{r+1}, \dots, \mu_\el)$.   One has $T_{n+1, a}\inv  T_j  T_{n+1, a}  = T_{j+1}$  if
 $a \le j \le n-1$.  This follows from the identity in the braid group:
 $$
( \sigma_{a}\inv \cdots \sigma_n\inv) \sigma_j (\sigma_n \cdots \sigma_{a}) = \sigma_{j+1},
 $$
 for $a \le j \le n-1$,  where the elements $\sigma_i$ are the Artin generators of the braid group.
 From this, we obtain:
 $$
 m_{\nu'} =  T_{n+1, a}\inv  m_\mu  T_{n+1, a}.
 $$
 Note that $\mathfrak S_{\nu'} \subset \mathfrak S_\nu$  and $D(\beta) = \sum  q^{\el(x)} T_x$, as where the sum is over the distinguished right coset representatives of 
 $\mathfrak S_{\nu'}$ in  $\mathfrak S_\nu$.  Hence $m_\nu = m_{\nu'} D(\beta)$, and part (1) follows.
 The first assertion in part (2) is evident and the second statement follows because
 $T_{(n+1, \dots, a)} = T_{a, n+1} = (T_{n+1, a})^*$.  
 
 For part (3),  $m_{\mathsf T^\mu \cup \beta , \mft^\nu} = \sum_{\mfs} q^{\el(s)} (T_{d(\mfs)} )^*m_\nu$, 
 where the sum is over standard tableaux $\mfs$ of shape $\nu$ such that $(\mu \cup \omega)(\mfs) = \mathsf T^\mu \cup \beta$,
 according to the definition~\eqref{equation: definition of c S t}.   But there is only one such standard tableau, namely $\mfs = \mft^\mu \cup \beta$.  
 Applying parts (1) and  (2),   
 $$
 \begin{aligned}
 m_{\mathsf T^\mu \cup \beta , \mft^\nu}  &= q^{n+1-a} \,T_{n+1, a} \,m_\nu \\
 &= q^{n+1-a} \, m_\mu \, T_{n+1, a} \, D(\beta).
 \end{aligned}
 $$ 
 Part (4) follows from part (3) together with the description of $\varphi$ in Equation~\eqref{equation: description of isomorphism induced permutation module  mod N to induced cell module}.
 \end{proof}
 
\begin{corollary}  \label{corollary: u branching factors for the Hecke algebra}
The branching factors $u_{\mu \to \nu}^{(n+1)}$ can be chosen as
follows:   Let $\mu \in \widehat{\mathcal H}_n$  and $\nu \in  \widehat{\mathcal H}_{n+1}$ with
$\mu \to \nu$.  Let $\beta = \nu \setminus \mu$.   Suppose that $\beta$ is in row $r$ and let
$a = \sum_{j= 1}^r \nu_j$.    
Then:
\begin{equation}
\begin{aligned}
u_{\mu \to \nu}^{(n+1)} &=   T_{n+1, a} D(\beta)  =     T_{n+1, a}  \sum_{k = 0}^{\mu_r}  q^k T_{a, a-k} 
= \sum_{k = 0}^{\mu_r}  q^k T_{n+1, a-k}  \\
\end{aligned}
\end{equation}
\end{corollary}

\begin{proof}  In \hyperref[theorem: explicit form of cell filtration for induced cell module of Hecke algebra]{Theorem~\ref*{theorem: explicit form of cell filtration for induced cell module of Hecke algebra}}, we have for $j \ge 1$, 
$$
J^j =  (m_{\mathsf{T}^\mu \cup \alpha_j, \mft^{\nu\power j} } +N) \mathcal H_{n+1}  +  J^{j-1}.
$$
Set $I^j = \varphi(J^j)$.   Then  $I^j/I^{j-1} \cong  \cell {\mathcal H} {n+1} {\nu \power j}$ and
$$
I^j =  \varphi(m_{\mathsf{T}^\mu \cup \alpha_j, \mft^{\nu\power j} } + N) \mathcal H_{n+1}  +  I^{j-1}.
$$
Hence, the statement follows from \hyperref[lemma:   explicit isomorphism induced perm module mod N with induced cell module]{Lemma~\ref*{lemma:   explicit isomorphism induced perm module mod N with induced cell module}}, part (4).
\end{proof}

 \section{Algebras with  Jones basic construction} \label{framework axioms}
 \label{section:  algebras with basic construction}
 \subsection{Cellularity and the Jones basic construction: a correction}
 \label{subsection:  cellularity and jones setting and correction}
 In ~\cite{MR2794027, MR2774622},  Goodman and Graber developed a theory of cellularity for algebras with a  Jones basic construction.  Examples of such algebras include the Birman--Murakami--Wenzl, Brauer, partition, and Jones--Temperley--Lieb algebras, among others.   There was, however, a mistake in the proof in ~\cite{MR2794027}  that these algebras constitute coherent towers of cellular algebras.  In this section, we will review the setting of ~\cite{MR2794027, MR2774622},  describe the error, and explain what needs to be done to correct it.
 
 The setting in ~\cite{MR2794027},  as modified in ~\cite{MR2774622} is the following.  First recall that 
 an {\em essential idempotent} in an algebra $A$ over a ring $R$ is an element $e$ such that 
$e^2 = \delta e$ for some non--zero $\delta \in R$.   Let $R$ be an integral domain with field of fractions $F$ and consider two towers of algebras with common multiplicative identity,
\begin{align}\label{p-r}
A_0\subseteq A_1\subseteq A_2\subseteq\cdots \qquad \text{and} \qquad
H_0\subseteq H_1\subseteq H_2\subseteq\cdots.
\end{align}
It is assumed that the two towers satisfy the following list of axioms:

\medskip
\begin{enumerate}
\item \label{axiom: involution on An}  There is an algebra involution $*$  on $\cup_n A_n$ such that $(A_n)^* = A_n$, and likewise, there is 
an algebra involution $*$  on $\cup_n H_n$ such that $(H_n)^* = H_n$.
\item  \label{axiom: A0 and A1}  
$A_0 = H_0 = R$   and $A_1 = H_1$  (as algebras with involution).
\item \label{axiom:  idempotent and Hn as quotient of An}
 For $n \ge 2$,  $A_n$ contains an  essential idempotent $e_{n-1}$ such that $e_{n-1}^* = e_{n-1}$ and
\break $A_n/(A_n e_{n-1} A_n)  \cong H_n$ as algebras with involution.

\item \label{axiom: en An en} For $n \ge 1$,   $e_{n}$ commutes with $A_{n-1}$ and $e_{n} A_{n} e_{n} \subseteq  A_{n-1} e_{n}$.
\item  \label{axiom:  An en}
For $n \ge 1$,  $A_{n+1} 	e_{n} = A_{n} e_{n}$,  and the map $x \mapsto x e_{n}$ is injective from
$A_{n}$ to $A_{n} e_{n}$.
\item \label{axiom: e(n-1) in An en An} For $n \ge 2$,   $e_{n-1} \in A_{n+1} e_n A_{n+1}$.
\item  \label{axiom: semisimplicity}
For all $n$,  $A_n^F : = A_n \otimes_R F$   is split semisimple.  
\item \label{axiom Hn coherent}  $(H_n)_{n \ge 0}$ is a strongly coherent tower of cellular algebras.
\setcounter{saveenumi}{\theenumi}
\end{enumerate}

\medskip

Under these hypotheses, it is claimed in ~\cite{MR2794027, MR2774622}  that $(A_n)_{n \ge 0}$  is a strongly coherent tower of cellular algebras.  The strategy of the proof is to show by induction that the following  statements hold for all $n \ge 0$:
\begin{itemize}
\item  $A_n$ is a cellular algebra.
\item   For $2 \le n$,   $J_n = A_n e_{n-1} A_n$ is a cellular ideal in $A_n$.  
\item   For $2 \le n$,  the cell modules of $J_n$  are of the form $\Delta =  \Delta' \otimes_{A_{n-2}} e_{n-1}  A_n$, where $\Delta'$ is a cell module of $A_{n-2}$.
\item  The finite tower $(A_k)_{0 \le k \le n}$  is strongly coherent.  
\end{itemize}
For $n \le 1$  these statements are evident.   Assuming the statements hold for some fixed $n \ge 1$, one first proves that $J_{n+1}$ is a cellular ideal in $A_{n+1}$  with cell modules of the form
$\Delta = \Delta' \otimes_{A_{n-1}}  e_n A_{n+1}$,  where $\Delta'$ is a cell module of $A_{n-1}$.  It follows from \hyperref[lemma: extensions of cellular algebras]{Lemma~\ref*{lemma: extensions of cellular algebras}} that $A_{n+1}$ is cellular.

 It then remains to show that for each cell module $\Delta$ of $A_{n+1}$,  the restriction of $\Delta$ to $A_n$  has an order preserving cell filtration, and that for each cell module $\Delta$ of $A_{n}$,   the induction of $\Delta$ to $A_{n+1}$  has an order preserving cell filtration.  In fact, we will go over the details of the proof of these last two steps below in \hyperref[theorem:  recursive determination of the branching factors]{Theorem~\ref*{theorem:  recursive determination of the branching factors}}.  For now, we note that in the proof of the statement about induced modules, it was falsely claimed in ~\cite{MR2794027},  in the last paragraph on page 335,   that if $\Delta$ is a cell module of $J_n$   then  $\Delta J_n = \Delta$.   In fact, this does not follow from the axioms~\eqref{axiom: involution on An}--\eqref{axiom Hn coherent}   listed above, so it is necessary to add an additional axiom to our framework, as follows:
 
 \medskip
 \begin{enumerate}
\setcounter{enumi}{\thesaveenumi}
\item  \label{axiom:  Delta J}
For $n \ge 2$,   $e_{n-1} A_n  e_{n-1}  A_n  =   e_{n-1}  A_n$.
\setcounter{saveenumi}{\theenumi}
\end{enumerate}
\medskip

From this, it follows that for a cell module $\Delta =  \Delta' \otimes_{A_{n-2}} e_{n-1}  A_n$  of $J_n$, we have $\Delta J_n = \Delta$, and the proof in   ~\cite{MR2794027}  can proceed as before. 

Let us now consider the applicability of the augmented framework axioms~\eqref{axiom: involution on An}--\eqref{axiom:  Delta J}  to the principal examples considered in ~\cite{MR2794027, MR2774622}.   In fact, in each example,  a stronger version of axiom
\eqref{axiom: e(n-1) in An en An} holds, namely
$$
e_{n-1} e_n e_{n-1} = e_{n-1}  \quad\text{and}\quad  e_n e_{n-1} e_n = e_n \quad\text{for}\quad  n \ge 2.
$$
Thus for $n \ge 3$, 
$$
e_{n-1} A_n e_{n-1}  A_n  \supseteq e_{n-1}  e_{n-2}  e_{n_1}  A_n  = e_{n-1} A_n.
$$
Therefore,  Axiom~\eqref{axiom:  Delta J} boils down to the statement
$$
e_1 A_2 e_1 A_2 = e_1 A_2.
$$
When $A_n$ is the $n$--th  BMW, Brauer, partition, or Jones--Temperley--Lieb algebra defined over an integral ground ring $R$,    we have  $A_1 = H_1 = R$.  Let $\delta$ be the non--zero element of $R$ such that $e_1^2 = \delta e_1$.   Then we have

$$
e_1 A_2  e_1 A_2 = e_1 A_1 e_1 A_2 =  e_1^2 A_2 = \delta e_1 A_2,
$$
where we have used $e_1 A_2 = e_1 A_1 = R e_1$.    In each of these examples,  $e_1 A_2$ is  free as an $R$--module,  and hence 
 Axiom~\eqref{axiom:  Delta J}  holds holds if and only if   $\delta$ is  invertible in $R$.     It follows that Axiom~\eqref{axiom:  Delta J} does not hold when $R$ is the generic ground ring,  but it does hold when $R$ is the generic ground ring with $\delta\inv$ adjoined. 
   
In fact,  for these algebras, it is false that $(A_n)_{n \ge 0}$  is a coherent tower of cellular algebras, over the generic ground ring,  but, by ~\cite{MR2794027}, as corrected above,  it is true over the generic ground ring with $\delta\inv$ adjoined.  This is illustrated by the example of the Jones--Temperley--Lieb algebras in the following section.  

\subsection{An example:  the Jones--Temperley--Lieb algebras}   We first state an elementary result about the commutativity  of specialization and induction.

Let $A$ be an algebra over an integral domain $R$ and let $\varphi : R \to k$ be  a ring homomorphism from $R$ to a field $k$.   Write $A^k$  for $A \otimes_R k$,  and for a right $A$--module $M$,  write $M^k$   for the right $A^k$--module $M \otimes_R k$.

\begin{lemma} \label{lemma:  commutativity of induction and specialization}
Let $A \subseteq B$ be algebras over an integral domain $R$, let $\varphi : R \to k$ be a ring homomorphism from $R$ to a field $k$, and let $M$ be  a right $A$--module. Then
$$
\Ind_A^B(M) \otimes_R k \cong \Ind_{A^k}^{B^k}(M^k),
$$   
as right $B^k$--modules.
\end{lemma}

\begin{corollary} \label{corollary:  commutativity of induction and specialization}
 If, in the situation of the lemma,  $\Ind_A^B(M)$ is free as an $R$--module,  then 
$\dim_k( \Ind_{A^k}^{B^k}(M^k))$  is independent of the choice of $k$ and of the homomorphism $\varphi : R \to k$.   
\end{corollary}

Now we consider the Jones--Temperley--Lieb algebras $A_n = A_n(R_0; \deltabold)$  defined over the generic ground ring $R_0 = \Z[\deltabold]$,  where $\deltabold$ is an indeterminant.   For the definition of these algebras and a description of their cellular structure, see \hyperref[subsection: temperley lieb algebras]{Section~\ref*{subsection: temperley lieb algebras}} of this paper, and further references there.

The algebra $A_2$ has two cell modules, each of rank 1.  They are  $\Delta_0 = e_1 A_2 = R e_1$  and
$\Delta_1 = A_2/R e_1 $.    When $k = \Q(\deltabold)$,   $\Ind_{A_2^k}^{A_3^k}(\Delta_0^k)$ is two dimensional and  $\Ind_{A_2^k}^{A_3^k}(\Delta_1^k)$ is three dimensional, as one sees by examining the generic branching diagram for the tower $(A_n^k)_{n \ge 0}$.  However, when $k = \Q$ and $\delta = 0$,   $\Delta_0^k  \cong  \Delta_1^k$, so also $\Ind_{A_2^k}^{A_3^k}(\Delta_0^k) \cong \Ind_{A_2^k}^{A_3^k}(\Delta_1^k)$.  It follows from this and \hyperref[corollary:  commutativity of induction and specialization]{Corollary~\ref*{corollary:  commutativity of induction and specialization}} that at least one of $\Ind(\Delta_0)$ or $\Ind(\Delta_1)$  fails to be free as an $R$--module,  and in particular one of these induced modules does not have a cell filtration.  

\begin{corollary}  The tower of Jones--Temperley--Lieb algebras $(A_n(R_); \deltabold))_{n \ge 0}$  over the generic ground ring $R_0 = \Z[\deltabold]$  is not a coherent tower of cellular algebras. 
\end{corollary}

\subsection{Standing assumptions} \label{subsection:  standing assumptions}
For the remainder of the paper we will work in the setting described by axioms~\eqref{axiom: involution on An}--\eqref{axiom:  Delta J}  of \hyperref[subsection:  cellularity and jones setting and correction]{Section~\ref*{subsection:  cellularity and jones setting and correction}}, and assume in addition that

\medskip
\begin{enumerate}
\setcounter{enumi}{\thesaveenumi}
\item \label{axiom: Hn cyclic cellular}  Each $H_n$ is a cyclic cellular algebra. 
\setcounter{saveenumi}{\theenumi}
\end{enumerate}

\medskip

\subsection{Cellularity of the algebras $A_n$}
\label{subsection results on cellularity from Goodman Graber}

Next we review some of the consequence of our axioms that were obtained in ~\cite{MR2794027, MR2774622}, as corrected above in \hyperref[subsection:  cellularity and jones setting and correction]{Section~\ref{subsection:  cellularity and jones setting and correction}}.     In the following let $({H}_i,*,\widehat{{H}}_i,\unrhd,\mathscr{H}_i)$ denote the cell datum for $H_n$.  

\begin{enumerate}
\item   Each $A_n$ is a cellular algebra.  In fact, this is demonstrated by showing that $J_n = A_n e_{n-1} A_n$ is a cellular ideal of $A_n$.  Since the quotient algebra $H_n = A_n/J_n$ is assumed to be cellular, it follows from \hyperref[lemma: extensions of cellular algebras]{Lemma~\ref*{lemma: extensions of cellular algebras}} that $A_n$ is cellular.   
\item The partially ordered set $\hat A_n$ in the cell datum for $A_n$ can be realized as
$$
\hat A_n = \leftbrace (\la, \el) \ \big \vert \   0 \le \el \le \lfloor n/2 \rfloor \text{ and }  \la \in \widehat H_{n - 2 \el} \rightbrace,
$$
with the partial order $(\la, \el)  \rhd  (\mu, m)$ if $\el > m$ or if $\el = m$ and $\la \rhd \mu$ in $\widehat{H}_{n-2 \el}$.
\item The cell modules $\cell A n {(\la, 0)}$ for $\la \in \widehat H_n$   are those such that $\cell A n {(\la,  0)} J_n = 0$.    Let $\pi_n : A_n \to A_n/J_n = H_n$ denote the quotient map.  
The cell module $\cell A n {(\la, 0)}$ can be identified with $\cell H n \la$ via
$x a = x \pi_n(a)$ for $x \in \cell H n \la$ and $a \in A_n$, as in \hyperref[remark on extensions of cellular algebras]{Remark~\ref*{remark on extensions of cellular algebras}}.  
The cell modules $\cell A n {(\la,  \el)}$ for $\el >0$ are the cell modules of the cellular ideal $J_n$.
For $\el > 0$, we have
$$
\cell A n {(\la, \el)} \cong \cell A {n -2} {(\la, \el -1)} \otimes_{A_{n-2}}  e_{n-1} A_{n}
= \cell A {n -2} {(\la, \el -1)} \otimes_{A_{n-2}}  e_{n-1} A_{n-1}.
$$

\item  The sequence $(A_n)_{n\ge 0}$ is a strongly coherent tower of cellular algebras.  Since
$A_n^F$ and $H_n^F$ are split semisimple for all $n$,  the two towers have branching diagrams, by \hyperref[corollary:  multiplicities in cell filtrations 3]{Corollary~\ref*{corollary:  multiplicities in cell filtrations 3}}.
\item The branching diagram $\hat A$ for the tower $(A_n)_{n\ge 0}$ is that ``obtained by reflections'' from the branching diagram $\widehat H$ of the tower $(H_n)_{n\ge 0}$.   That is,  for 
$(\la, \el) \in \hat A_n$ and $(\mu, m) \in \hat A_{n+1}$, we have $(\la, \el) \to (\mu, m)$ only if $m \in \{\el, \el +1\}$;  moreover,  $(\la, \el) \to (\mu, \el)$ if and only if $\la \to \mu$ in $\widehat H$, and $(\la, \el) \to (\mu, \el + 1)$ if and only if $ \mu \to \la $ in $\widehat H$.
\end{enumerate}

\begin{remark} The parameterization of $\hat A_n$ given here differs from that used in ~\cite{MR2794027, MR2774622}.  
\end{remark}

Taking Axiom~\eqref{axiom: Hn cyclic cellular} into account, we obtain:

\begin{theorem} 
\label{theorem: jones tower coherent tower cyclic cellular}
The tower $(A_n)_{n \ge 0}$ is a strongly coherent tower of cyclic cellular algebras.
\end{theorem}

\begin{proof}   From ~\cite{MR2794027, MR2774622},  with the correction noted in \hyperref[subsection:  cellularity and jones setting and correction]{Section~\ref*{subsection:  cellularity and jones setting and correction}}, we have that 
  the tower is a strongly coherent tower of cellular algebras.  It remains to show that each $A_n$ is cyclic cellular.
We prove this by induction on $n$.  The statement is known for $n = 0$ and $n = 1$, since
$A_0 = R$  and $A_1 = H_1$.   Fix $n \ge 0$ and assume the algebras $A_k$ for $k \le n$ are cyclic cellular.   The cell modules $\cell A {n+1} {(\la, 0)}$ are cell modules of $H_{n+1}$, so cyclic by axiom (9).
For $\el >  0$, we can take
$$
\cell A {n+1} {(\la, \el)} =  \cell A {n-1} {(\la, \el-1)} \otimes_{A_{n-1}} e_n A_{n+1},
$$
By the induction hypothesis, $\cell A {n-1} {(\la, \el-1)}$ is cyclic, say with generator 
$\generator {A_{n-1}} {(\la, \el-1)}$.  
It follows that $\cell A {n+1} {(\la, \el)}$
 is cyclic with generator
$\generator {A_{n+1}}  {(\la, \el)} = \generator {A_{n-1}}  {(\la, \el -1)} \otimes_{A_{n-1}} e_n$.
\end{proof}

\subsection{Data associated with the cell modules $\cell A n {(\la, \el)}$}
We suppose that generators $\generator {H_n} \la$ of $\cell H n \la$ have been chosen for all $n \ge 0$ and for all $\la \in \widehat H_n$.  We suppose also that $H_n$--$H_n$ bimodule isomorphisms 
$\alpha_\la : H_n^{\unrhd \la}/H_n^{\rhd \la} \to (\cell H n \la)^* \otimes_R \cell H n \la$ have been chosen, satisfying $* \circ \alpha_\la = \alpha_\la \circ *$.  Finally, we suppose that elements
$c_\la \in H_n^{\unrhd \la}$ have been chosen with $\alpha_\la(c_\la + H_n^{\rhd \la}) = (\generator {H_n} \la)^* \otimes \generator {H_n} \la$.  

Now we want to do the following:
\begin{enumerate}
\item establish models of  cell modules $\cell A n {(\la, \el)}$  of $A_n$ for all $n$ and all 
$(\la, \el)  \in \hat A_n$;  
\item select generators $\generator {A_n} {(\la, \el)}$ for each cell module;  
\item choose $A_n$--$A_n$  bimodule isomorphisms 
$$
\alpha_{(\la, \el)} : A^{\unrhd (\la, \el)}/A^{\rhd (\la, \el)}
\to (\cell A n {(\la, \el)})^* \otimes_R \cell A n {(\la, \el)}
$$  
satisfying $*\circ \alpha_{(\la, \el)} = \alpha_{(\la, \el)}\circ *$;  
\item and finally choose elements 
$c_{(\la, \el)} \in A^{\unrhd (\la, \el)}$ such that
$$\alpha_{(\la, \el)}(c_{(\la, \el)} +  A^{\rhd (\la, \el)} ) = (\generator {A_n} {(\la, \el)})^* \otimes  \generator {A_n} {(\la, \el)}.$$
\end{enumerate}

When $\el = 0$,  we identify $\cell A n {(\la, 0)}$ with $\cell H n \la$, and we proceed according to the prescription of \hyperref[remark on extensions of cellular algebras]{Remark~\ref*{remark on extensions of cellular algebras}} and \hyperref[remark: extensions of cyclic cellular algebras]{Remark~\ref*{remark: extensions of cyclic cellular algebras}}. Namely, 
$\generator {A_n} {(\la, 0)} =  \generator {H_n} \la$;   
$\alpha_{(\la, 0)} :  a +  A_n^{\rhd (\la, 0)} \mapsto \alpha_\la(\pi_n(a) + H_n^{\rhd \la})$;
and $c_{(\la, 0)}$ is any element of $\pi_n\inv(c_\la)$.  

We continue by induction on $n$.  For $n \le 1$ there is nothing to do, since $A_0 = R$ and
$A_1 = H_1$.   Fix $n \ge 2$ and suppose that all the desired data has been chosen for all 
$k \le n$ and all $(\mu, m) \in \hat A_k$.   We have to consider $(\la, \el) \in \hat A_{n+1}$ with
$\el > 0$.   As a model of the cell module $\cell A {n+1} {(\la, \el)}$ we can take
$\cell A {n-1} {(\la, \el -1)} \otimes_{A_{n-1}} e_n A_{n+1}$, and for the generator of the cell module we can take $\generator {A_{n+1}}  {(\la, \el)} = \generator {A_{n-1}}  {(\la, \el -1)} \otimes_{A_{n-1}} e_n$.

Next we  define $\alpha_{(\la, \el)}$.  According to ~\cite[Sect.~4]{MR2794027},
\begin{equation*}
\begin{aligned}
A_{n+1}^{\unrhd (\la, \el)} &= A_{n+1} A_{n-1}^{\unrhd (\la, \el-1)} e_n A_{n+1} \\
&\cong A_{n+1} e_n \otimes_{A_{n-1}} A_{n-1}^{\unrhd (\la, \el-1)} \otimes_{A_{n-1}} e_n A_{n+1},
\end{aligned}
\end{equation*}
as $A_{n+1}$--$A_{n+1}$ bimodules,   with the isomorphism determined by $a_1 x e_n a_2 \mapsto a_1 e_n \otimes x \otimes e_n a_2$.  
Similarly 
\begin{equation*}
\begin{aligned}
A_{n+1}^{\rhd (\la, \el)} &= A_{n+1} A_{n-1}^{\rhd (\la, \el-1)} e_n A_{n+1} \\
&\cong A_{n+1} e_n \otimes_{A_{n-1}} A_{n-1}^{\rhd (\la, \el-1)} \otimes_{A_{n-1}} e_n A_{n+1}.
\end{aligned}
\end{equation*}
Moreover, we have an isomorphism 
\begin{equation*}
\begin{aligned}
\varphi: A_{n+1}^{\unrhd (\la, \el)}/A_{n+1}^{\rhd (\la, \el)} 
\to  A_{n+1} e_n \otimes_{A_{n-1}} (A_{n-1}^{\unrhd (\la, \el-1)}/A_{n-1}^{\rhd (\la, \el-1)}  )\otimes_{A_{n-1}} e_n A_{n+1},
\end{aligned}
\end{equation*}
determined  by 
\begin{equation*}
\varphi(a_1 x e_n a_2 + A_{n+1}^{\rhd (\la, \el)}) = a_1 e_n \otimes (x +  A_{n- 1}^{\rhd (\la, \el-1)} )\otimes e_n a_2.
\end{equation*}
We identify $(e_n A_{n+1})^*$ with $A_{n+1} e_n$   (as $A_{n+1}$--$A_{n-1}$ bimodules).  
Thus
\begin{equation*}
(\cell A {n+1} {(\la, \el)})^* = (\cell A {n-1} {(\la, \el -1)} \otimes_{A_{n-1}} e_n A_{n+1})^*
= A_{n+1} e_n \otimes_{A_{n-1}}  (\cell A {n-1} {(\la, \el -1)})^*.
\end{equation*}
We define
\begin{equation*}
\alpha_{(\la, \el)} = (\id_{A_{n+1} e_n} \otimes \,\alpha_{(\la, \el-1)} \otimes \,\id_{e_n A_{n+1} }) \circ \varphi.
\end{equation*}
Thus
\begin{equation*}
\begin{aligned}
\alpha_{(\la, \el)} : A_{n+1}^{\unrhd (\la, \el)}/A_{n+1}^{\rhd (\la, \el)}  &\to
A_{n+1} e_n \otimes_{A_{n-1}}  (\cell A {n-1} {(\la, \el -1)})^* \otimes_R
\cell A {n-1} {(\la, \el -1)} \otimes_{A_{n-1}} e_n A_{n+1} \\
&= (\cell A {n+1} {(\la, \el)})^*  \otimes_R  \cell A {n+1} {(\la, \el)}.
\end{aligned}
\end{equation*}
Now one can check that $* \circ \alpha_{(\la, \el)} = \alpha_{(\la, \el)} \circ *$.  

Note that $c_{(\la, \el-1)} e_n \in  A_{n-1}^{\unrhd (\la, \el-1)} e_n  \subseteq 
A_{n+1}^{\unrhd (\la, \el)}$ and 
$$
\alpha_{(\la, \el)}(c_{(\la, \el-1)} e_n + A_{n+1}^{\rhd (\la, \el)}) =   (e_n \otimes (\generator {A_{n-1}} {(\la, \el-1)})^* )\otimes (\generator {A_{n-1}} {(\la, \el-1)} \otimes e_n) = (\generator {A_{n+1}}  {(\la, \el)})^* \otimes \generator {A_{n+1}}  {(\la, \el)}, 
$$
so we can take $c_{(\la, \el)} = c_{(\la, \el-1)} e_n$.  

Let us restate this last observation, replacing $n+1$ by $n$.  We have shown that if
$(\la, \el) \in \hat A_n$ and $\el > 0$,  then (we can take) 
\begin{equation}
c_{(\la, \el)} = c_{(\la, \el-1)} e_{n-1}
\end{equation}
 By induction, we have
\begin{equation}
c_{(\la, \el)} = c_{(\la, 0)} e_{n-2\el+1}e_{n-2\el+3}\cdots e_{n-1}.
\end{equation}
Expressions of this form will appear again, so we establish the notation

\begin{equation} \label{notation e sub i super ell}
e_{n-1}^{(\el)} =  
\begin{cases}
1  & \text{ if }   \el = 0 \\
\underbrace{e_{n-2\el+1}e_{n-2\el+3}\cdots e_{n-1}}_{\text{$\el$ factors}}
&\text{ if }  \el=1,\ldots,\lfloor n/2\rfloor,  \text{ and }\\
0 &\text{ if }  \el>\lfloor n/2\rfloor.
\end{cases}
\end{equation}

With this notation, we have
\begin{equation}  \label{equation:  formula for c lambda ell}
c_{(\la, \el)} = c_{(\la, 0)} e_{n-1} \power \el.
\end{equation}

\subsection{Branching factors}  We continue to work with a pair of towers of algebras~\eqref{p-r}  satisfying the standing assumptions of \hyperref[subsection:  standing assumptions]{Section~\ref*{subsection:  standing assumptions}}.

We know already that both of the towers $(H_n)_{n \ge 0}$  and $(A_n)_{n \ge 0}$  are strongly coherent towers of cyclic cellular algebras, with $H_n^F$ and $A_n^F$ split semisimple for all $n$.  Therefore, the analysis of \hyperref[subsection: bases of cell modules for coherent cyclic towers]{Section~\ref{subsection: bases of cell modules for coherent cyclic towers}}, concerning branching factors and path bases,  is applicable to both towers.  We will show that the branching factors and path bases for the tower $(A_n)_{n \ge 0}$  can be computed by explicit formulas from those for the tower 
$(H_n)_{n \ge 0}$. 

Suppose that we have chosen  an order preserving cell filtration  of 
 $\Res_{H_n}^{H_{n+1}}(\cell H {n+1} \mu)$ and of  \break
 $\Ind_{H_n}^{H_{n+1}}(\cell H n \la)$  for all $n$ and for all
$\la \in \widehat H_n$ and $\mu \in \widehat H_{n+1}$,  
\begin{align} \label{equation: ordered filtration of restricted module 00}
\{0\}=M_0\stackrel{\lambda^{(1)}}{\subseteq} M_1\stackrel{\lambda^{(2)}}{\subseteq}\cdots\stackrel{\lambda^{(r)}}{\subseteq} M_r= \Res_{H_n}^{H_{n+1}}(\cell H {n+1} \mu),
\end{align} 
and
\begin{align} \label{equation: ordered filtration of induced module}
\{0\}=N_0\stackrel{\mu^{(1)}}{\subseteq} N_1\stackrel{\mu^{(2)}}{\subseteq}\cdots\stackrel{\mu^{(p)}}{\subseteq} N_p= \Ind_{H_n}^{H_{n+1}}(\cell H n \la).
\end{align}
 Fix  a generator $\generator  {H_k} \nu$  of $\cell H k \nu$  for each $k$ and each $\nu \in \widehat H_k$.
 Moreover, suppose that we have chosen, once and for all,  branching factors    and $\dd \la \mu {n+1}$ 
and $\uu \la \mu {n+1}$   in $H_{n+1}$, so that 
$\generator {H_{n+1}} \mu  \dd {\la \power j} \mu {n+1} + M_{j-1}$ is a generator of  $M_j/M_{j-1} \cong 
\cell H n {\la \power j}$, and 
$\generator {H_{n}} \la  \otimes  \uu \la {\mu\power i} {n+1} +  N_{i-1}$ is a generator of 
$N_i/N_{i-1} \cong \cell H {n+1} {\mu\power i}$, following observations
\eqref{observation 1 regarding branching factors}   and \eqref{observation 2 regarding branching factors}, in 
\hyperref[subsection: Cyclic cellularity and branching factors]{Section~\ref*{subsection: Cyclic cellularity and branching factors}}.
For each $n$, $\la$ and $\mu$,  choose $\uubar \la \mu {n+1}  \in  \pi_{n+1}\inv(\uu \la \mu {n+1})$   and 
$\ddbar \la \mu {n+1}  \in  \pi_{n+1}\inv(\dd \la \mu {n+1})$ arbitrarily.  

We know that for each $n$ and for each $(\la, \el) \in \hat A_n$,  there exists an order preserving cell filtration of $\Ind_{A_n}^{A_{n+1}}(\cell A n  {(\la, \el)})$
\begin{align} \label{equation: ordered filtration of induced module of A n}
\{0\}=N_0\stackrel{(\mu^{(1)}, \,m_1)}{\subseteq} N_1\stackrel{(\mu^{(2)}, \,m_2)}{\subseteq}\cdots\stackrel{(\mu^{(s)}, \,m_s)}{\subseteq} N_s= \Ind_{A_n}^{A_{n+1}}(\cell A n  {(\la, \el)}),
\end{align}
and there exist branching factors $\uu {(\la, \el)} {(\mu, m)} {n+1} \in A_{n+1}$ such that
\begin{equation} \label{equation: characterization of elements uu for A n}
\generator {A_n} {(\la, \el)} \otimes_{A_n} \uu {(\la, \el)} {(\mu \power j, m_j)} {n+1}  + N_{j-1} \mapsto \generator {A_{n+1}} {(\mu \power j, m_j)}
\end{equation}
under the isomorphism $N_j/N_{j-1} \cong \cell A {n+1} {(\mu \power j, m_j)}$.

Likewise, we  know that for each $n$ and for each $(\mu, m) \in \hat A_{n+1}$,  there exists an order preserving cell filtration of $\Res_{A_n}^{A_{n+1}}(\cell A {n+1}  {(\mu, m)})$
\begin{align} \label{equation: ordered filtration of restricted module of A n plus 1}
\{0\}=M_0\stackrel{(\la^{(1)}, \,\el_1)}{\subseteq} M_1\stackrel{(\la^{(2)}, \,\el_2)}{\subseteq}\cdots\stackrel{(\la^{(t)}, \,\el_t)}{\subseteq} M_t= \Res_{A_n}^{A_{n+1}}(\cell A {n+1}  {(\mu, m)}),
\end{align}
and there exist branching factors $\dd {(\la, \el)} {(\mu, m)} {n+1} \in A_{n+1}$ such that
\begin{equation} \label{equation: characterization of elements dd for A n}
\generator {A_{n+1}} {(\mu, m)} \dd {(\la \power j, \el_j)} {(\mu, m)} {n+1} + M_{j-1} \mapsto
\generator {A_n} {(\la \power j, \el_j)} 
\end{equation}
under the isomorphism $M_j/M_{j-1} \cong \cell A {n} {(\la \power j, \el_j)}$.

Neither the cell filtrations~\eqref{equation: ordered filtration of induced module of A n}   and~\eqref{equation: ordered filtration of restricted module of A n plus 1} nor the branching factors in~\eqref{equation: characterization of elements uu for A n} and
\eqref{equation: characterization of elements dd for A n} are canonical.  However, it was shown in 
~\cite{MR2794027} that  cell filtrations of the induced and restricted modules for the tower $(A_n)$ can be obtained recursively, based on the cell filtrations of induced and restricted modules for the tower $(H_n)$.  We will show here that the  branching factors for the tower $(A_n)$  can also be chosen to satisfy recursive relations, so that they are determined completely by the liftings
$\uubar \la \mu {n+1}$ and $\ddbar \la \mu {n+1}$ of the branching factors for the tower $(H_n)$. 

Each of the statements in the following theorem should be interpreted as applying whenever they make sense.  For example, in statement (2),  the branching factor $\dd {(\la, \el)} {(\mu, m+1)} {n+1}$ makes sense when $n \ge 1$,  $(\la, \el) \in \hat A_n$,  $(\mu, m+1) \in \hat A_{n+1}$, and
$(\la, \el) \to (\mu, m+1)$ in the branching diagram $\hat A$.  This implies that $(\mu, m) \in \hat A_{n-1}$ and $(\mu, m) \to (\la, \el)$ in $\hat A$, so that the branching factor   $\uu {(\mu, m)} {(\la, \el)} {n}$ also makes sense.  

\begin{theorem} \label{theorem:  recursive determination of the branching factors}
The branching factors for the tower $(A_n)_{n\ge 0}$ can be chosen to satisfy:
\begin{enumerate}
\item  $\dd {(\la, 0)}{(\mu,0)} {n+1} = \ddbar {\la} {\mu} {n+1}$. 
\item   $\dd {(\la, \el)} {(\mu, m+1)} {n+1} =   \uu {(\mu, m)} {(\la, \el)} {n}$.
\item   $\uu {(\la, 0)}{(\mu,0)} {n+1} = \uubar {\la} {\mu} {n+1}$. 
\item     $\uu {(\la, \el)} {(\mu, m+1)} {n+1} =   \dd {(\mu, m)} {(\la, \el)} {n} \, e_n$.
\end{enumerate}
\end{theorem}

\begin{proof}
To prove this result, we have to look into, and add some detail to, the proof in ~\cite{MR2794027, MR2774622}  that the tower $(A_n)$ is strongly coherent.

First we consider branching factors for reduced modules.   The argument is an elaboration of the proof of~\cite[Proposition~4.10]{MR2794027}.  
Let $n \ge 0$.  Consider a cell module $\cell A {n+1}  {(\mu, 0)}$  of $A_{n+1}$.
 We identify $\cell A {n+1}  {(\mu, 0)}$ with the cell module $\cell H {n+1} \mu$ of  $H_{n+1}$,  and we identify the chosen generators of these modules,   $\generator {A_{n+1}}   {(\mu, 0)}$ with
  $\generator {H_{n+1}}   {\mu}$.
 It follows from Axiom~\eqref{axiom: e(n-1) in An en An} that $J_n \subseteq J_{n+1}$ and hence $\Res_{A_n}^{A_{n+1}}(\cell A {n+1}  {(\mu, 0)}) J_n = 0$.    Therefore,
$\Res_{A_n}^{A_{n+1}}(\cell A {n+1}  {(\mu, 0)})$ is an $H_n$ module and can be identified with 
$\Res_{H_n}^{H_{n+1}}(\cell H {n+1} \mu)$.   Consider the chosen cell filtration of  $\Res_{H_n}^{H_{n+1}}(\cell H {n+1} \mu)$
$$
\{0\} \subseteq M_1 \subseteq M_1 \subseteq \cdots \subseteq M_r =  \Res_{H_n}^{H_{n+1}}(\cell H {n+1} \mu),
$$
with $M_j/M_{j-1} \cong \cell H n {\la \power j}$ for each $j$.   The isomorphism
$M_j/M_{j-1}  \to \cell H n {\la \power j}$ maps $\generator {H_{n+1}} \mu \dd {\la \power j} \mu {n+1} + M_{j-1}$ to 
$\generator {H_n}  {\la \power j}$.  But we identify $\cell H n {\la \power j}$ with
$\cell A n {({\la \power j}, 0)}$ and  $\generator {H_n}  {\la \power j}$ with $\generator {A_n}  {(\la \power j, 0)}$, 
so the isomorphism  sends  $\generator {A_{n+1}}   {(\mu, 0)}  \ddbar {\la \power j} {\mu} {n+1} + M_{j-1} = \generator {H_{n+1}}   {\mu} \dd {\la \power j} {\mu} {n+1} + M_{j-1} $ to 
$\generator {A_n}  {(\la \power j, 0)}$.  Thus we can choose
$\dd {(\la\power j, 0)}{(\mu,0)} {n+1} $ to be $\ddbar {\la\power j} {\mu} {n+1}$.  This proves point (1).  

Next, let $n \ge 1$ and consider a cell module 
$$
\Delta = \cell A {n+1} {(\mu, m+1)}  = 
\cell A {n-1} {(\mu, m)} \otimes_{A_{n-1}} e_n A_{n}
$$ 
of
$A_{n+1}$.   The restricted module $\Res_{A_n}^{A_{n+1}}(\Delta)$ is $ \cell A {n-1} {(\mu, m)} \otimes_{A_{n-1}} e_n A_{n}$ regarded as a right $A_n$ module.  But  $e_n A_n \cong A_n$ as an $A_{n-1}$--$A_n$  module,
so we have an isomorphism
$$
\varphi:  \cell A {n-1} {(\mu, m)} \otimes_{A_{n-1}} e_n A_{n} \to 
 \cell A {n-1} {(\mu, m)} \otimes_{A_{n-1}}  A_{n}   = \Ind_{A_{n-1}}^{A_{n}}( \cell A {n-1} {(\mu, m)} ),
$$
defined by $\varphi(x \otimes e_n a) = x \otimes a$.   We suppose we already have a chosen cell filtration of $\Ind_{A_{n-1}}^{A_{n}}( \cell A {n-1} {(\mu, m)} )$,
$$
\{0\} \subseteq N_1 \subseteq N_2 \subseteq \cdots \subseteq N_s =   \Ind_{A_{n-1}}^{A_{n}}( \cell A {n-1} {(\mu, m)} ),
$$
with isomorphisms $N_j/N_{j-1} \to \cell  A n {(\la \power j, \el_j)}$,  as well as  branching
factors  $\uu {(\mu, m)} {(\la \power j, \el_j)} n$ such that the isomorphism
$N_j/N_{j-1} \to \cell  A n {(\la \power j, \el_j)}$ takes
$\generator {A_{n-1}}  {(\mu, m)} \otimes_{A_{n-1}}  \uu {(\mu, m)} {(\la \power j, \el_j)} n + N_{j-1}$   to
$\generator {A_n} {(\la \power j, \el_j)}$.    Pulling all this data back via $\varphi$, we have a cell filtration of $\Res_{A_n}^{A_{n+1}}(\Delta)$, 
 $$
\{0\} \subseteq N'_1 \subseteq N'_2 \subseteq \cdots \subseteq N'_s =   \Res(\Delta),
$$
with isomorphisms 
$\varphi_j: N'_j/N'_{j-1} \to \cell  A n {(\la \power j, \el_j)}$  taking
$\generator {A_{n-1}}  {(\mu, m)} \otimes_{A_{n-1}}  e_n \uu {(\mu, m)} {(\la \power j, \el_j)} n + N'_{j-1}$  to
$\generator {A_n} {(\la \power j, \el_j)}$.  But
$\generator {A_{n-1}}  {(\mu, m)} \otimes_{A_{n-1}}  e_n$ is the generator
$\generator {A_{n+1}}   {(\mu, m+1)}$ of  $\Delta$.   Thus   
$$
\varphi_j : \generator {A_{n+1}}   {(\mu, m+1)}\uu {(\mu, m)} {(\la \power j, \el_j)} n + N'_{j-1}
\mapsto \generator {A_n} {(\la \power j, \el_j)}.
$$
 This means that we can choose $\uu {(\mu, m)} {(\la \power j, \el_j)} n$ for
 $\dd {(\la \power j, \el_j)}  {(\mu, m+1)} {n+1}$.  This proves point (2).

Next we turn to the branching factors for induced modules.  Statement (3) is evident when $n = 0$ since $A_0 = H_0 = R$ and $A_1 = H_1$.  Statement (4) only makes sense when $n \ge 1$,  so it remains to verify both statements (3) and (4) for $n \ge 1$.  
The argument is an elaboration of the proof of  Proposition 4.14 in ~\cite{MR2794027}.

Let $n \ge 1$ and let $\Delta$ be a cell module of $A_n$.
According to ~\cite[Proposition~4.14]{MR2794027}, $\Delta \otimes_{A_n}  J_{n+1}$ imbeds in $\Ind_{A_n}^{A_{n+1}}(\Delta)$  and the quotient $\Ind_{A_n}^{A_{n+1}}(\Delta)/(\Delta \otimes_{A_n}  J_{n+1})$  is isomorphic to  $\Delta \otimes_{A_n}  H_{n+1}$.   Moreover,  both of the $A_{n+1}$--modules
$\Delta \otimes_{A_n}  J_{n+1}$ and $\Delta \otimes_{A_n}  H_{n+1}$ have cell filtrations;  one obtains a cell filtration of $\Ind_{A_n}^{A_{n+1}}(\Delta)$ by gluing the cell filtrations of the submodule and the quotient module.  The cell modules of $A_{n+1}$  appearing as subquotients of the cell filtration of
$\Delta \otimes_{A_n}  J_{n+1}$  are of the form $\cell A {n+1}  {(\mu, m)}$ with $m > 0$;  that is, they are cell modules of the cellular ideal $J_{n+1}$.  The cell modules appearing as subquotients of the cell filtration of $\Delta \otimes_{A_n}  H_{n+1}$  are of the form $\cell A {n+1}  {(\mu, 0)}$; that is, they are cell modules of the quotient algebra $H_{n+1}$.   

Now consider in particular a cell module $\cell A n {(\la, 0)}$ of $A_n$ for $n \ge 1$.   According to the previous paragraph, to find the branching factors $\uu {(\la, 0)} {(\mu, 0)} {n+1}$ with $\mu \in \widehat H_{n+1}$, we have only to construct a particular cell filtration of $\cell A n {(\la, 0)} \otimes_{A_n} H_{n+1}$.  
We identify $\cell A n {(\la, 0)}$ with the cell module  $\cell H n \la$ of $H_n$, and 
$\cell A n {(\la, 0)} \otimes_{A_n} H_{n+1}$  with $\cell H n \la \otimes_{H_n} H_{n+1} = \Ind_{H_n}^{H_{n+1}}(\cell H n \la )$.   The remainder of the proof of statement (3) proceeds by considering the chosen
cell filtration of $ \Ind_{H_n}^{H_{n+1}}(\cell H n \la )$ and the associated branching factors
$\uu \la \mu {n+1}$;  the proof is similar to the proof of statement (1).  

Finally, let $n \ge 1$ and consider a cell module $\Delta = \cell A n {(\la, \el)}$ of $A_n$.  
Write $\Res(\Delta)$  for $\Res_{A_{n-1}}^{A_n}(\Delta)$.  
To find the branching factors  $\uu {(\la, \el)} {(\mu, m+1)} {n+1}$,  we have to construct a particular cell filtration of
$\Delta \otimes_{A_n}  J_{n+1}$.    By axiom (7) and \cite[Corollary~4.6]{MR2794027}, we have
$$
J_{n+1} = A_{n} e_{n} A_{n} \cong A_n e_n \otimes_{A_{n-1}}  e_n A_n,
$$
as   $A_n$--$A_{n+1}$  bimodules,  the isomorphism being given by $a_1 e_n a_2 \mapsto a_1 e_n \otimes_{A_{n-1}} e_n a_2$.   We have $A_n e_n \cong A_n$ as an $A_n$--$A_{n-1}$  bimodule, so
$$
\begin{aligned}
\Delta  \otimes_{A_n}  J_{n+1} &\cong \Delta  \otimes_{A_n}  A_n e_n \otimes_{A_{n-1}}  e_n A_n \\
&\cong \Delta  \otimes_{A_n}  A_n  \otimes_{A_{n-1}}  e_n A_n  \\
&\cong \Res(\Delta)  \otimes_{A_{n-1}}  e_n A_n.
\end{aligned}
$$
The composite isomorphism  $\varphi: \Delta  \otimes_{A_n}  J_{n+1}  \to  \Res(\Delta)  \otimes_{A_{n-1}}  e_n A_n $  is given by $\varphi(x \otimes_{A_n} a_1 e_n a_2) = x a_1 \otimes_{A_{n-1}} e_n a_2$.  
In particular,  $\varphi(x \otimes_{A_{n}} e_n )= x \otimes_{A_{n-1}} e_n$.   
We assume that we have a chosen cell filtration of $\Res(\Delta)$, 
$$
\{0\} \subseteq M_1 \subseteq M_2 \subseteq \cdots  \subseteq M_t = \Res(\Delta),
$$  
with isomorphisms $ M_j/M_{j-1} \to \cell A {n-1}  {(\mu \power j, m_j)}$ and we have chosen branching factors $\dd {(\mu \power j, m_j)} {(\la, \el)} {n}$ such that  the isomorphism 
$ M_j/M_{j-1} \to \cell A {n-1}  {(\mu \power j, m_j)}$ takes 
$\generator {A_n} {(\la, \el)}  \dd {(\mu \power j, m_j)} {(\la, \el)} {n} + M_{j-1}$ to
$\generator {A_{n-1}}  {(\mu \power j, m_j)}$.  
By~\cite[Lemma~4.12]{MR2794027},  $M_{j-1} \otimes_{A_{n-1}} e_n A_n$ imbeds in $M_{j} \otimes_{A_{n-1}} e_n A_n$ for each $j$, and the quotient is isomorphic to 
$$
M_j/M_{j-1} \otimes_{A_{n-1}} e_n A_n \cong \cell A {n-1}  {(\mu \power j, m_j)} \otimes_{A_{n-1}} e_n A_n   = \cell A {n+1}  {(\mu \power j, m_j+1)} .
$$
Writing $M_j' = M_j \otimes_{A_{n-1}} e_n A_n$,  we obtain a cell filtration of $\Res(\Delta) \otimes_{A_{n-1}} e_n A_n$,
$$
\{0\} \subseteq M'_1 \subseteq M'_2 \subseteq \cdots  \subseteq M'_t = \Res(\Delta) \otimes_{A_{n-1}} e_n A_n,
$$  
with isomorphisms $M_j'/M_{j-1}' \to \cell A {n+1}  {(\mu \power j, m_j+1)} $ taking
$\generator {A_n} {(\la, \el)}  \dd {(\mu \power j, m_j)} {(\la, \el)} {n} \otimes_{A_{n-1}} e_n + M_{j-1}'$ to
$\generator {A_{n-1}}  {(\mu \power j, m_j)} \otimes_{A_{n-1}} e_n = \generator {A_{n+1}}  {(\mu \power j, m_j+1)}$.   Pulling back this data via the isomorphism  $\varphi: \Delta  \otimes_{A_n}  J_{n+1}  \to  \Res(\Delta)  \otimes_{A_{n-1}}  e_n A_n $, we get a cell filtration of $ \Delta  \otimes_{A_n}  J_{n+1}$,
$$
\{0\} \subseteq M''_1 \subseteq M''_2 \subseteq \cdots  \subseteq M''_t =  \Delta  \otimes_{A_n}  J_{n+1},
$$  
with isomorphisms  $M_j''/M_{j-1}'' \to \cell A {n+1}  {(\mu \power j, m_j+1)} $ taking
$$\generator {A_n} {(\la, \el)}  \dd {(\mu \power j, m_j)} {(\la, \el)} {n} \otimes_{A_{n}} e_n + M_{j-1}''
=\generator {A_n} {(\la, \el)}   \otimes_{A_{n}} \dd {(\mu \power j, m_j)} {(\la, \el)} {n} e_n + M_{j-1}''
$$ 
to $\generator {A_{n+1}}  {(\mu \power j, m_j+1)}$.  We conclude that we can take
$$
\uu {(\la, \el)}  {(\mu\power j, m_j + 1)}  {n+1} =  \dd {(\mu \power j, m_j)} {(\la, \el)} {n} \, e_n,
$$
which proves point (4), and completes the proof of the theorem.
\end{proof}

 Next we apply the recursion of \hyperref[theorem:  recursive determination of the branching factors]{Theorem~\ref*{theorem:  recursive determination of the branching factors}} to obtained closed formulas for the branching factors for the tower $(A_n)_{n \ge 0}$.
Since the branching diagram $\hat A$ is obtained by reflections from the branching diagram $\widehat H$, it follows that $(\la, \el) \to (\mu, m)$  only if $m \in \{\el, \el + 1\}$;  in particular,  $(\la, \el) \to (\mu, 0)$  only if $\el = 0$.   Moreover, $(\la, \el) \to (\mu, \el)$ in $\hat A$ if and only if $\la \to \mu$ in $\widehat H$,  and 
$(\la, \el) \to (\mu, \el + 1)$ in $\hat A$ if and only if $\mu \to \el$ in $\widehat H$.

\begin{theorem} \label{theorem:  closed form determination of the branching factors}
The branching factors  for the tower $(A_n)_{n\ge 0}$ can be chosen to satisfy:
\begin{enumerate}
\item $\dd {(\la, \el)} {(\mu, \el)} {n+1} = \ddbar \la \mu {n + 1 - 2\el}  e_{n-1} \power {\el}$.
\item $\uu {(\la, \el)} {(\mu, \el)} {n+1} = \uubar \la \mu {n + 1 - 2\el}  e_{n} \power {\el}$.
\item $\dd {(\la, \el)} {(\mu, \el+1)} {n+1}  = \uubar \mu \la {n-2\el} e_{n-1} \power \el$.  
\item   $\uu {(\la, \el)} {(\mu, \el+1)} {n+1}  = \ddbar \mu \la {n-2\el} e_{n} \power {\el+1}$. 
\end{enumerate}
\end{theorem}

\begin{proof}  We suppose that the branching factors are determined by the recursive formulas of \hyperref[theorem:  recursive determination of the branching factors]{Theorem~\ref*{theorem:  recursive determination of the branching factors}}.  

For part (1), the formula is given by \hyperref[theorem:  recursive determination of the branching factors]{Theorem~\ref*{theorem:  recursive determination of the branching factors}}, part (1)  if $\el = 0$.  Assume $\el >0$ and observe
$$
\dd {(\la, \el)} {(\mu, \el)} {n+1} = \uu {(\mu, \el-1)} {(\la, \el)} n = \dd {(\la, \el -1 )} {(\mu, \el -1)} {n-1} e_{n-1};
$$
Repeating this a total of $\el$ times, we get
$$
\dd {(\la, \el)} {(\mu, \el)} {n+1}  = \dd {(\la, 0)} {(\mu, 0)} {n+1 - 2\el} e_{n+1 - 2\el} \cdots e_{n-3} e_{n-1}
=  \ddbar \la \mu {n + 1 - 2 \el} e_{n-1} \power {\el}.
$$
The proof of part (2) is similar.  
For part (3), we have
$$
\dd {(\la, \el)} {(\mu, \el+1)} {n+1} = \uu {(\mu, \el)} {(\la, \el)} n,
$$
and we apply part (2) to get the desired formula.  For part (4),
$$
\uu {(\la, \el)} {(\mu, \el+1)} {n+1}  = \dd {(\mu, \el)} {(\la, \el)} n  e_n.
$$
Apply part (1)  to get  
$$
\uu {(\la, \el)} {(\mu, \el+1)} {n+1}  = \dd {(\mu, \el)} {(\la, \el)} n  e_n = \ddbar \mu \la {n-2\el}  e_{n-2} \power \el e_n = \ddbar \mu \la {n-2\el}  e_n \power {\el+1}.
$$
\end{proof}

\section{Applications}   \label{section: applications}

We will apply our results to the following examples:  the BMW algebras, the Brauer algebras, the partition algebras, and the Jones--Temperley--Lieb algebras.  For each example, let $R_0$  denote the generic ground ring and let $R = R_0[\deltabold\inv]$,  where $e_1^2 = \deltabold e_1$.   We show that our results apply to the algebras defined over $R$, and we give explicit Murphy bases for the algebras.    

We are then able to check, by a computation specific to each algebra,  that the Murphy bases  are, in fact, 
 bases for the algebras defined over the generic ground ring $R_0$.  

\subsection{Preliminaries on tangle diagrams}  \label{subsection:  preliminaries on tangle diagrams}
Several of our examples involve {\em tangle diagrams} in the rectangle $\mathcal R = [0, 1] \times [0, 1]$.
Fix points $a_i \in [0, 1]$,  $i \ge 1$,  with $0 < a_1 < a_2 < \cdots$. Write
$\p i = (a_i, 1)$ and $\overline{ \p i} = (a_i, 0)$.

Recall that a {\em knot diagram} means a collection of piecewise smooth closed curves in the plane
which may have intersections and self-intersections, but only simple
transverse intersections.  At each intersection or crossing, one of the
two strands (curves) which intersect is indicated as crossing
over the other.  

An {\em $(n,n)$--tangle diagram}  is a piece of a
knot diagram in $\mathcal R$  consisting of exactly $n$ topological intervals and possibly some number of closed curves, such that:  (1)   the endpoints of the intervals are the points $\p 1, \dots, \p n, \pbar 1, \dots, \pbar n$, and these are the only points of intersection of the family of curves with the boundary of the rectangle, and (2)  each interval intersects the boundary of the rectangle transversally.

An  {\em $(n,n)$--Brauer diagram} is a ``tangle'' diagram  containing no closed curves, 
in which information about over and under crossings is ignored.  Two Brauer diagrams are identified if the pairs of boundary points joined by curves is the same in the two diagrams.
By convention, there is a unique $(0, 0)$--Brauer diagram,  the empty diagram with no curves.
For $n \ge 1$,  the number of $(n,n)$--Brauer diagrams is  $(2n-1)!! = (2n-1)(2n-3)\cdots (3)(1)$.

\ignore{
A {\em Temperley--Lieb} diagram is a Brauer diagram without crossings.  For $n \ge 0$,  number of $(n, n)$--Temperley--Lieb diagrams is the Catalan number $\frac{1}{n+1} {2n \choose n}$.
}

For any of these types of diagrams, we call $P = \{\p 1, \dots, \p n, \pbar 1,\dots,  \pbar n\}$ the set of {\em vertices} of the diagram,  $P^+ =    \{\p 1, \dots, \p n\}$ the set of {\em top vertices},  and
$P^- =  \{\pbar 1,\dots,  \pbar n\}$ the set of {\em bottom vertices}.  A curve or {\em strand} in the diagram is called a {\em vertical} or {\em through} strand if it connects a top vertex and a bottom vertex,  and a {\em horizontal} strand if it connects two top vertices or two bottom vertices.

\subsection{Birman--Murakami--Wenzl algebras}  \label{subsection: BMW algebras}
\def\hods{\unskip\kern.55em\ignorespaces}
\def\mods{\vskip-\lastskip\vskip4pt}
\def\ods{\vskip-\lastskip\vskip4pt plus2pt}
\def\bods{\vskip-\lastskip\vskip12pt plus2pt minus2pt}

\begin{definition}  \label{definition: BMW algebra}
Let $S$ be an integral domain with invertible elements $z$ and $q$ and an element $\delta$ satisfying $z\inv - z = (q\inv -q)(\delta -1)$.  The {\em Birman--Murakami--Wenzl algebra}
$W_n=W_n(S; z, q, \delta)$ is the unital $S$--algebra  with generators $g_i^{\pm 1}$  and $e_i$ ($1 \le i \le n-1$) and relations:
\begin{enumerate}
\item (Inverses) \hods $g_i g_i\inv = g_i\inv g_i = 1$.
\item (Essential idempotent relation)\hods $e_i^2 = \delta e_i$.
\item (Braid relations) \hods $g_i g_{i+1} g_i = g_{i+1} g_i g_{i+1}$ 
and $g_i g_j = g_j g_i$ if $|i-j|  \ge 2$.
\item (Commutation relations)  \hods $g_i e_j = e_j g_i$  and
$e_i e_j = e_j e_i$  if $|i-j|\ge 2$. 
\item (Tangle relations)\hods $e_i e_{i\pm 1} e_i = e_i$, $g_i
g_{i\pm 1} e_i = e_{i\pm 1} e_i$, and $ e_i  g_{i\pm 1} g_i=   e_ie_{i\pm 1}$.
\item (Kauffman skein relation)\hods  $g_i - g_i\inv = (q - q\inv)(1- e_i )$.
\item (Untwisting relations)\hods $g_i e_i = e_i g_i = z\inv e_i$,
and $e_i g_{i \pm 1} e_i = z e_i$.
\end{enumerate}
\end{definition}
Morton and Wassermann~\cite{Morton-Wassermann} give a realization of the BMW algebra as an algebra of  $(n, n)$--tangle diagrams modulo regular isotopy and the following  {\em Kauffman skein relations:}  
\def\negcrossing{ 
\begin{array}{c}
\begin{tikzpicture}
\draw (-6.5,0) -- (-5.7,-0.8); 
\draw (-6.5,-0.8) -- (-6.2,-0.5);
\draw (-6.0,-0.3) -- (-5.7,0);
\end{tikzpicture}
\end{array}
}

\def\poscrossing
{ 
\begin{array}{c}
\begin{tikzpicture}
\draw (-4.2,0) --  (-5.0,-0.8); 
\draw (-5,0) -- (-4.7,-0.3);
\draw (-4.2,-0.8) -- (-4.5,-0.5);
\end{tikzpicture}
\end{array}
}

\def\horizontalSmoothing
{
\begin{array}{c}
\begin{tikzpicture}
\draw (-2.0,0) .. controls (-1.6,-0.3) .. (-1.2,0);
\draw (-2.0,-0.8) .. controls (-1.6,-0.5) .. (-1.2,-0.8);
\end{tikzpicture}
\end{array}
}

\def\verticalSmoothing
{
\begin{array}{c}
\begin{tikzpicture}
\draw (1.0,0) .. controls (1.1,-0.4) .. (1.0,-0.8);
\draw (1.6,0) .. controls (1.5,-0.4) .. (1.6,-0.8);
\end{tikzpicture}
\end{array}
}

\def\negTwist{
\begin{array}{c}
\begin{tikzpicture}
\draw (0,0) -- (0.3,-0.3); 
\draw (0.4,-0.4) .. controls (1.0,-1.0) and  (-0.4,-1.0)   .. (0.4,-0.3) -- (0.8,0);
\end{tikzpicture}
\end{array}
}

\def\posTwist{
\begin{array}{c}
\begin{tikzpicture}
\draw (0.5,-0.3) -- (0.8,0); 
\draw (0,0) -- (0.4,-0.3) .. controls (1.2,-1.0) and  (-0.2,-1.0) ..  (0.4,-0.4);
\end{tikzpicture}
\end{array}
}

\def\HorizontalLine{
\begin{array}{c}
\begin{tikzpicture}
\draw (1.7,0) .. controls (2.1,-0.4) .. (2.5,0);
\end{tikzpicture}
\end{array}
}

\begin{enumerate}
\item Crossing relation:\quad
$$\poscrossing \ - \  \negcrossing \ = \  (q-q\inv)\  \Bigg( \ \verticalSmoothing \ - \  \horizontalSmoothing \ \Bigg).$$

\item Untwisting relation:\quad
$$\negTwist = z\inv \ \HorizontalLine  \quad \text{and} \quad  \posTwist  = z \ \HorizontalLine.$$
\item  Free loop relation:  $T\, \cup \, \bigcirc = \delta \, T, $  where $T\, \cup \, \bigcirc$ means the union of a tangle diagram $T$ and a closed loop having no crossings with $T$.
\end{enumerate}
In the tangle picture, the generators $g_i$ and $e_i$ are represented by the diagrams
\begin{align*}
\begin{matrix}
\begin{tikzpicture}
\node at (-7.4,-0.5) {$\displaystyle g_i= \color{black}$};
\node at (-6.3,-0.4) {$\displaystyle \cdots \color{black}$};
\draw (-6.8,0) -- (-6.8,-0.8);
\draw (-5.8,0) -- (-5.8,-0.8);
\draw (-5,0) -- (-4.7,-0.3);
\draw (-4.2,-0.8) -- (-4.5,-0.5);
\draw (-4.2,0) -- (-5.0,-0.8);
\draw (-3.4,0) -- (-3.4,-0.8);
\draw (-2.4,0) -- (-2.4,-0.8);
\filldraw [black] (-6.8,0) circle (1.2pt);
\filldraw [black] (-6.8,-0.8) circle (1.2pt);
\filldraw [black] (-5.8,0) circle (1.2pt);
\filldraw [black] (-5.8,-0.8) circle (1.2pt);
\filldraw [black] (-5.0,0) circle (1.2pt);
\filldraw [black] (-5.0,-0.8) circle (1.2pt);
\filldraw [black] (-4.2,0) circle (1.2pt);
\filldraw [black] (-4.2,-0.8) circle (1.2pt);
\filldraw [black] (-3.4,0) circle (1.2pt);
\filldraw [black] (-3.4,-0.8) circle (1.2pt);
\filldraw [black] (-2.4,0) circle (1.2pt);
\filldraw [black] (-2.4,-0.8) circle (1.2pt);
\node at (-2.9,-0.4) {$\displaystyle \cdots \color{black}$};
\node at (-5.0,-1.1) {$\displaystyle i \color{black}$};
\node at (-4.1,-1.1) {$\displaystyle i+1 \color{black}$};
\node at (-1.4,-0.4) {$\displaystyle \text{and} \color{black}$};
\node at (0.0,-0.5) {$\displaystyle e_i= \color{black}$};
\node at (1.1,-0.4) {$\displaystyle \cdots \color{black}$};
\draw (0.6,0) -- (0.6,-0.8);
\draw (1.6,0) -- (1.6,-0.8);
\draw (2.4,0)  .. controls (2.6,-0.2) and (3.0,-0.2) .. (3.2,0);
\draw (2.4,-0.8)  .. controls (2.6,-0.6) and (3.0,-0.6) .. (3.2,-0.8);
\draw (4.0,0) -- (4.0,-0.8);
\draw (5.0,0) -- (5.0,-0.8);
\filldraw [black] (0.6,0) circle (1.2pt);
\filldraw [black] (0.6,-0.8) circle (1.2pt);
\filldraw [black] (1.6,0) circle (1.2pt);
\filldraw [black] (1.6,-0.8) circle (1.2pt);
\filldraw [black] (2.4,0) circle (1.2pt);
\filldraw [black] (2.4,-0.8) circle (1.2pt);
\filldraw [black] (3.2,0) circle (1.2pt);
\filldraw [black] (3.2,-0.8) circle (1.2pt);
\filldraw [black] (4,0) circle (1.2pt);
\filldraw [black] (4,-0.8) circle (1.2pt);
\filldraw [black] (5,0) circle (1.2pt);
\filldraw [black] (5,-0.8) circle (1.2pt);
\node at (4.5,-0.4) {$\displaystyle \cdots \color{black}$};
\node at (2.4,-1.1) {$\displaystyle i \color{black}$};
\node at (3.3,-1.1) {$\displaystyle i+1 \color{black}$};
\end{tikzpicture}
\end{matrix}
\end{align*}
There is evidently a unital algebra homomorphism from $W_n$ to $W_{n+1}$ taking generators to generators; from the tangle realisation, one can see that this homomorphism is injective, so $W_n$ is a subalgebra of $W_{n+1}$. 
The symmetry of the defining relations for $W_n$ ensures that the  assignments 
\begin{align*}
g_i^*=g_i, \qquad e_i^*=e_i. 
\end{align*}
determine an involutory algebra anti-automorphism of $W_n$.   In the tangle picture, the involution $*$ acts on tangles by flipping them over
a horizontal line.

If $v\in\mathfrak{S}_n$ and $v=s_{i_1}s_{i_2}\cdots s_{i_j}$ is a reduced expression then the element
$g_v=g_{i_1}g_{i_2}\cdots g_{i_j}$
depends only on $v$. For $i,j=1,2,\ldots,$ let
\begin{align*}
g_{i,j}&=
\begin{cases}
g_ig_{i+1}\cdots g_{j-1},&\text{if $j\ge i$,}\\
g_{i-1}g_{i-2}\cdots g_j,&\text{if $i>j$.}
\end{cases}
\end{align*}
Let $J_n$ denote the ideal $W_ne_{n-1}W_n$;  in the tangle picture, this is the ideal spanned by tangle diagrams with at least one horizontal strand.
The map $W_n/J_n \rightarrow \mathcal{H}_n=\mathcal{H}_n(S, q^2)$ determined by
$g_v+ J_n  \mapsto T_v$, for $v\in\mathfrak{S}_n$,  is an algebra isomorphism. 
\ignore{
\begin{align*}
g_v+ J_n&\mapsto T_v, \quad \text{for $v\in\mathfrak{S}_n,}
\end{align*}
}

\subsubsection{The Murphy basis} 
The generic  ground ring for the BMW algebras is
 $$
 R_0 = \Z[\zbold^{\pm1}, \qbold^{\pm1}, \deltabold]/\langle \zbold\inv - \zbold = (\qbold\inv - \qbold)(\deltabold -1) \rangle,
 $$
 where $\zbold$, $\qbold$, and $\deltabold$ are indeterminants over $\Z$.
$R_0$ is an integral domain whose field of fractions is $F \cong \Q(\zbold, \qbold)$, with
\begin{equation}  \label{equation:  solution for delta in field of fractions}
\deltabold   
 = \frac{\zbold - \zbold\inv}{\qbold - \qbold\inv} + 1  = \frac{(\zbold + \qbold)(\qbold \zbold -1)}{\zbold(\qbold^2  - 1)}.
\end{equation}
Let $R = R_0[\deltabold\inv]$, and     
Write $\bmw n(R)$ for 
 $\bmw n(R; \zbold, \qbold, \deltabold)$  and $\mathcal H_n(R)$ for $\mathcal H_n(R; \qbold^2)$.
 It is observed  in ~~\cite{MR2794027}, Section 5.4,  that the pair of towers  $(\bmw n(R))_{n \ge 0}$  and $(\mathcal H_n(R))_{n \ge 0}$ satisfy the framework axioms~\eqref{axiom: involution on An}--\eqref{axiom: semisimplicity}  of \hyperref[subsection:  cellularity and jones setting and correction]{Section~\ref*{subsection:  cellularity and jones setting and correction}}.  Axiom~\eqref{axiom Hn coherent}  holds by \hyperref[corollary:  Hecke algebras strongly coherent]{Corollary~\ref*{corollary:  Hecke algebras strongly coherent}}.   Axiom~\eqref{axiom:  Delta J} hold for $\bmw n(R)$,  by the remarks at the end of \hyperref[subsection:  cellularity and jones setting and correction]{Section~\ref{subsection:  cellularity and jones setting and correction}}.  Finally, Axiom~\eqref{axiom: Hn cyclic cellular}  holds by \hyperref[corollary Hecke algebras cyclic cellular]{Corollary~\ref*{corollary Hecke algebras cyclic cellular}}.  Therefore, by \hyperref[theorem: jones tower coherent tower cyclic cellular]{Theorem~\ref*{theorem: jones tower coherent tower cyclic cellular}}, the tower of algebras $(\bmw n(R))_{n \ge 0}$ is a strongly coherent tower of cyclic cellular algebras.

By the discussion in \hyperref[subsection results on cellularity from Goodman Graber]{Section~\ref*{subsection results on cellularity from Goodman Graber}}, the partially ordered set $\widehat{W}_n$ in the cell datum for $\bmw n(R)$  can be realized as
\begin{align*}
\widehat{W}_n=\left\{ 
(\lambda,l)\,\big|\,  0 \le \el \le \lfloor n/2\rfloor \text{ and } 
\lambda\in\widehat{\mathcal{H}}_{n-2l}
\right\}
\end{align*}
with $(\lambda,l)\unrhd(\mu,m)$  if 
$\el > m$ or if $\el = m$ and $\la \rhd \mu$ in $\widehat{\mathcal H}_{n-2 \el}$.  The branching diagram 
$\widehat{W}$ of the tower $(\bmw n)_{n \ge 0}$ is that obtained by reflections from $\widehat{\mathcal H}$ ($=$ Young's lattice).   Thus, the branching relation is $(\la, \el) \to (\mu, m)$ only if $m \in \{\el, \el + 1\}$; $(\la, \el) \to (\mu, \el)$ if and only if $\la \to \mu$ in Young's lattice, and $(\la, \el) \to (\mu, \el + 1)$ if and only if $\mu \to \la$ in Young's lattice.

For each $n \ge 0$ for for each $\mu \in \widehat{\mathcal H}_n$, define
\begin{equation*}
{c}_{(\mu,0)}=\sum_{v\in\mathfrak{S}_\mu}q^{\el(v)}g_v;
\end{equation*}
thus ${c}_{(\mu,0)}$ is a preimage in $\bmw n$ of $m_\mu \in \mathcal H_n$ (defined in Equation~\eqref{equation:  element m mu in Hecke algebra}).
For $n \ge 2$ and $(\mu,m)\in\widehat{W}_n$, let
\begin{align*}
c_{(\mu,m)}={c}_{(\mu,0)}e_{n-1}^{(m)},
\end{align*}
where $e_{n-1}^{(m)}$ is defined in Equation~\eqref{notation e sub i super ell}.

Let $i \ge 1$ and $\lambda\in\widehat{\mathcal{H}}_{i-1}$ and $\mu\in\widehat{\mathcal{H}}_{i}$, with $\la \to \mu$  in $\widehat{\mathcal H}$.  
If $\mu=\lambda\cup\{(r,\mu_r)\}$, let $a=\sum_{j=1}^{r}\mu_j$, and define
\begin{align}
\bar{u}_{\lambda\to\mu}^{\,(i)}=g_{i,a}
\sum_{k=0}^{\lambda_r}q^{k}g_{a,a-k}\quad \text{and} \quad
\bar{d}_{\lambda\to\mu}^{\,(i)}=g_{a,i}.
\end{align}
These are liftings in $\bmw i$ of the branching factors in the Hecke algebra $\mathcal H_i$, as determined in \hyperref[corollary:  d branching factors for the Hecke algebra]{Corollary~\ref*{corollary:  d branching factors for the Hecke algebra}} and \hyperref[corollary: u branching factors for the Hecke algebra]{Corollary~\ref*{corollary: u branching factors for the Hecke algebra}}.

\ignore{
By \hyperref[theorem:  closed form determination of the branching factors]{Theorem~\ref*{theorem:  closed form determination of the branching factors}}, the branching factors for the tower $(W_n)_{n \ge 0}$ can be chosen as follows:
If $(\lambda,l)\in\widehat{W}_i$, $(\mu,l)\in\widehat{W}_{i+1}$, and $(\lambda,l)\to(\mu,l)$ in $\widehat{W}$, define 
\begin{align*}
d_{(\lambda,l)\to(\mu,l)}^{(i+1)}= \bar{d}_{\lambda\to\mu}^{\,(i-2l+1)}e_{i-1}^{(l)}
\end{align*}
and for $(\lambda,l)\in\widehat{W}_i$, $(\mu,l+1)\in\widehat{W}_{i+1}$, such that $(\lambda,l)\to(\mu,l+1)$ in $\widehat{W}$, define 
\begin{align*} d_{(\lambda,l)\to(\mu,l+1)}^{(i+1)}=\bar{u}_{\mu\to\lambda}^{(i-2l)}e_{i-1}^{(l)}.
\end{align*}
For   $(\la, \el) \in \widehat{W}_i$  and a path
$$\mathfrak{t}=((\emptyset,0),(\lambda^{(1)},l_1),\ldots,(\lambda^{(i)},l_{i}) = (\la, \el) )\in\widehat{W}_{i}^{(\lambda,l)},$$
 let 
\begin{align*}
d_\mathfrak{t}=
d_{(\lambda^{(i-1)},l_{i-1})\to (\lambda^{(i)},l_{i})}^{(i)}
d_{(\lambda^{(i-2)},l_{i-2})\to (\lambda^{(i-1)},l_{i-1})}^{(i-1)}
\cdots d_{(\emptyset,0)\to(\lambda^{(1)},l_{1})}^{(1)}.
\end{align*}
}  

For $(\la, \el) \in \widehat W_{i}$ and $(\mu, m) \in \widehat W_{i+1}$ with $(\la, \el) \to (\mu, m)$, determine the branching factors $\dd {(\la, \el)} {(\mu, m)} {i+1}$  
according to the formulas of \hyperref[theorem:  closed form determination of the branching factors]{Theorem~\ref*{theorem:  closed form determination of the branching factors}};
for a path $\mft \in \widehat W_n \power {\la, \el}$,  define $d_\mft$ to be the ordered product of these branching factors along the path $\mft$, as in Equation\eqref{equation: ordered product of d coefficients}. 
From  \hyperref[label:1]{Corollary~\ref*{label:1}} \ignore{and \hyperref[theorem:  closed form determination of the branching factors]{Theorem~\ref*{theorem:  closed form determination of the branching factors}} }
we obtain:

\begin{proposition} \label{proposition:  Murphy basis of BMW over generic with delta inverse adjoined}
 Let $R_0$ denote the generic ground ring for the BMW algebras and let $R = R_0[\deltabold\inv]$.  
 Let $W_n(R) = W_n(R; \zbold, \qbold, \deltabold)$ denote the BMW algebra over $R$.  
For $n \ge 0$,   the set 
\begin{align}\label{b-m-b}
\mathscr{W}_n=
\left\{
d_\mathfrak{s}^*c_{(\lambda,l)}d_\mathfrak{t}\,\big|\,
\text{$\mathfrak{s},\mathfrak{t}\in\widehat{W}_n^{(\lambda,l)}$, $(\lambda,l)\in\widehat{W}_n$}
\right\}
\end{align}
is an $R$--basis for $\bmw n(R)$, and $(\bmw n(R),*,\widehat{W}_n,\unrhd,\mathscr{W}_n)$ is a cell datum for $\bmw n(R)$. 
\end{proposition}

In the remainder of this section,  we will show that the Murphy bases $\mathscr W_n$  are bases of the BMW algebras defined over the generic ground ring $R_0$.     First note that the elements $d_\mathfrak{s}^*c_{(\lambda,l)}d_\mathfrak{t}$  are actually defined over $R_0$ and are linearly independent.  The issue is to show that $\mathscr W_n$  spans the BMW algebra over $R_0$.    To do this, we examine the transition matrix between a Morton--Wassermann basis of the BMW algebra and $\mathscr W_n$.

\subsubsection{Morton-Wasserman tangle bases}  
We begin by describing the Morton--Wassermann tangle bases of the BMW algebras.  We identify the BMW algebras with their tangle realizations, following ~\cite{Morton-Wassermann}.

To each $(n, n)$--tangle diagram $T$, associate a Brauer diagram $\conn(T)$ by deleting the closed strands in $T$  and forgetting information about over and under crossings.  Thus   $\conn(T)$ has a strand connecting two vertices if and only if $T$ has a strand connecting the same two vertices.

Order the vertices of a tangle or Brauer diagram by $\p 1 < \p2 \cdots < \p n < \pbar n <\cdots < \pbar 1$, that is,  in clockwise order around the boundary of $\mathcal R$.    The {\em length} $\ell(D)$   of a Brauer diagram $D$ is the minimal number of crossings of strands in a physical drawing of the diagram, that is, the number of $4$--tuples of vertices $(a, b , c ,d)$  such that $a < b < c < d$  and $(a, c)$ and $(b, d)$ are strands of $D$.  

\begin{definition}  Say that an $(n,n)$--tangle diagram $T$  is {\em layered}  with respect to some total ordering $(t_1, t_2, \dots,  t_k)$ of  its strands,  if (1)  whenever $i <j$,   every crossing of $t_i$ with $t_j$ is an over crossing,  and (2)   each individual strand of $T$ is unknotted, i.e. ambient isotopic to a strand with no self--crossings.       Say that $T$ is {\em layered} if it is layered with respect to some total ordering of its strands.   Say that a layered tangle diagram is {\em simple} if it has no closed strands and no strand has self--crossings.
\end{definition}

Note that any simple layered tangle diagram $T$  is ambient isotopic to a simple layered tangle  diagram in which any two distinct strands have at most one crossing;    the number of crossings  in such a representative of $T$ is the length of $\conn(T)$.  

\begin{theorem}[\cite{Morton-Wassermann},  Theorems 2.10 and 4.2]   For each $(n, n)$--Brauer diagram $D$,  chose a simple  layered   $(n, n)$--tangle diagram $T$ with $\conn(T) = D$.  Then the resulting collection of tangle diagrams is  a basis of the BMW algebra $\bmw n(R_0)$.   
\end{theorem}

Call any such basis a {\em Morton-Wassermann tangle basis}.

\begin{lemma}  \label{lemma:  transition between MW bases 0}
 Let $T$ and $S$  be simple layered $(n, n)$--tangle diagrams with the same underlying Brauer diagram,  $\conn(T) = \conn(S) = D$.      Then  $T - S$  is in the $\Z[\qbold - \qbold\inv]$--span of simple layered tangle diagrams with fewer than $\ell(D)$ crossings.
\end{lemma}

\begin{proof}   Assume without loss of generality that the number of crossings of $T$ and of $S$ is the length of $D$.   Suppose that $S$ is layered with respect to an ordering $(t_1, t_2, \dots, t_n)$ of its strands and $T$ is layered with respect to an ordering $(t_{\pi(1)},  t_{\pi(2)}, \dots, t_{\pi(n)})$  for some permutation $\pi$ of $\{1, 2, \dots, n\}$.   For brevity, say that $T$ is layered with respect to $\pi$.   The permutation $\pi$ may not be unique, so assume that $\pi$ has been chosen with minimal length for the given tangle diagram $T$.   

If $\pi$ is the identity permutation,  then $T$ and $S$ are ambient isotopic,  so represent the same element of $\bmw n$.   
Assume that $\pi$ is not the identity and assume inductively that the assertion holds when $T$ is replaced by a simple layered tangle diagram $T'$ with $\conn(T') = D$,  whenever $T'$  is layered with respect to a permutation $\pi'$ with $\ell(\pi') < \ell(\pi)$.  

Since $\pi$ is not the identity permutation, there exists $i$ such that $\pi(i) > \pi(i+1)$.  If  the strands $t_{\pi(i)}$ and $t_{\pi(i + 1)}$   do not cross, then $T$ is also layered with respect to the shorter permutation $\pi'= (i, i+1) \circ \pi$,  contradicting the choice of $\pi$ as having minimal length.   Therefore 
$t_{\pi(i)}$ and $t_{\pi(i +1)}$ have a (unique) crossing, with $t_{\pi(i)}$ crossing over $t_{\pi(i + 1)}$.  
Because $T$ is layered with respect to $\pi$  there is no third strand $t = t_{\pi(k)}$ such that $t_{\pi(i)}$ has an over crossing with  $t$ and $t$ has an over crossing with  $t_{\pi(i+1)}$.    Let $U$ be the tangle diagram obtained by changing the crossing of $t_{\pi(i)}$ and $t_{\pi(i +1)}$, and let $T_0$ and $T_\infty$ be the two tangle diagrams obtained by smoothing this crossing.    It follows that all three of these tangle diagrams are simple and layered,    $T_0$ and $T_\infty$  have fewer than $\ell(D)$ crossings, 
and by the Kauffman skein relation,
$$
T = U + (\qbold - \qbold\inv)(T_0 - T_\infty).
$$
Since $U$ is layered with respect to $\pi' = (i, i+1) \circ \pi$, with $\ell(\pi') = \ell(\pi) -1$,   the conclusion follows from the induction hypothesis.
  \end{proof}

\begin{proposition}  \label{proposition:  transition between MW bases 1}
 Let $\mathcal B$ be a Morton--Wassermann tangle basis of $\bmw n(R_0)$ and let $T$ be a simple layered $(n,n)$--tangle diagram.  The the coefficients of $T$ with respect to the basis $\mathcal B$ are in $\Z[\qbold - \qbold\inv]$.   In fact, $T$ is in the  $\Z[\qbold - \qbold\inv]$--span of basis elements with no more than
 $\ell(D)$ crossings, where $D = \conn(T)$.  
\end{proposition}

\begin{proof}   
We can assume that the number of crossings of $T$ is $\ell(D)$, where $D = \conn(T)$.     We proceed by induction on the number of crossings.  If $T$ has no crossings, then $T$ is an element of $\mathcal B$,  because up to ambient isotopy,  there is a  unique simple layered tangle diagram with underlying Brauer diagram $D$.     Assume that   $\ell(D)$ is positive and that the statement holds for all simple layered tangle diagrams with fewer than $\ell(D)$  crossings.   There is a simple layered tangle diagram $S$ in $\mathcal B$ with $\conn(S) = D$.   By the previous lemma,  $T - S$ is a 
$\Z[\qbold-\qbold\inv]$--linear combination of simple layered tangle diagrams with fewer  than $\ell(D)$ crossings, and thus the result follows from the induction hypothesis.
\end{proof}

\begin{corollary} \label{corollary:  transition between MW bases 1}
The transition matrix between any two Morton--Wassermann  tangle bases of $\bmw n$  has entries in $\Z[\qbold - \qbold\inv]$.  
\end{corollary}

\begin{lemma}  \label{lemma:  transition between MW bases 2}
Let $\mathcal B$ be a Morton--Wassermann tangle basis of $\bmw n(R_0)$.  The matrix with respect to $\mathcal B$ of left or right multiplication by $g_i$ or $g_i\inv$  has entries in $\Z[\zbold^{\pm 1},  (\qbold - \qbold\inv)]$.
\end{lemma}

\begin{proof}     Let $T$ be an element of $\mathcal B$;  assume without loss of generality that the number of crossings of $T$ is $\ell(D)$ where $D$ denotes $\conn(T)$.    We have to show that $T g_i$  is in the 
$\Z[\zbold^{\pm 1},  (\qbold - \qbold\inv)]$--span of $\mathcal B$.    We proceed by induction on the number of crossings of $T$.    If $T$ has no crossings, then $T g_i$ is simple and layered, so the assertion follows from \hyperref[proposition:  transition between MW bases 1]{Proposition~\ref{proposition:  transition between MW bases 1}}.  

Assume that $\ell(D) > 0$ and that the assertion holds when $T$ is replaced by an element of $\mathcal B$ with fewer crossings.   If the vertices $\pbar i$ and $\pbar {i+1}$ of $T$ are connected by a strand, then
$T g_i = z\inv T$, so we are done.    Otherwise, let $s$ and $t$ denote the distinct strands of $T$  incident on the vertices $\pbar i$ and $\pbar {i+1}$.    Let $S$ be a simple layered tangle diagram such that
$\conn(S) = D$,  $S$ has $\ell(D)$ crossings, and     $S$ is layered with respect to an ordering $(t, s, \dots)$ of the strands.  
Then  $S g_i$ is simple and layered,  so is in the $\Z[\qbold - \qbold\inv]$--span  of $\mathcal B$, by \hyperref[proposition:  transition between MW bases 1]{Proposition~\ref*{proposition:  transition between MW bases 1}}.    Moreover $(T - S) g_i$  is in 
the $\Z[\zbold^{\pm 1},  (\qbold - \qbold\inv)]$--span of $\mathcal B$,  by combining  \hyperref[lemma:  transition between MW bases 0]{Lemma~\ref*{lemma:  transition between MW bases 0}}, \hyperref[proposition:  transition between MW bases 1]{Proposition~\ref* {proposition:  transition between MW bases 1}}, and the induction hypothesis.

The proof for right multiplication by $g_i\inv$ or by left multiplication by $g_i^{\pm 1}$ is similar.
\end{proof} 

\begin{remark}   \label{remark:  transition between MW bases 3}
   Let $T$ be a simple layered tangle diagram.  From the proof of \hyperref[lemma:  transition between MW bases 0]{Lemma~\ref*{lemma:  transition between MW bases 0}}  and \hyperref[proposition:  transition between MW bases 1]{Proposition~\ref*{proposition:  transition between MW bases 1}}, one sees that all the elements of the Morton--Wassermann basis $\mathcal B$ that figure in  the expansion of $T$ with respect to $\mathcal B$ are obtained by changing or smoothing various crossings of $T$.   Hence, if $T$ has a strand $s$ connecting two vertices $v_1, v_2$, such that $s$ has no crossings with any other strand, then all elements of $\mathcal B$ appearing in the expansion of $T$ also have a strand connecting $v_1$ and $v_2$.  Likewise, from the proof of \hyperref[lemma:  transition between MW bases 2]{Lemma~\ref*{lemma:  transition between MW bases 2}}, if $\{\pbar i, \pbar{i+1}\} \cap \{v_1, v_2\} = \emptyset$,   then all elements of $\mathcal B$ appearing in the expansion of $T g_i$ have a strand connecting connecting $v_1$ and $v_2$. 
\end{remark}

\subsubsection{The transition matrix from a tangle basis to the Murphy basis}   

 We examine the coefficients of the expansion of an element 
 \begin{equation}  \label{equation:  typical Murphy basis element for BMW}
d_\mathfrak{s}^*  c_{(\la, 0)}  e_{n-1}  \power \el d_\mathfrak{t}, 
\end{equation}
 of $\mathscr W_n$    with respect to a Morton--Wassermann tangle basis $\mathcal B$ of $\bmw n(R_0)$.

\begin{definition}  Let $k \le n$ and $m \le \lfloor k/2 \rfloor$.    A tangle diagram $T$ is {\em of type} $(k, m)$ if $T$ has strands connecting the adjacent pairs of bottom vertices
\begin{equation}  
(\pbar{k-2m + 1}, \pbar{k-2m +2}),  \dots,  (\pbar{k-1}, \pbar k)  \qquad  \text{($m$ strands)}. 
\end{equation}
\end{definition}

\begin{lemma} \label{lemma:  expansion of GE basis in MW basis 1} 
If $T$ is an element of $\mathcal B$ of type $(k, m)$,  and $(\la, \el) \to (\mu, m)$ is an edge in 
$\widehat{\mathcal H}$ from level $k-1$ to level $k$, then  $T \dd {(\la, \el)} {(\mu, m)} k$ is a 
 $\Z[\qbold^{\pm 1}, \zbold^{\pm 1}]$--linear combination of elements of $\mathcal B$ of type $(k-1, \el)$.  
\end{lemma}

\begin{proof}  There are two cases to consider.    

\textsc{Case 1},   $\el = m$  and $\la \subset \mu$.     Then for some $a \le k - 2 m$,
 $$\dd {(\la, \el)} {(\mu, m)} k =  g_{a, k-2m}  e_{k-2}\power{m}.$$  
 By \hyperref[lemma:  transition between MW bases 2]{Lemma~\ref*{lemma:  transition between MW bases 2}} and \hyperref[remark:  transition between MW bases 3]{Remark~\ref*{remark:  transition between MW bases 3}},    $T  g_{a, k-2m} $  is a $\Z[\qbold^{\pm 1}, \zbold^{\pm 1}]$--linear combination of elements of $\mathcal B$  of type $(k, m)$.   But for any element  $S$  of  $\mathcal B$  of type $(k, m)$,    $S e_{k-2} \power{m}$ is a simple layered tangle diagram of type 
  $(k-1, m)$,   see \hyperref[figure: Multiplication by e power fk]{Figure~\ref*{figure: Multiplication by e power fk}}.   Therefore by \hyperref[roposition:  transition between MW bases 1]{Proposition~\ref*{proposition:  transition between MW bases 1}}  and \hyperref[remark:  transition between MW bases 3]{Remark~\ref*{remark:  transition between MW bases 3}},   $S e_{k-2} \power{m}$  is a 
 $\Z[\qbold - \qbold\inv]$--linear combination of elements of $\mathcal B$ of type   $(k-1, m)$
  Taking into account that $\el = m$,  this gives the result. 
  
\begin{figure}
\begin{tikzpicture}[cross line/.style={preaction={draw=white,-,line width=6pt}}]
\foreach \x in {5.6,7.2,8.8,10.4} 
	\draw (\x,1.8) .. controls (\x+0.2,1.2) and (\x+0.6,1.2) .. (\x+0.8,1.8);
\foreach \x in {4.8,6.4,8.0} 
	\draw (\x,0.5) .. controls (\x+0.2,1.1) and (\x+0.6,1.1) .. (\x+0.8,0.5);
\foreach \x in {4.0,5.6,7.2} 
\draw (\x,0) .. controls (\x+0.2,-0.6) and (\x+0.6,-0.6) .. (\x+0.8,0);
\foreach \x in {4.0,5.6,7.2} 
\draw (\x,-1.3) .. controls (\x+0.2,-0.7) and (\x+0.6,-0.7) .. (\x+0.8,-1.3);
\foreach \x in {8.8} 
	\draw (\x,0) -- (\x,-1.3);
\foreach \x in {4.0,4.8,...,9.6} 
	\filldraw [black] (\x,0) circle (1.1pt);
\foreach \x in {5.6,6.4,...,11.2} 
	\filldraw [black] (\x,1.8) circle (1.1pt);\foreach \x in {4.0,4.8,...,9.6} 
	\filldraw [black] (\x,-1.3) circle (1.1pt);\foreach \x in {4.0,4.8,5.6,...,9.6} 
	\filldraw [black] (\x,0.5) circle (1.1pt);
\foreach \x in {4.0,4.8,5.6,...,9.6} 
 	\draw[dotted] (\x,0) -- (\x,0.5);
\draw (4.0,0.5) -- (3.8,1); \draw[dashed] (3.8,1) -- (3.7,1.25);
\end{tikzpicture}
\caption{}
\label{figure: Multiplication by e power fk}
\end{figure}

 \textsc{ Case 2},  $\el = m -1$   and $\mu \subset \la$.    Then  $\dd {(\la, \el)} {(\mu, m)} k$ is a sum of terms of the form
  $q^s g_{k+1 - 2 m, a}  e_{k-2}\power{m -1}$.    But  $T g_{k+1 - 2 m, a}  e_{k-2}\power{m -1}$ is a simple layered tangle diagram of type $(k-1, m -1)$, see \hyperref[figure: Multiplication by e power f k-1]{Figure~\ref*{figure: Multiplication by e power f k-1}}.   Therefore, again by 
  \hyperref[proposition:  transition between MW bases 1]{Proposition~\ref*{proposition:  transition between MW bases 1}}  and \hyperref[remark:  transition between MW bases 3]{Remark~\ref*{remark:  transition between MW bases 3}},   $T g_{k+1 - 2 m, a}  e_{k-2}\power{m -1}$  is a 
  $\Z[\qbold - \qbold\inv]$--linear combination of elements of $\mathcal B$ of type    $(k-1, m-1)$. Since $\el = m -1$,  this proves the result.
\end{proof}

\begin{figure}[h!]
\begin{tikzpicture}[cross line/.style={preaction={draw=white,-,line width=6pt}}]
\foreach \x in {5.6,7.2,8.8,10.4} 
	\draw (\x,1.8) .. controls (\x+0.2,1.2) and (\x+0.6,1.2) .. (\x+0.8,1.8);
\foreach \x in {4.8,6.4,8.0} 
	\draw (\x,0.5) .. controls (\x+0.2,1.1) and (\x+0.6,1.1) .. (\x+0.8,0.5);
\foreach \x in {5.6,7.2} 
\draw (\x,0) .. controls (\x+0.2,-0.6) and (\x+0.6,-0.6) .. (\x+0.8,0);
\foreach \x in {5.6,7.2} 
\draw (\x,-1.3) .. controls (\x+0.2,-0.7) and (\x+0.6,-0.7) .. (\x+0.8,-1.3);
\foreach \x in {0,0.8,1.6,8.8} 
	\draw (\x,0) -- (\x,-1.3);
\foreach \x in {2.4,3.2,4.0} 
	\draw (\x,0) -- (\x+0.8,-1.3);
\draw[cross line,-] (4.8,0) -- (2.4,-1.3);
\foreach \x in {0,0.8,1.6,...,9.6} 
	\filldraw [black] (\x,0) circle (1.1pt);
\foreach \x in {5.6,6.4,...,11.2} 
	\filldraw [black] (\x,1.8) circle (1.1pt);\foreach \x in {0,0.8,1.6,...,9.6} 
	\filldraw [black] (\x,-1.3) circle (1.1pt);\foreach \x in {4.8,5.6,...,9.6} 
	\filldraw [black] (\x,0.5) circle (1.1pt);
\foreach \x in {4.8,5.6,...,9.6} 
 	\draw[dotted] (\x,0) -- (\x,0.5);
\end{tikzpicture}
\caption{}
\label{figure: Multiplication by e power f k-1}
  \end{figure}

\begin{proposition} \label{proposition:  expansion of GE basis in MW basis}
$d_\mathfrak{s}^*  c_{(\la, 0)}  e_{n-1}  \power \el d_\mathfrak{t}$ is in the 
$\Z[\qbold^{\pm 1}, \zbold^{\pm 1}]$--span of $\mathcal B$.
\end{proposition}

\begin{proof}
Taking into account \hyperref[corollary:  transition between MW bases 1]{Corollary~\ref*{corollary:  transition between MW bases 1}},  we can assume without loss of generality  that the elements  $g_v   e_{n-1}  \power \el $  for $v \in \mathfrak{S}_{n-2f}$   are elements of $\mathcal B$,  as these are simple layered tangle diagrams (with distinct underlying Brauer diagrams).    Moreover,  $g_v  e_{n-1}  \power \el $ is of type $(n, \el)$.  
Thus $ c_{(\la, 0)}  e_{n-1}\power \el  = \sum_{v\in\mathfrak{S}_\la}q^{\el(v)}g_v e_{n-1}\power \el $ is in the 
$\Z[q]$--span of elements of elements of $\mathcal B$ of type $(n, \el)$.   

Let 
$$
\mfs = ((\la\power 0, \el_0),  (\la\power 1, \el_1), \dots,  (\la\power n, \el_n)),
$$
where  $(\la\power 0, \el_0) = (\emptyset, 0)$  and $(\la\power n, \el_n) =  (\la, \el)$.     Then 
$$
d_\mfs =  d\power n_{(\la\power {n-1}, \el_{n-1}) \to (\la\power n, \el_n)}
d\power {n-1}_{(\la\power {n-2}, \el_{n-2}) \to (\la\power {n-1}, \el_{n-1})}  \cdots
$$
By repeated use of \hyperref[lemma:  expansion of GE basis in MW basis 1]{Lemma~\ref*{lemma:  expansion of GE basis in MW basis 1}},   
$ c_{(\la, 0)}  e_{n-1}\power \el  d_\mfs$    is in the $\Z[\qbold^{\pm 1}, \zbold^{\pm 1}]$--span of elements of $\mathcal B$.  But   the expansion of $ (c_{(\la, 0)}  e_{n-1}\power \el  d_\mfs)^* = d_\mfs^*  c_{(\la, 0)}  e_{n-1}\power \el$  involves only elements of $\mathcal B$ of type $(n, \el)$.   By repeated application of \hyperref[lemma:  expansion of GE basis in MW basis 1]{Lemma~\ref*{lemma:  expansion of GE basis in MW basis 1}} once more,   $d_\mathfrak{s}^*  c_{(\la, 0)}  e_{n-1}  \power \el d_\mathfrak{t}$  is in the $\Z[\qbold^{\pm 1}, \zbold^{\pm 1}]$--span of $\mathcal B$.
\end{proof}

\subsubsection{The Murphy basis and the generic ground ring}
\label{subsubsection:  invertibility of transition matrix for BMW}
 Let $B$ denote the matrix of expansion coefficients of the elements of $\mathscr W_n$   with respect to some Morton--Wasserman tangle basis  $\mathcal B$  of  $\bmw n(R_0)$  (and some choice of ordering of  $\mathscr W_n$ and of $\mathcal B$.)      By \hyperref[proposition:  expansion of GE basis in MW basis]{Proposition~\ref*{proposition:  expansion of GE basis in MW basis}}, we know that the matrix $B$ has entries in $ \Z[\zbold^{\pm1}, \qbold^{\pm1}] \subset R_0$.
 On the other hand,  since $\mathscr W_n$ is a basis of the BMW algebra over $R = R_0[\deltabold\inv]$,  it follows that  $B$ is invertible over $R$.     We are going to show that $B$ is invertible over $\Z[\zbold^{\pm1}, \qbold^{\pm1}] $  and therefore
 $\mathscr W_n$ is a basis of $\bmw n$ over $R_0$.

The Brauer algebra $B_n$ over $\Z[\delta]$   is the specialization of $W_n(R_0)$  at $q = 1$ and $z = 1$. (See the following \hyperref[b-r-e]{Section~\ref*{b-r-e}} for details.)   Under the specialization,   the Morton--Wassermann basis  of $\bmw n(R_0)$  specializes to the usual diagram basis of the Brauer algebra, and $\mathscr W_n$ specializes to the corresponding collection of elements of the Brauer algebra, denoted $\mathscr B_n$. 
Moreover, the evaluation of $B$ at $q = 1$ and $z = 1$,  which we denote by $B_{\Z}$,   is the matrix of expansion coefficients of the elements of $\mathscr B_n$ with respect to the diagram basis of   the Brauer algebra.
Let $d$ denote the determinant of $B$ and $\bar d$ the determinant of $B_{\Z}$, which is the evaluation of $d$ at $q =1$ and $z=1$.  Since $B$ is a matrix over $\Z[\zbold^{\pm1}, \qbold^{\pm1}]$,    it follows that $B_{\Z}$ is a matrix over $\Z$, and hence $\bar d$ is an integer.

\begin{lemma}  \label{lemma:  transition matrix from diagram basis to M basis for Brauer}
$B_{\Z}$ is invertible over $\Z$.
\end{lemma}

\begin{proof}  Since $B$ is invertible over $R$,   it follows that $B_{\Z}$ is invertible over $\Z[\deltabold^{\pm 1}]$.   Equivalently,  $\bar d = \det(B_{\Z})$  is a unit in $\Z[\deltabold^{\pm 1}]$.  But $\bar d$ is an integer, so it follows that $\bar d = \pm 1$ and thus $B_{\Z}$ is invertible over $\Z$.  
\end{proof}

\begin{lemma}  $B$ is invertible over $R_0$.  
\end{lemma}

\begin{proof}  Since $B$ is invertible over $R$,  $d = \det(B)$ is a unit in $R$.
We can regard $R$ as a subring of 
$$\overline R = \Z[\zbold^{\pm1}, \qbold^{\pm1}, (\qbold- 1)\inv, (\qbold +1)\inv, (\zbold + \qbold)\inv, (\qbold \zbold -1)\inv], $$
see Equation~\eqref{equation:  solution for delta in field of fractions}.
    Since $d$ is an element of $\Z[\zbold^{\pm1}, \qbold^{\pm1}] \subseteq R_0$ which is a unit in
 in $\overline R$,  it has the form
$$
d = \pm \qbold^a \zbold^b  (\qbold - 1)^c  (\qbold + 1)^e (\zbold + \qbold)^f (\qbold \zbold -1)^g
$$
for some integers $a, b$ and some natural numbers $c, e, f, g$.    But the specialization of $d$ at $q = 1$ and $z = 1$ is equal to $\pm 1$ and therefore we must have $c = e = f = g = 0$.     Thus $d = \pm \qbold^a \zbold^b$ is a unit in $R_0$,  so $B$ is invertible over $R_0$. 
\end{proof}

The invertibility of $B$ over $R_0$ together with \hyperref[proposition:  Murphy basis of BMW over generic with delta inverse adjoined]{Proposition~\ref*{proposition:  Murphy basis of BMW over generic with delta inverse adjoined}} implies the following theorem:

\begin{theorem}  \label{theorem:  Murphy basis of BMW over generic ground ring}
Let $W_n$ denote the BMW algebra over the generic ground ring $R_0$.  The set
\begin{align}\label{b-m-b over generic ground ring}
\mathscr{W}_n=
\left\{
d_\mathfrak{s}^*c_{(\lambda,l)}d_\mathfrak{t}\,\big|\,
\text{$\mathfrak{s},\mathfrak{t}\in\widehat{W}_n^{(\lambda,l)}$, $(\lambda,l)\in\widehat{W}_n$}
\right\}
\end{align} 
is an $R_0$--basis of $W_n$, and $(W_n,*,\widehat{W}_n,\unrhd,\mathscr{W}_n)$ is a cell datum for $W_n$.
 \end{theorem}

\begin{remark}
The basis~\eqref{b-m-b} differs from the Murphy--type basis for the BMW algebras given in~\cite{MR2348099} by a triangular transformation.
\end{remark}

\begin{corollary}  \label{corollary:   cell filtrations for BMW over generic ground ring}
For  $n \ge 0$ and for  $\Delta$ a cell module of  $W_{n+1}$,    the restricted module $\Res^{W_{n+1}}_{W_n}(\Delta)$ has an order preserving cell filtration.
\end{corollary}

\begin{proof}  For  $k \ge 0$,  $(\la, \el) \in \widehat W_k$,  and  $\mft \in  \widehat{W}_{k}\power{\la, \el}$, let
$
m_\mft  =   (c_{(\la, \el)} + W_k^{\rhd (\la, \el)}) d_\mft.
$
Then   $\{m_\mft  : \mft \in \widehat{W}_k\power{\la, \el}\}$   is the basis of the cell module
$\cell W {k} {(\la, \el)}$  derived from the cellular bases $\mathscr W$.   The collection of these bases, as 
$k$ and  $(\la, \el)$ vary,  is a family of path bases, because the path basis condition holds over 
$R = R_0[\deltabold\inv]$, according to \hyperref[lemma: path basis condition for delta-d t basis]{Lemma~\ref*{lemma: path basis condition for delta-d t basis}}, 
and therefore it holds over $R_0$ as well.  It follows from 
\hyperref[lemma:  path bases and cell filtrations]{Lemma~\ref*{lemma:  path bases and cell filtrations}} that
restrictions of cell modules have an order preserving cell filtration.
\end{proof}

\subsection{Brauer algebras}\label{b-r-e}
The Brauer algebras  were defined by Brauer~\cite{MR1503378}. Wenzl~\cite{MR951511} showed that the Brauer algebras are obtained from the group algebra of the symmetric group by the Jones basic construction, and that the Brauer algebras over a field of characteristic zero are generically semisimple. Cellularity of the Brauer algebras was established by Graham and Lehrer~\cite{MR1376244}.

\ignore{
If $\rho_1,\rho_2$ are $(n,n)$--Brauer diagrams, then the composition of diagrams $\rho_1\circ\rho_2$ is the diagram obtained by placing $\rho_1$ above $\rho_2$, identifying each vertex in the bottom row of $\rho_1$ with the corresponding vertex in the top row of $\rho_2$ and deleting any connected components of the resulting diagram which contains only vertices from the middle row. 
}

Let $S$ be an integral domain with a distinguished element $\delta$.  The Brauer algebra
$B_n  = B_n(S; \delta)$ is the free $S$--module with basis the set of $(n, n)$--Brauer diagrams.
The product of two Brauer diagrams is obtained by stacking  them and then replacing each closed loop by a factor of $\delta$;  see ~\cite{MR1503378} or ~\cite{MR951511} for details.

\ignore{
We define a product of two Brauer diagrams, which is a multiple of another Brauer diagram,  as follows.  For 
$(n,n)$--Brauer diagrams $\rho_1,\rho_2$, define
\begin{align*}
\rho_1\rho_2=\delta^c \rho_1\circ \rho_2,
\end{align*}
where $c$ is the number of components removed from the middle row in constructing the composition $\rho_1\circ\rho_2$.  This product extends uniquely to a bilinear product on $B_n$, and it is easy to see that the product is associative.
}

\begin{definition}  \label{definition: Brauer algebra}
Let $S$ be an integral domain and $\delta\in S$.  The {Brauer algebra} $B_n=B_n(S;\delta)$ is the 
free $S$--module with basis the set of $(n, n)$--Brauer diagrams, with bilinear product determined by the multiplication of Brauer diagrams.  By convention, $B_0(S; \delta) = S$.  
\end{definition}

The involution $*$ on $(n, n)$--Brauer diagrams which reflects a diagram in the axis $y = 1/2$
extends linearly to an algebra involution of $B_n(S; \delta)$.
Note that the Brauer diagrams with only vertical strands are in
bijection with permutations of $\{1, \dots, n\}$, and that the
multiplication of two such diagrams coincides with the multiplication of
permutations.  Thus the  Brauer algebra contains the group algebra $S\mf S_n$ of
the permutation group $\mathfrak S_n$ as a unital subalgebra.   The identity element of the Brauer algebra is the diagram corresponding to the trivial permutation.  We will note below that $S \mf S_n$ is also a quotient of $B_n(S; \delta)$.  

Let $s_i$ and $e_i$ denote the following $(n, n)$--Brauer diagrams:
\begin{align*}
\begin{matrix}
\begin{tikzpicture}
\node at (-7.4,-0.5) {$\displaystyle {s_i}= \color{black}$};
\node at (-6.3,-0.4) {$\displaystyle \cdots \color{black}$};
\draw (-6.8,0) -- (-6.8,-0.8);
\draw (-5.8,0) -- (-5.8,-0.8);
\draw (-5,0) -- (-4.2,-0.8);
\draw (-4.2,0) -- (-5,-0.8);
\draw (-3.4,0) -- (-3.4,-0.8);
\draw (-2.4,0) -- (-2.4,-0.8);
\filldraw [black] (-6.8,0) circle (1.2pt);
\filldraw [black] (-6.8,-0.8) circle (1.2pt);
\filldraw [black] (-5.8,0) circle (1.2pt);
\filldraw [black] (-5.8,-0.8) circle (1.2pt);
\filldraw [black] (-5.0,0) circle (1.2pt);
\filldraw [black] (-5.0,-0.8) circle (1.2pt);
\filldraw [black] (-4.2,0) circle (1.2pt);
\filldraw [black] (-4.2,-0.8) circle (1.2pt);
\filldraw [black] (-3.4,0) circle (1.2pt);
\filldraw [black] (-3.4,-0.8) circle (1.2pt);
\filldraw [black] (-2.4,0) circle (1.2pt);
\filldraw [black] (-2.4,-0.8) circle (1.2pt);
\node at (-2.9,-0.4) {$\displaystyle \cdots \color{black}$};
\node at (-5.0,-1.1) {$\displaystyle i \color{black}$};
\node at (-4.1,-1.1) {$\displaystyle i+1 \color{black}$};
\node at (-1.4,-0.4) {$\displaystyle \text{and} \color{black}$};
\node at (0.0,-0.5) {$\displaystyle e_i= \color{black}$};
\node at (1.1,-0.4) {$\displaystyle \cdots \color{black}$};
\draw (0.6,0) -- (0.6,-0.8);
\draw (1.6,0) -- (1.6,-0.8);
\draw (2.4,0)  .. controls (2.6,-0.2) and (3.0,-0.2) .. (3.2,0);
\draw (2.4,-0.8)  .. controls (2.6,-0.6) and (3.0,-0.6) .. (3.2,-0.8);
\draw (4.0,0) -- (4.0,-0.8);
\draw (5.0,0) -- (5.0,-0.8);
\filldraw [black] (0.6,0) circle (1.2pt);
\filldraw [black] (0.6,-0.8) circle (1.2pt);
\filldraw [black] (1.6,0) circle (1.2pt);
\filldraw [black] (1.6,-0.8) circle (1.2pt);
\filldraw [black] (2.4,0) circle (1.2pt);
\filldraw [black] (2.4,-0.8) circle (1.2pt);
\filldraw [black] (3.2,0) circle (1.2pt);
\filldraw [black] (3.2,-0.8) circle (1.2pt);
\filldraw [black] (4,0) circle (1.2pt);
\filldraw [black] (4,-0.8) circle (1.2pt);
\filldraw [black] (5,0) circle (1.2pt);
\filldraw [black] (5,-0.8) circle (1.2pt);
\node at (4.5,-0.4) {$\displaystyle \cdots \color{black}$};
\node at (2.4,-1.1) {$\displaystyle i \color{black}$};
\node at (3.3,-1.1) {$\displaystyle i+1 \color{black}$};
\end{tikzpicture}
\end{matrix}
\end{align*}
It is easy to see that $e_1, \dots, e_{n-1}$ and $s_1, \dots, s_{n-1}$ generate $B_n(S; \delta)$ as an algebra.   We have $e_i^2 = \delta e_i$, so that $e_i$ is an essential idempotent if $\delta \ne 0$ and nilpotent otherwise.  Note that $e_i^* = e_i$ and $s_i^* = s_i$.  

The products $a b$ and $b a$  of two Brauer diagrams have  at most as many through strands as $a$.  Consequently, the span of diagrams with fewer than $n$ through strands is an ideal 
$J_n$ in $B_n(S; \delta)$.  
 The
ideal  $J_n$  is generated by $e_{n-1}$.  We have   $B_n(S; \delta)/J_n \cong S\mf S_n$, as algebras with involutions; in fact, the isomorphism is determined  by $v + J_n \mapsto v$, for $v \in \mf S_n$.

Morton and Wassermann show ~\cite{Morton-Wassermann} that 
$B_n(S; \delta)$ is a specialisation of the BMW algebra \break $W_n(S;q,z,\delta)$ at $q=1$ and $z=1$.  Consequently,  
$B_n(S; \delta)$  has a presentation by generators ${s_i}$  and $e_i$ ($1 \le i \le n-1$) and relations specializing those of the BMW algebra.
\ignore{
\begin{enumerate}
\item (Inverses) \hods ${s_i}^2 = 1$.
\item (Essential idempotent relation)\hods $e_i^2 = \delta e_i$.
\item (Braid relations) \hods ${s_i} {s_{i+1}}{s_i} = {s_{i+1}} {s_i} {s_{i+1}}$ 
and ${s_i} {s_j} ={s_j} {s_i}$ if $|i-j|  \ge 2$.
\item (Commutation relations)  \hods ${s_i} e_j = e_j {s_i}$  and
$e_i e_j = e_j e_i$  if $|i-j|\ge 2$. 
\item (Tangle relations)\hods $e_i e_{i\pm 1} e_i = e_i$, $s_i
s_{i\pm 1} e_i = e_{i\pm 1} e_i$, and $ e_i  {s_{i\pm 1}} s_i=   e_ie_{i\pm 1}$.
\item (Untwisting relations)\hods $s_i e_i = e_i {s_i} = e_i$,
and $e_i s_{i \pm 1} e_i = e_i$.
\end{enumerate}
}

\subsubsection{The Murphy basis}
The generic ground ring for the Brauer algebras is  $R_0 = \Z[\deltabold]$,  where $\deltabold$ is an indeterminant.   Write $R = \Z[\deltabold^{\pm1}]$,  and       write $B_n(R) = B_n(R; \deltabold)$. 

For $n \ge 0$ write   $H_n=R\mathfrak{S}_n$. Specialising the cellular basis for $\mathcal{H}_n(q^2)$ given in \hyperref[m-b]{Theorem~\ref*{m-b}} at $q=1$ gives a cellular basis for $H_n$.   As for the Hecke algebras,
$\widehat H_n$ is the set  $\mathcal Y_n$ of Young diagrams of size $n$, and the branching diagram for the tower
$(H_n)_{n \ge 0}$ of symmetric group algebras is Young's lattice.

 It is shown in~\cite[Sect.~5.2]{MR2794027} that the pair of towers  $(B_n(R))_{n \ge 0}$  and $(H_n)_{n \ge 0}$ satisfy the framework axioms~\eqref{axiom: involution on An}--\eqref{axiom: semisimplicity}  of \hyperref[subsection:  cellularity and jones setting and correction]{Section~\ref*{subsection:  cellularity and jones setting and correction}}.  Axiom~\eqref{axiom Hn coherent}  holds by \hyperref[corollary:  Hecke algebras strongly coherent]{Corollary~\ref*{corollary:  Hecke algebras strongly coherent}}, and specialization from the Hecke algebras to the symmetric group algebras.  Axiom~\eqref{axiom:  Delta J} hold for $B_n(R)$,  by the remarks at the end of \hyperref[subsection:  cellularity and jones setting and correction]{Section~\ref*{subsection:  cellularity and jones setting and correction}}.
 Moreover, by 
 \hyperref[corollary Hecke algebras cyclic cellular]{Corollary~\ref*{corollary Hecke algebras cyclic cellular}},  the symmetric group algebras are cyclic cellular, so Axiom~\eqref{axiom: Hn cyclic cellular}  is satisfied as well.  Therefore, by \hyperref[theorem: jones tower coherent tower cyclic cellular]{Theorem~\ref*{theorem: jones tower coherent tower cyclic cellular}}, the tower of algebras $(B_n(R))_{n \ge 0}$ is a strongly coherent tower of cyclic cellular algebras.

By the discussion in \hyperref[subsection results on cellularity from Goodman Graber]{Section~\ref*{subsection results on cellularity from Goodman Graber}}, the partially ordered set $\widehat{B}_n$ in the cell datum for $B_n$ can be realized as
\begin{align*}
\widehat{B}_n=\left\{ 
(\lambda,l)\,\big|\,  0 \le \el \le \lfloor n/2\rfloor \text{ and } 
\lambda\in\widehat{{H}}_{n-2l}
\right\}.
\end{align*}
The order relation on $\widehat B_n$, and the branching rule for the branching diagram $\widehat B$ for the tower $(B_n)_{n \ge 0}$ is exactly the same as for the BMW algebras discussed in the previous section. 

For each $n \ge 0$ for for each $\mu \in \widehat{ H}_n$, define
$
{c}_{(\mu,0)}=\sum_{v\in\mathfrak{S}_\mu} v$;
thus ${c}_{(\mu,0)}$ is a preimage in $B_n$ of $m_\mu \in  H_n$ (defined in Equation~\eqref{equation:  element m mu in Hecke algebra}).
For $n \ge 2$ and $(\mu,m)\in\widehat{B}_n$, let
$
c_{(\mu,m)}={c}_{(\mu,0)}e_{n-1}^{(m)}$,
where $e_{i-1}^{(m)}$ is defined in Equation~\eqref{notation e sub i super ell}.

 For $1 \le i \le j$ let
 \begin{equation} \label{definition of s i j}
 s_{i, j} = s_is_{i+1}\cdots s_{j-1} = (j, j-1, \dots, i),
 \end{equation}
 and let $s_{j, i} = s_{i, j}\inv$. 
 
 Let $i \ge 1$ and $\lambda\in\widehat{\mathcal{H}}_{i-1}$ and $\mu\in\widehat{\mathcal{H}}_{i}$, with $\la \to \mu$  in $\widehat{\mathcal H}$.  
If $\mu=\lambda\cup\{(r,\mu_r)\}$, let $a=\sum_{j=1}^{r}\mu_j$, and define
\begin{align}
\bar{u}_{\lambda\to\mu}^{\,(i)}=s_{i,a}
\sum_{k=0}^{\lambda_r}  s_{a,a-k}\quad \text{and} \quad
\bar{d}_{\lambda\to\mu}^{\,(i)}=s_{a,i}.
\end{align}
These are liftings in $B_i$ of the branching factors in the symmetric group algebra $ H_i$, as determined in \hyperref[corollary:  d branching factors for the Hecke algebra]{Corollary~\ref*{corollary:  d branching factors for the Hecke algebra}} and \hyperref[corollary: u branching factors for the Hecke algebra]{Corollary~\ref*{corollary: u branching factors for the Hecke algebra}}. 

\ignore{
By \hyperref[theorem:  closed form determination of the branching factors]{Theorem~\ref*{theorem:  closed form determination of the branching factors}}, the branching factors for the tower $(B_n)_{n \ge 0}$ can be chosen as follows:
If $(\lambda,l)\in\widehat{B}_i$, $(\mu,l)\in\widehat{B}_{i+1}$, and $(\lambda,l)\to(\mu,l)$ in $\widehat{B}$, define 
\begin{align*}
d_{(\lambda,l)\to(\mu,l)}^{(i+1)}= \bar{d}_{\lambda\to\mu}^{\,(i-2l+1)}e_{i-1}^{(l)}
\end{align*}
and for $(\lambda,l)\in\widehat{B}_i$, $(\mu,l+1)\in\widehat{B}_{i+1}$, such that $(\lambda,l)\to(\mu,l+1)$ in $\widehat{B}$, define 
\begin{align*} d_{(\lambda,l)\to(\mu,l+1)}^{(i+1)}=\bar{u}_{\mu\to\lambda}^{(i-2l)}e_{i-1}^{(l)}.
\end{align*}
For   $(\la, \el) \in \widehat{B}_i$  and a path
$$\mathfrak{t}=((\emptyset,0),(\lambda^{(1)},l_1),\ldots,(\lambda^{(i)},l_{i}) = (\la, \el) )\in\widehat{B}_{i}^{(\lambda,l)},$$
 let 
\begin{align*}
d_\mathfrak{t}=
d_{(\lambda^{(i-1)},l_{i-1})\to (\lambda^{(i)},l_{i})}^{(i)}
d_{(\lambda^{(i-2)},l_{i-2})\to (\lambda^{(i-1)},l_{i-1})}^{(i-1)}
\cdots d_{(\emptyset,0)\to(\lambda^{(1)},l_{1})}^{(1)}.
\end{align*}
} 

For $(\la, \el) \in \widehat B_{i}$ and $(\mu, m) \in \widehat B_{i+1}$ with $(\la, \el) \to (\mu, m)$, determine the branching factors $\dd {(\la, \el)} {(\mu, m)} {i+1}$  and  $\uu {(\la, \el)} {(\mu, m)} {i+1}$  according to the formulas of \hyperref[theorem:  closed form determination of the branching factors]{Theorem~\ref*{theorem:  closed form determination of the branching factors}};
for a path $\mft \in \widehat W_n \power {\la, \el}$,  define $d_\mft$ to be the ordered product of these branching factors along the path $\mft$, as in Equation\eqref{equation: ordered product of d coefficients}. 
From  \hyperref[label:1]{Corollary~\ref*{label:1}} \ignore{and \hyperref[theorem:  closed form determination of the branching factors]{Theorem~\ref*{theorem:  closed form determination of the branching factors}} }
we obtain:

\begin{theorem}  Let $B_n$ denote the Brauer algebra over the generic ground ring $R_0 = \Z[\deltabold]$.
 For $n \ge 0$,  
 the set
\begin{align}\label{b-m-a}
\mathscr{B}_n=
\left\{
d_\mathfrak{s}^*c_{(\lambda,l)}d_\mathfrak{t}\,\big|\,
\text{$\mathfrak{s},\mathfrak{t}\in\widehat{B}_n^{(\lambda,l)}$, $(\lambda,l)\in\widehat{B}_n$}
\right\},
\end{align}
is an $R_0$--basis for $B_n$, and $(B_n, *,\widehat{B}_n,\unrhd,\mathscr{B}_n)$ is a cell datum for $B_n$. 
\end{theorem}

\begin{proof}   Let $R = \Z[\deltabold^{\pm 1}]$.   From the preceding discussion and  \hyperref[label:1]{Corollary~\ref*{label:1}}, we have that  $\mathscr B_n$  is a cellular basis of $B_n(R; \deltabold)$.    In \hyperref[subsubsection:  invertibility of transition matrix for BMW]{Section~\ref*{subsubsection:  invertibility of transition matrix for BMW}},  we have shown that the transition matrix $B_{\Z}$  from the diagram basis of the Brauer algebra to  $\mathscr B_n$  is integer valued and invertible over $\Z$.  It follows that $\mathscr B_n$  is a basis of the Brauer algebra $B_n$ over the generic ground ring  $R_0$.  
\end{proof}

\begin{corollary}  \label{corollary:   cell filtrations for Brauer over generic ground ring}
For  $n \ge 0$ and for  $\Delta$ a cell module of  $B_{n+1}$,    the restricted module $\Res^{B_{n+1}}_{B_n}(\Delta)$ has an order preserving cell filtration.
\end{corollary}

\begin{proof}  The proof is the same as that of  \hyperref[corollary:   cell filtrations for BMW over generic ground ring]{Corollary~\ref*{corollary:   cell filtrations for BMW over generic ground ring}}.
\end{proof}

\begin{remark}
The basis~\eqref{b-m-a} coincides with the Murphy--type basis for $B_n(\delta)$ given in~\cite{MR2348099}.
\end{remark}

\subsection{Jones--Temperley--Lieb algebras} \label{subsection: temperley lieb algebras}

The Jones--Temperley--Lieb algebras were defined by Jones~\cite{MR696688}, and were used to define the Jones link invariant in~\cite{MR766964}. The cellularity of Jones--Temperley--Lieb algebras was established by Graham and Lehrer~\cite{MR1376244}. H\"arterich~\cite{MR1680384} has given Murphy bases for generalised Temperley--Lieb algebras.

Let $S$ be an integral domain and $\delta \in S$.   The Jones--Temperley--Lieb algebra $A_n={A}_n(S; \delta)$ is the unital $S$--algebra presented by the generators $e_1,\ldots,e_{n-1}$ and the relations
$e_i e_{i\pm 1}  e_i = e_i$,   $e_i e_j = e_j e_i$  if $|i-j |\ge 2$,  and $e_i^2 = \delta e_i$.
\ignore{
\begin{align*}
&e_ie_{i-1}e_i=e_i,&&i=2,\ldots,n-1,\\
&e_{i-1}e_{i}e_{i-1}=e_{i-1},&&i=2,\ldots,n-1,\\
&e_i^2=\delta e_i,&&i=1,\ldots,n-1,\\
&e_ie_j=e_je_i,&&|i-j|\ge2.
\end{align*}
}
The Jones--Temperley--Lieb algebra can also be realised as the subalgebra of the Brauer algebra, with parameter $\delta$, spanned by Brauer diagrams \emph{without crossings}. 
Because of the symmetry of the relations the assignment $e_i \mapsto e_i$  determines an involution $*$ of $A_n$.    The span of diagrams with at least one horizontal strand (that is, all diagrams other than the identity diagram)  is an ideal $J_n$;  it is the ideal generated by $e_{n-1}$.  
The map
$A_n/J_n  \rightarrow S$  determined  by $1_{A_n}+ J_n \mapsto 1_{S}$
\ignore{
\begin{align*}
{A}_n/({A}_ne_{n-1}{A}_n)&\rightarrow S\\
1_{A_n}+(e_{n-1})&\mapsto 1_{S},
\end{align*}
}
is an isomorphism of algebras with involution.

The generic ground ring for the Jones--Temperley--Lieb algebras is $R_0 = \Z[\deltabold]$,  where 
$\deltabold$ is an indeterminant over $\Z$.  Set $R = \Z[\deltabold^{\pm 1}]$.   Write $A_n(R) = A_n(R; \deltabold)$, and $H_n = R$ for $n \ge 0$.

\subsubsection{The Murphy basis}
 It is shown in~\cite[Sect.~5.3]{MR2794027} that the pair of towers  $(A_n(R))_{n \ge 0}$  and $(H_n)_{n \ge 0}$  satisfies the framework axioms~\eqref{axiom: involution on An}--\eqref{axiom: semisimplicity}  of \hyperref[subsection:  cellularity and jones setting and correction]{Section~\ref*{subsection:  cellularity and jones setting and correction}}.     Axioms~\eqref{axiom Hn coherent}  and~\eqref{axiom: Hn cyclic cellular}  are evident since $H_n = R$ for all $n$.  Axiom~\eqref{axiom:  Delta J} hold for $A_n(R)$,  by the remarks at the end of \hyperref[subsection:  cellularity and jones setting and correction]{Section~\ref*{subsection:  cellularity and jones setting and correction}}.   Therefore, by \hyperref[theorem: jones tower coherent tower cyclic cellular]{Theorem~\ref*{theorem: jones tower coherent tower cyclic cellular}}, the tower of algebras $(A_n(R))_{n \ge 0}$ is a strongly coherent tower of cyclic cellular algebras.

For each $n \ge 0$,  the partially ordered set $\widehat H_n$ in the cell datum for $H_n$ is a singleton
which we label as  $\{n\}$, and the branching diagram for the tower $(H_n)_{n \ge 0}$ is 
$\emptyset = 0 \to 1 \to 2 \to \cdots.$  
 The branching diagram $\hat A$  for the tower $(A_n)_{n\ge 0}$ is that obtained by reflections from $\widehat H$.  It can be realized as follows:   For $n \ge 0$, let
\begin{align*}
\hat{{A}}_n=
\big\{
j \ \big \vert \  0\le j\le n\text{ and } n-j\text{ is even}
\big\}
\end{align*}
and order $\hat{{A}}_n$ by writing $m\unrhd l$ if $l\ge m$ as integers. The branching diagram $\hat{{A}}$ has an edge connecting $j$ on level $n$   and   $k$  on level $n+1$ if and only if $|j - k| = 1$.  
\ignore{
\begin{enumerate}[label=(\arabic{*}), ref=\arabic{*},leftmargin=0pt,itemindent=1.5em]
\item vertices on level $i$: $\hat{A}_i$ and 
\item an edge $l\to m$ if $|l-m|=1$. 
\end{enumerate}
The first few levels of $\hat{{A}}$ are given in~\eqref{t-i-l}.
\begin{align}\label{t-i-l}
\begin{matrix}
\xymatrix{
0&\emptyset\ar[d]\\
 &1 \ar[d]\ar[dr] \\
2&0\ar[d]&2\ar[d]\ar[dl]\\
 &1\ar[d]\ar[dr]&3\ar[d]\ar[dr]\\
4&0\ar[d]&2\ar[dl]\ar[d] &4\ar[dl]\ar[d]\\
 &1 &3&5\\}
\end{matrix}
\end{align}
\ignore{
For $i,l=2,3,\ldots,$ let
\begin{align*}
e_{i-1}^{(l)}=\underbrace{e_{i-2l+1}e_{i-2l+3}\cdots e_{i-1}}_{\text{$l$ factors}}
&&\text{if $l=0,1,\ldots,\lfloor i/2\rfloor$,}
\end{align*}
and write
\begin{align*}
e_{i-1}^{(l)}=0&&\text{if $l>\lfloor i/2\rfloor$.}
\end{align*}
}
}

Evidently, the algebra $H_n = R$ has the cellular basis $\{1\}$.  We can choose the element 
$c_n$ in $H_n$  (see \hyperref[lemma: equivalent conditions for cyclic cellular algebra]{Lemma~\ref*{lemma: equivalent conditions for cyclic cellular algebra}}) to be $1$  and also all the branching factors
$d_{(n-1) \to n} \power n$ and $u_{(n-1)\to n}\power n$ to be $1$.  
According to Equation~\eqref{equation:  formula for c lambda ell},  for 
 $j\in\hat{A}_n$, 
we can take  
$$c_j=e_{n-1}^{(\el)}, \qquad \text{where }  \el=(n-j)/2,$$
and $e_{n-1}^{(\el)}$ is defined in Equation~\eqref{notation e sub i super ell}. 
By \hyperref[theorem:  closed form determination of the branching factors]{Theorem~\ref*{theorem:  closed form determination of the branching factors}}, the branching factors for the tower $(A_n)_{n \ge 0}$ can be chosen as follows:
If $j \in \hat A_i$ and $k \in \hat A_{i+1}$ with $j \to k$, we  take
\begin{align*}
d_{j\to k}^{(i+1)}=e_{i-1}^{(\el)},\qquad   \text{where }  \el=(i-j)/2.
\end{align*}
For a path $\mft \in \hat A_n \power {\la, \el}$,  define $d_\mft$ to be the ordered product of these branching factors along the path $\mft$, as in Equation\eqref{equation: ordered product of d coefficients}. 
\ignore{
For $\mathfrak{t}=(0,1,l_2,\ldots,l_{i})\in\hat{{A}}_{i}^{l_i}$, let 
\begin{align*}
d_\mathfrak{t}=
d_{l_{i-1}\to l_i}^{(i)}
d_{l_{i-2}\to l_{i-1}}^{(i-1)}
\cdots 
d_{0\to 1}^{(1)}.
\end{align*}
}
From \hyperref[label:1]{Corollary~\ref*{label:1}}  we obtain:
\begin{proposition} \label{proposition:  Murphy basis for TL over Laurent polynomial ring}
 Let $R = \Z[\deltabold^{\pm 1}]$  and let $A_n(R) = A_n(R; \deltabold)$  denote the Jones--Temperley--Lieb algebra over $R$.
For $n \ge 0$,  the set
\begin{align*}
\mathscr{A}_n=
\left\{
d_\mathfrak{s}^*c_{l}d_\mathfrak{t}\ \big|\ 
\text{$\mathfrak{s},\mathfrak{t}\in\hat{{A}}_n^{l}$\text{ and } $l\in\hat{{A}}_n$}
\right\},
\end{align*}
is an $R$--basis for ${A}_n$, and $({A}_n,*,\hat{{A}}_n,\unrhd,\mathscr{A}_n)$ is a cell datum for ${A}_n$. 
\end{proposition}
\ignore{
\begin{remark}
In the literature, it is usual to write
\begin{align*}
\hat{{A}}_i=\{\lambda\mid
\text{$\lambda$ is a partition of $i$ with at most two non--zero parts}
\}
\end{align*}
so that the branching diagram~\eqref{t-i-l} is represented as 
\begin{align*}
\begin{matrix}
\xymatrix{
0&\emptyset\ar[d]\\
 &\text{\tiny\Yvcentermath1$\yng(1)$} \ar[d]\ar[dr] \\
2&\text{\tiny\Yvcentermath1$\yng(1,1)$}\ar[d]&\text{\tiny\Yvcentermath1$\yng(2)$}\ar[d]\ar[dl]\\
 &\text{\tiny\Yvcentermath1$\yng(2,1)$}\ar[d]\ar[dr]&\text{\tiny\Yvcentermath1$\yng(3)$}\ar[d]\ar[dr]\\
4&\text{\tiny\Yvcentermath1$\yng(2,2)$}\ar[d]&\text{\tiny\Yvcentermath1$\yng(3,1)$}\ar[dl]\ar[d] &\text{\tiny\Yvcentermath1$\yng(4)$}\ar[dl]\ar[d]\\
 &\text{\tiny\Yvcentermath1$\yng(3,2)$} &\text{\tiny\Yvcentermath1$\yng(4,1)$}&\text{\tiny\Yvcentermath1$\yng(5)$}\\}
\end{matrix}
\end{align*}
\end{remark}
}

\subsubsection{The Murphy basis coincides with the diagram basis} 
\label{subsubsction:  Murphy and diagram basis coincide for TL}
Next, we will show that the Murphy type cellular basis $\mathscr A_n$ of $A_n$   given in \hyperref[proposition:  Murphy basis for TL over Laurent polynomial ring]{Proposition~\ref*{proposition:  Murphy basis for TL over Laurent polynomial ring}}   actually coincides with the diagram basis, so is in particular a basis for the Jones--Temperley--Lieb algebra over the generic ground ring $\Z[\deltabold]$.

Let $S$ be an integral domain and $\delta \in S$. 
Let $k$ and $n$ be non-negative integers of the same parity.   A {\em  $(k, n)$--Temperley--Lieb diagram} is a 
planar diagram with $k$ upper vertices and  $n$ lower vertices connected in pairs with no crossings.   The product of a $(k, n)$--TL diagram and an $(n, m)$--TL diagram is defined by the same rule as the product of two ordinary TL diagrams of the same size;  the result is a power of $\delta$ times a $(k, m)$--TL diagram.
The {\em Temperley-Lieb category}  is  category whose objects are non--negative integers;   if $n - k$ is odd, then $\Hom(k, n) = 0$,  and if $n - k$  is even then $\Hom(k, n)$  is the free $S$--module on the basis of $(k, n)$--TL diagrams.  Composition of morphisms is the bilinear extension of the  product of diagrams described above.    There is a map $*$  from $(k, n)$--TL diagrams to $(n, k)$--TL diagrams defined by reflection in a horizontal line.  The linear extension of $*$ is a contravariant functor from the TL category to itself with $*\circ* = \id$.    The {\em rank}  of a $(m, n)$--TL diagram is the number of its vertical strands.

Fix $n \ge 0$.   A TL   $n$--{\em dangle} of rank $k$ is  a $(k, n)$--TL diagram with $k$ vertical strands and $(n-k)/2$  horizontal strands.       Any $(n, n)$--TL diagram  $T$ of rank $k$  can be written uniquely as $T = y^* x$,  where  $x$ and $y$ are  $n$--dangles of rank $k$. 
A {\em Dyck sequence}  of length $n$ and rank $k$  is a sequence $(a_1, \dots, a_n)$  such that $a_i \in \{\pm 1\}$,   each partial sum $\sum_{i = 1}^j a_i$ is non--negative,  and $\sum_{i = 1}^n a_i = k$.   
There is a bijection between Dyck sequences of length $n$ and rank $k$,   and $n$--dangles of rank $k$, given as follows.     Given a Dyck sequence $(a_i)$ of length $n$ and rank $k$,  there is a unique $n$--dangle  $x$ of rank $k$ with the following property:   a  vertex $\pbar j$  is the right endpoint of a horizontal strand of $x$  if and only if  $a_j = -1$. 
    Conversely,  given an $n$--dangle $x$ of rank $k$,  label the right endpoint of each horizontal strand with $-1$   and all other bottom vertices with $+1$.  Then the resulting sequence  of labels in $\{\pm 1\}$,  read from left to right, is a Dyck sequence of rank $k$.    The two maps, from Dyck sequences to dangles and from dangles to Dyck sequences, are inverses.

There is a bijection between paths on the generic branching diagram for the Temperley--Lieb algebras,  of length $n$, from $\emptyset$  to $k$,   
 and Dyck sequences of length $n$ and rank $k$.   A path is given by a sequence $(0 = b_0,  1= b_1, b_2,  \dots,   k = b_n)$  with $b_j - b_{j-1} = \pm 1$ for each $j$.    Then  the sequence $(b_i - b_{i -1})_{i = 1}^n$  is a Dyck sequence of length $n$ and rank $k$.   Conversely, given a Dyck sequence of length $n$ and rank $k$,  its sequence of partial sums defines a path on the branching diagram,  of length $n$, from $\emptyset$  to $k$.
    Evidently, the two maps,  from paths to Dyck sequences and from Dyck sequence to paths, are inverses.

Composing the two bijections described above, we have a bijection between paths on the branching diagram and dangles.
For a path $\mft$ on the branching diagram, let $x(\mft)$ denote the corresponding dangle.

\ignore{
Let $R = \Z[\deltabold]$,  where $\deltabold$ is an indeterminant over $\Z$.  Let $A_n$ denote the $n$--th Temperley--Lieb algebra over $R$.  
We have defined elements   $d_\mathfrak{s}^*c_{k}d_\mathfrak{t}$   in $A_n$,  where $k \in \hat A_n$ and $\mfs, \mft$ are in $\hat A_n^k$,  namely paths of length $n$ on the generic branching diagram for the TL algebras from from $\emptyset$ to  $k $.  
} 

\begin{theorem}   Fix $n$ and $k \le n$ with $n - k$ even.    Let $\mfs$ and $\mft$ be elements of $\hat A_n^k$.  Then

 $$d_\mathfrak{s}^*c_{k}d_\mathfrak{t}  = x(\mfs)^*  x(\mft).
 $$
Thus  the Murphy type basis
\begin{align*}
\mathscr{A}_n=
\left\{
d_\mathfrak{s}^*c_{k}d_\mathfrak{t}\ \big|\ 
\text{$\mathfrak{s},\mathfrak{t}\in\hat{{A}}_n^{k}$\text{ and } $k\in\hat{{A}}_n$}
\right\},
\end{align*}
is just the set of all Temperley--Lieb diagrams on $2n$ vertices,  and in particular is a cellular basis of the Jones--Temperley--Lieb algebra $A_n$  over the generic ground ring $\Z[\deltabold]$.   
\end{theorem}

\begin{proof}
Recall that $c_k =  e_{n-1}\power {\el}$,  where $\el = (n - k)/2$.       Let $x_{n-1}\power \el$ be the bottom half of $ e_{n-1}\power \el$,  namely the $n$-dangle of rank $k$  with horizontal strands 
connecting the adjacent pairs of vertices 
$$ 
(\pbar{k+ 1}, \pbar{k +2}),  \dots,  (\pbar{n-1}, \pbar n)  \qquad  \text{($\el$ strands)}.  
$$
Thus  $ e_{n-1}\power \el = ( x_{n-1}\power \el)^*    x_{n-1}\power \el$.    To prove the proposition it suffices to show that
\begin{equation}  \label{equation:  TL diagram basis and path basis 1}
 x_{n-1}\power \el d_\mft  = x(\mft).
\end{equation}
We do this by induction on $n$,  the case $n = 1$ being evident.    Assume that the assertion holds for some fixed $n$,  for all $k$  with $k \le n$ and $n-k$ even, and for all $\mft \in \hat A_n^k$.    Let $\mfs \in \hat A_{n+1}^j$ for some $j$, 
$$
\mfs = (k_0, k_1, \dots,  k_n = k,    k_{n+1} = j),
$$
and let $\mft$ be the truncation of $\mfs$ of length $n$,
$$
\mft = (k_0, k_1, \dots,  k_n = k ).
$$
Write $\el = (n-k)/2$  and $\el' = (n+1 - j)/2$.    There are two cases:

\begin{figure}
\begin{tikzpicture}
\foreach \x in {6.4,8.0,9.6,11.2} 
	\foreach \y in {2.3}
		\draw (\x,\y) .. controls (\x+0.2,\y+0.6) and (\x+0.6,\y+0.6) .. (\x+0.8,\y);
\foreach \x in {5.6, 7.2, 8.8, 10.4} 
	\foreach \y in {1.8}
	\draw (\x,\y) .. controls (\x+0.2,\y-0.6) and (\x+0.6,\y-0.6) .. (\x+0.8,\y);
\foreach \x in {4.0,6.4,8.8,10.4} 
	\foreach \y in {0.5} 
	\draw (\x,0.5) .. controls (\x+0.2,\y+0.6) and (\x+0.6,\y+0.6) .. (\x+0.8,\y);
\foreach \x in {5.6,6.4,...,11.2} 
	\filldraw [black] (\x,1.8) circle (1.1pt);
\foreach \x in {3.2,4.0,4.8,5.6,...,12.8} 
\foreach \y in {2.3}	
	\filldraw [black] (\x,\y) circle (1.1pt);
\foreach \x in {3.2,4.0,4.8,5.6,...,12.8} 
\foreach \y in {1.8}	
	\filldraw [black] (\x,\y) circle (1.1pt);
\foreach \x in {3.2,4.0,4.8,5.6,...,12.8} 
\foreach \y in {0.5}	
	\filldraw [black] (\x,\y) circle (1.1pt);
\foreach \x in {3.2,4.0,4.8,5.6,...,12.8} 
	\foreach \y in {1.8} 
 	\draw[dotted] (\x,\y) -- (\x,\y+0.5);
\foreach \x in {3.2,12}
	\draw (\x,0.5) -- (\x,1.8);
\foreach \x in {3.2,4.0,...,5.6}
	\draw (\x,2.3) -- (\x,2.6);
\foreach \x in {3.2,4.0,...,5.6}
 \draw[dashed] (\x,2.5) -- (\x,2.8);
\draw (4.0,1.8) -- (5.6,0.5);
\draw (4.8,1.8) .. controls (5.6,0.8) and (7.2,1.5) .. (8.0,0.5);
\end{tikzpicture}
\caption{}
\label{figure:  inductive step for TL dangle argument 1}
  \end{figure}

{\em Case 1.} \quad  $j = k+1$,  $\el' = \el$.   In this case,   $x(s)$  is obtained from $x(t)$  by adding a vertical strand at the new vertex $\pbar {n+1}$.  On the other hand,
$$
\begin{aligned}
x_{n}\power{\el'}  d_\mfs  &=    x_n\power \el  d_{k \to k+1}\power{n+1}  d_t \\
&=  x_n\power \el  e_{n-1}\power \el  d_t   \\
&=     x_n\power \el  (x_{n-1}\power \el)^*   x_{n-1}\power \el   d_t   \\
&=   x_n\power \el  (x_{n-1}\power \el)^*   x(\mft),
\end{aligned}
$$
using the induction hypothesis at the last step.    Multiplication of an $n$--dangle of rank $k$  on the left by 
$ x_n\power \el  (x_{n-1}\power \el)^* $   adds a vertical strand on the right, as shown in \hyperref[figure:  inductive step for TL dangle argument 1]{Figure~\ref*{figure:  inductive step for TL dangle argument 1}}.
 Hence we have $x_{n}\power{\el'}  d_\mfs  = x(\mfs)$.

\begin{figure}[ht]
\begin{tikzpicture}
\foreach \x in {4.8,6.4,8.0,9.6,11.2} 
	\foreach \y in {2.3}
		\draw (\x,\y) .. controls (\x+0.2,\y+0.6) and (\x+0.6,\y+0.6) .. (\x+0.8,\y);
\foreach \x in {5.6, 7.2, 8.8, 10.4} 
	\foreach \y in {1.8}
	\draw (\x,\y) .. controls (\x+0.2,\y-0.6) and (\x+0.6,\y-0.6) .. (\x+0.8,\y);
\foreach \x in {4.0,6.4,8.8,10.4} 
	\foreach \y in {0.5} 
	\draw (\x,0.5) .. controls (\x+0.2,\y+0.6) and (\x+0.6,\y+0.6) .. (\x+0.8,\y);
\foreach \x in {5.6,6.4,...,11.2} 
	\filldraw [black] (\x,1.8) circle (1.1pt);
\foreach \x in {3.2,4.0,4.8,5.6,...,12.8} 
\foreach \y in {2.3}	
	\filldraw [black] (\x,\y) circle (1.1pt);
\foreach \x in {3.2,4.0,4.8,5.6,...,12.8} 
\foreach \y in {1.8}	
	\filldraw [black] (\x,\y) circle (1.1pt);
\foreach \x in {3.2,4.0,4.8,5.6,...,12.8} 
\foreach \y in {0.5}	
	\filldraw [black] (\x,\y) circle (1.1pt);
\foreach \x in {3.2,4.0,4.8,5.6,...,12.8} 
	\foreach \y in {1.8} 
 	\draw[dotted] (\x,\y) -- (\x,\y+0.5);
\foreach \x in {3.2,12}
	\draw (\x,0.5) -- (\x,1.8);
\foreach \x in {3.2,4.0}
	\draw (\x,2.3) -- (\x,2.6);
\foreach \x in {3.2,4.0}
 \draw[dashed] (\x,2.5) -- (\x,2.8);
\draw (4.0,1.8) -- (5.6,0.5);
\draw (4.8,1.8) .. controls (5.6,0.8) and (7.2,1.5) .. (8.0,0.5);
\end{tikzpicture}
\caption{}
\label{figure:  inductive step for TL dangle argument 2}
\end{figure}

  {\em Case 2.} \quad  $j = k- 1$,  $\el' = \el+ 1$.   In this case,   $x(s)$  is obtained from $x(t)$  by ``closing'' the rightmost vertical strand;  that is,  if $\pbar j$ is the vertex adjacent to this strand, the strand is replaced by
 a horizontal strand joining $\pbar j$ and $\pbar {n+1}$.    On the other  hand,
 $$
\begin{aligned}
x_{n}\power{\el'}  d_\mfs  &=    x_n\power {\el+1}  d_{k \to k+1}\power{n+1}  d_t \\
&=   x_n\power {\el+1}  (x_{n-1}\power \el)^*   x(\mft),
\end{aligned}
$$
by the same computation as in the previous case.  But multiplication of an $n$--dangle of rank $k$ 
 on the left by 
$ x_n\power {\el+1} (x_{n-1}\power \el)^* $    closes the rightmost vertical strand, as shown in \hyperref[figure:  inductive step for TL dangle argument 2]{Figure~\ref*{figure:  inductive step for TL dangle argument 2}}.  So again we have $x_{n}\power{\el'}  d_\mfs  = x(\mfs)$, and this completes the inductive proof.
  \end{proof}
  
\begin{corollary} \label{corollary:  cell filtrations for TL over Z[delta]}
Let $A_n$ denote the Jones--Temperley--Lieb algebra over the generic ground ring $\Z[\deltabold]$.  
For  $n \ge 0$ and for  $\Delta$ a cell module of  $A_{n+1}$,    the restricted module $\Res^{A_{n+1}}_{A_n}(\Delta)$ has an order preserving cell filtration.
\end{corollary}

\begin{proof}  One can either use the same proof as for 
 \hyperref[corollary:   cell filtrations for BMW over generic ground ring]{Corollary~\ref*{corollary:   cell filtrations for BMW over generic ground ring}}, or one can check directly using a diagrammatic model of the cell modules that $\Res^{A_{n+1}}_{A_n}(\cell A {n+1} k)$ has a filtration
 $$
 0 \subseteq N \subseteq \Res^{A_{n+1}}_{A_n}(\cell A {n+1} k),
 $$
 with $N \cong  \cell A n {k-1}$  and $\Res^{A_{n+1}}_{A_n}(\cell A {n+1} k)/N \cong \cell A n {k+1}$. 
\end{proof}

\subsection{Partition algebras}\label{p-m-s}
The partition algebras $A_n(k)$, for $k,n\in\mathbb{Z}_{\ge0},$ are a family of algebras defined in the work of Martin and Jones in ~\cite{MR1768036, MR1103994, MR1265453, MR1317365} in connection with the Potts model and higher dimensional statistical mechanics.  Martin~\cite{MR1317365} showed that the even partition algebra $A_{2n}(k)$ is in Schur--Weyl duality with the symmetric group $\mathfrak{S}_k$ acting diagonally on the $n$--fold tensor product $V^{\otimes n}$ of its $k$--dimensional permutation representation $V$. In~\cite{MR1768036}, Martin defined the odd partition algebra $A_{2n+1}(k)$ as the centraliser of the subgroup $\mathfrak{S}_{k-1}\subseteq \mathfrak{S}_k$ acting on $V^{\otimes n}$. Including the algebras $A_{2n+1}(k)$ in the tower 
\begin{align}\label{tower}
A_0(k)\subseteq A_1(k)\subseteq A_{2}(k)\subseteq A_{3}(k)\subseteq\cdots 
\end{align}
allowed for the simultaneous analysis of the whole tower of algebras~\eqref{tower} using the Jones basic construction, by Martin~\cite{MR1768036} and Halverson and Ram~\cite{MR2143201}.

For $n\in\mathbb{Z}_{\ge0}$ let 
\begin{align*}
{P}_{2n}&=\left\{\text{set partitions of $\{\p 1,\p 2,\dots,\p n,\pbar 1,\pbar 2,\dots,\pbar n\}$}\right\},\qquad\text{and,}\\
{P}_{2n-1}&=\left\{d\in P_{2n}\ \big|\ \text{$\p n$ and $\pbar n$ are in the same block of $d$}\right\}.
\end{align*}
Any element $\rho\in{P}_{2n}$ may be represented as a graph with $n$ vertices in the top row, labelled from left to right, by $\p 1,\p 2,\dots,\p n$ and $n$ vertices in the bottom row, labelled, from left to right by $\pbar 1,\pbar 2,\dots,\pbar n$, with the connected components of the graph being the blocks of $\rho$. The representation of a partition by a diagram is not unique; for example the partition 
\begin{align*}
\rho=\left\{\{\p 1,\pbar 1,\p 3,\pbar 4,\pbar 5,\p 6\},\{\p 2,\pbar 2,\pbar 3,\p 4,\p 5,\pbar 6\}\right\}
\end{align*}

can be represented by the diagrams:
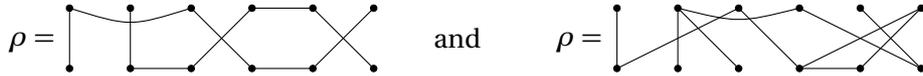
\begin{align*}
\begin{matrix}
\begin{tikzpicture}
\node at (-7.1,-0.45) {$\displaystyle \rho= \color{black}$};
\filldraw [black] (-6.6,0) circle (1.2pt);
\filldraw [black] (-6.6,-0.8) circle (1.2pt);
\filldraw [black] (-5.8,0) circle (1.2pt);
\filldraw [black] (-5.8,-0.8) circle (1.2pt);
\filldraw [black] (-5.0,0) circle (1.2pt);
\filldraw [black] (-5.0,-0.8) circle (1.2pt);
\filldraw [black] (-4.2,0) circle (1.2pt);
\filldraw [black] (-4.2,-0.8) circle (1.2pt);
\filldraw [black] (-3.4,0) circle (1.2pt);
\filldraw [black] (-3.4,-0.8) circle (1.2pt);
\filldraw [black] (-2.6,-0.0) circle (1.2pt);
\filldraw [black] (-2.6,-0.8) circle (1.2pt);
\draw (-6.6,0) -- (-6.6,-0.8);
\draw (-5.8,0) -- (-5.8,-0.8);
\draw (-5.8,-0.8) -- (-5.0,-0.8);
\draw (-6.6,0) .. controls (-5.8,-0.25) .. (-5.0,0.0);
\draw (-5.0,-0.0) -- (-4.2,-0.8);
\draw (-5.0,-0.8) -- (-4.2,-0.0);
\draw (-3.4,-0.0) -- (-4.2,-0.0);
\draw (-3.4,-0.8) -- (-4.2,-0.8);
\draw (-2.6,-0.8) -- (-3.4,-0.0);
\draw (-2.6,-0.0) -- (-3.4,-0.8);
\node at (-1.5,-0.40) {$\displaystyle \text{and} \color{black}$};
\node at (0.1,-0.45) {$\displaystyle \rho= \color{black}$};
\filldraw [black] (0.6,0) circle (1.2pt);
\filldraw [black] (0.6,-0.8) circle (1.2pt);
\filldraw [black] (1.4,0) circle (1.2pt);
\filldraw [black] (1.4,-0.8) circle (1.2pt);
\filldraw [black] (2.2,0) circle (1.2pt);
\filldraw [black] (2.2,-0.8) circle (1.2pt);
\filldraw [black] (3.0,0.0) circle (1.2pt);
\filldraw [black] (3.0,-0.8) circle (1.2pt);
\filldraw [black] (3.8,0.0) circle (1.2pt);
\filldraw [black] (3.8,-0.8) circle (1.2pt);
\filldraw [black] (4.6,0) circle (1.2pt);
\filldraw [black] (4.6,-0.8) circle (1.2pt);
\draw (0.6,0) -- (0.6,-0.8);
\draw (0.6,-0.8) -- (2.2,-0.0);
\draw (1.4,0) -- (1.4,-0.8);
\draw (1.4,0) -- (2.2,-0.8);
\draw (1.4,0) .. controls (2.2,-0.2) .. (3.0,-0.0);
\draw (2.2,-0.0) -- (3.0,-0.8);
\draw (3.0,-0.8) -- (3.8,-0.8);
\draw (3.0,-0.8) -- (4.6,-0.0);
\draw (3.8,-0.8) -- (4.6,-0.0);
\draw (3.0,-0.0) -- (4.6,-0.8);
\draw (3.8,-0.0) -- (4.6,-0.8);
\end{tikzpicture}
\end{matrix}
\end{align*}
If $\rho_1,\rho_2\in{P}_{2n}$, then the composition $\rho_1\circ\rho_2$ is the partition obtained by placing $\rho_1$ above $\rho_2$ and identifying each vertex in the bottom row of $\rho_1$ with the corresponding vertex in the top row of $\rho_2$ and deleting any components of the resulting diagram which contains only elements from the middle row. 
\begin{definition}
Let $S$ be a commutative unital ring and $\delta\in S$.   For $n \ge 1$,  the {\em partition algebra}  $A_{2n}(S;\delta)$ is the free $S$--module with basis $P_{2n}$, equipped with the product $\rho_1\rho_2=\delta^{\el}\rho_1\circ\rho_2$, for $\rho_1,\rho_2\in{P}_{2n}$,  
where $\el$ is the number of blocks removed from the middle row in constructing the composition $\rho_1\circ\rho_2$.  By convention , $A_0(S; \delta) = S$.  
Let $A_{2n-1}(S; \delta)$ denote the subalgebra of $A_{2n}(S; \delta)$ spanned by $P_{2n-1}$.
\end{definition}

\ignore{
The identity of  $A_{2n}(\delta)$ is the diagram
\begin{align*}
\begin{matrix}
\begin{tikzpicture}
\node at (-7.4,-0.45) {$\displaystyle 1= \color{black}$};
\node at (-5.35,-0.45) {$\displaystyle \cdots \color{black}$};
\draw (-6.8,0) -- (-6.8,-0.8);
\draw (-6,0) -- (-6,-0.8);
\draw (-4.7,0) -- (-4.7,-0.8);
\draw (-3.9,0) -- (-3.9,-0.8);
\filldraw [black] (-6.8,0) circle (1.2pt);
\filldraw [black] (-6.8,-0.8) circle (1.2pt);
\filldraw [black] (-6.0,0) circle (1.2pt);
\filldraw [black] (-6.0,-0.8) circle (1.2pt);
\filldraw [black] (-4.7,0) circle (1.2pt);
\filldraw [black] (-4.7,-0.8) circle (1.2pt);
\filldraw [black] (-3.9,0) circle (1.2pt);
\filldraw [black] (-3.9,-0.8) circle (1.2pt);
\end{tikzpicture}
\end{matrix}
\end{align*}
}

The Brauer algebra
$B_n(S; \delta)$ imbeds as a subalgebra of $A_{2n}(S;\delta)$, spanned by partitions with each block having two elements.  In particular, 
$A_{2n}(S;\delta)$ has a subalgebra isomorphic to the symmetric group algebra $S \mathfrak S_n$, spanned by permutation diagrams.   The permutation subalgebra is generated by the transpositions
\begin{align*}
\begin{matrix}
\begin{tikzpicture}
\node at (-7.4,-0.5) {$\displaystyle {s_i}= \color{black}$};
\node at (-6.3,-0.45) {$\displaystyle \cdots \color{black}$};
\draw (-6.8,0) -- (-6.8,-0.8);
\draw (-5.8,0) -- (-5.8,-0.8);
\draw (-5,0) -- (-4.2,-0.8);
\draw (-5,-0.8) -- (-4.2,-0.0);
\draw (-3.4,0) -- (-3.4,-0.8);
\draw (-2.4,0) -- (-2.4,-0.8);
\filldraw [black] (-6.8,0) circle (1.2pt);
\filldraw [black] (-6.8,-0.8) circle (1.2pt);
\filldraw [black] (-5.8,0) circle (1.2pt);
\filldraw [black] (-5.8,-0.8) circle (1.2pt);
\filldraw [black] (-5.0,0) circle (1.2pt);
\filldraw [black] (-5.0,-0.8) circle (1.2pt);
\filldraw [black] (-4.2,0) circle (1.2pt);
\filldraw [black] (-4.2,-0.8) circle (1.2pt);
\filldraw [black] (-3.4,0) circle (1.2pt);
\filldraw [black] (-3.4,-0.8) circle (1.2pt);
\filldraw [black] (-2.4,0) circle (1.2pt);
\filldraw [black] (-2.4,-0.8) circle (1.2pt);
\node at (-2.9,-0.45) {$\displaystyle \cdots \color{black}$};
\node at (-5.0,-1.1) {$\displaystyle i \color{black}$};
\node at (-4.1,-1.1) {$\displaystyle i+1 \color{black}$};
\end{tikzpicture}
\end{matrix}\ .
\end{align*}
The multiplicative identity of $A_{2n}(S; \delta)$ is the trivial permutation.   
It is not hard to see that the partition algebra $A_{2n}(S; \delta)$ is generated by the transpositions $s_i$  ($1 \le i \le n-1$) and elements $e_j$  ($1 \le j \le 2n -1$),  where
\begin{align*}
\begin{matrix}
\begin{tikzpicture}
\node at (-7.6,-0.45) {$\displaystyle e_{2i-1}= \color{black}$};
\node at (-6.3,-0.45) {$\displaystyle \cdots \color{black}$};
\draw (-6.8,0) -- (-6.8,-0.8);
\draw (-5.8,0) -- (-5.8,-0.8);
\draw (-4.2,0) -- (-4.2,-0.8);
\draw (-3.2,0) -- (-3.2,-0.8);
\filldraw [black] (-6.8,0) circle (1.2pt);
\filldraw [black] (-6.8,-0.8) circle (1.2pt);
\filldraw [black] (-5.8,0) circle (1.2pt);
\filldraw [black] (-5.8,-0.8) circle (1.2pt);
\filldraw [black] (-5.0,0) circle (1.2pt);
\filldraw [black] (-5.0,-0.8) circle (1.2pt);
\filldraw [black] (-4.2,0) circle (1.2pt);
\filldraw [black] (-4.2,-0.8) circle (1.2pt);
\filldraw [black] (-3.2,0) circle (1.2pt);
\filldraw [black] (-3.2,-0.8) circle (1.2pt);
\node at (-3.7,-0.45) {$\displaystyle \cdots \color{black}$};
\node at (-6.8,-1.1) {$\displaystyle 1 \color{black}$};
\node at (-5.0,-1.1) {$\displaystyle i \color{black}$};
\node at (-3.2,-1.1) {$\displaystyle n \color{black}$};
\node at (-1.9,-0.45) {$\displaystyle \text{and} \color{black}$};
\node at (-0.0,-0.45) {$\displaystyle e_{2i}= \color{black}$};
\node at (1.1,-0.45) {$\displaystyle \cdots \color{black}$};
\draw (0.6,0) -- (0.6,-0.8);
\draw (1.6,0) -- (1.6,-0.8);
\draw (2.4,0) -- (3.2,0);
\draw (2.4,0) -- (2.4,-0.8);
\draw (2.4,-0.8) -- (3.2,-0.8);
\draw (3.2,0) -- (3.2,-0.8);
\draw (4.0,0) -- (4.0,-0.8);
\draw (5.0,0) -- (5.0,-0.8);
\filldraw [black] (0.6,0) circle (1.2pt);
\filldraw [black] (0.6,-0.8) circle (1.2pt);
\filldraw [black] (1.6,0) circle (1.2pt);
\filldraw [black] (1.6,-0.8) circle (1.2pt);
\filldraw [black] (2.4,0) circle (1.2pt);
\filldraw [black] (2.4,-0.8) circle (1.2pt);
\filldraw [black] (3.2,0) circle (1.2pt);
\filldraw [black] (3.2,-0.8) circle (1.2pt);
\filldraw [black] (4,0) circle (1.2pt);
\filldraw [black] (4,-0.8) circle (1.2pt);
\filldraw [black] (5,0) circle (1.2pt);
\filldraw [black] (5,-0.8) circle (1.2pt);
\node at (4.5,-0.45) {$\displaystyle \cdots \color{black}$};
\node at (0.6,-1.1) {$\displaystyle 1 \color{black}$};
\node at (2.4,-1.1) {$\displaystyle i \color{black}$};
\node at (3.3,-1.1) {$\displaystyle i+1 \color{black}$};
\node at (5,-1.1) {$\displaystyle n \color{black}$};
\end{tikzpicture}
\end{matrix}\ ,
\end{align*}
Halverson and Ram~\cite[Theorem~1.11]{MR2143201} and East~\cite[Theorem~36]{MR2811310} give a presentation for $A_{2n}$ in terms of the generators $e_j$ and $s_i$.  
 The algebras $A_{2n}(S;\delta)$ and $A_{2n-1}(\delta)$  have an algebra involution $*$ which acts on diagrams by flipping them over the horizontal line $y = 1/2$.    The generators $s_i$ and $e_j$ are $*$--invarariant.
 
 $A_{2n -1}(S; \delta)$ is defined as a subalgebra of $A_{2n}(S; \delta)$, and $A_{2n}(S; \delta)$ imbeds in $A_{2n + 1}(S; \delta)$ as follows:  define a map $\iota: P_{2n} \to P_{2n+1}$  by adding an additional block $\{\p{n+1}, \pbar{n+1}\}$.    The linear extension of $\iota$ is a monomorphism of algebras with involution.  

\def\pn{{\rm pn}}
Let $d \in P_{2n}$.  Call a block of $d$  a {\em through block}  if the block has non--empty intersection with both
$[\p n ]$  and $[\pbar n]$.    The number of through blocks of $d$  is called the propagating number
of $d$, denoted ${\rm pn}(d)$.   Clearly, $\pn(d) \le n$ for all $d \in P_{2n}$.   The only $d \in P_{2n}$  with propagating number equal to $n$  are the  permutation diagrams.
If $x, y \in P_{2n}$  and $x y = \delta^r z$,  then
${\rm pn}(z)  \le \min\{ {\rm pn}(x), {\rm pn}(y)\}$.   Hence the span of the set of $d \in P_{2n}$  with
$\pn(d) < n$ is an ideal $J_{2n} \subset A_{2n}(S; \delta)$.   Moreover,  $J_{2n -1}  := J_{2n}  \cap A_{2n -1}$  is the span
of $d \in P_{2n-1}$  with $\pn(d) < n$.      One can check that for $k \ge 2$,  $J_k$ is the ideal of $A_k(S; \delta)$  generated by $e_{k -1}$.      The idea $J_k$  is $*$--invariant, and the span of permutation diagrams in $A_k$  is a $*$--invariant linear complement for $J_k$.  It follows that  $A_{2n}(S; \delta)/J_{2n} \cong S \mathfrak S_{n}$  and $A_{2n-1}(S; \delta)/J_{2n-1} \cong  S \mathfrak S_{n-1}$  as algebras with involution;   the isomorphisms are determined by $v  + J_k \mapsto v$,  where $v$ is a permutation diagram.

\ignore{
If $v\in\mathfrak{S}_n$ and $v=s_{i_1}\cdots s_{i_j}$ is a reduced expression, then $t_v=t_{s_{i_1}}\cdots t_{s_{i_j}}$ is well defined. For $i,j,=1,\ldots,n,$ let
\begin{align*}
t_{i,j}=
\begin{cases}
t_{s_i}t_{s_{i+1}}\cdots t_{s_{j-1}},&\text{if $i\le j$,}\\
t_{s_{i-1}}t_{s_{i-2}}\cdots t_{s_j},&\text{if $i>j$.}
\end{cases}
\end{align*}
}

\subsubsection{The Murphy basis}
The generic ground ring for the partition algebras is $R_0 = \Z[\deltabold]$, where $\deltabold$ is an indeterminant. Write $R = \Z[\deltabold^{\pm 1}]$, and let 
 Let $F = \Q(\deltabold)$ denote the field of fractions of $R$.  
 Write $A_n$  for $A_n(R; \deltabold)$ and write   
 $H_{2i} = H_{2i + 1} = R\mathfrak S_i$ for $i \ge 0$.   
 The tower $(H_n)_{n \ge 0}$ is a strongly coherent tower of cyclic cellular algebras, and  $H_n^F$ is split semisimple.  
The branching diagram of the tower  $(H_n)_{n \ge 0}$  is the graph $\widehat{H}$ with 
\begin{enumerate}[label=(\arabic{*}), ref=\arabic{*},leftmargin=0pt,itemindent=1.5em]
\item  $\widehat H_{2i} = \widehat H_{2i + 1} = $  the set  $\mathcal Y_i$ of Young diagrams of size $i$.
\item an edge $\lambda\to\mu$ in $\widehat{H}$ if
\begin{enumerate}
\item $\lambda\in\widehat{H}_{2i-1}$, $\mu\in\widehat{H}_{2i}$ and $\lambda\subseteq\mu$, or 
\item $\lambda\in\widehat{H}_{2i}$, $\mu\in\widehat{H}_{2i+1}$ and $\lambda=\mu$. 
\end{enumerate}
\end{enumerate}

\ignore{
The first few levels of $\widehat{H}$ are given in~\eqref{i-h-s}.
\begin{align}\label{i-h-s}
\begin{matrix}
\xymatrix{
0& & &\emptyset\ar[d] \\
 & & &\emptyset\ar[d] \\ 
2& & &\text{\tiny\Yvcentermath1$\yng(1)$}\ar[d]\\
 & & &\text{\tiny\Yvcentermath1$\yng(1)$}\ar[dl]\ar[dr]\\
4& &\text{\tiny\Yvcentermath1$\yng(2)$}\ar[d]& &\text{\tiny\Yvcentermath1$\yng(1,1)$}\ar[d]\\
& &\text{\tiny\Yvcentermath1$\yng(2)$}\ar[dl]\ar[dr]& &\text{\tiny\Yvcentermath1$\yng(1,1)$}\ar[dr]\ar[dl]\\ 
6&\text{\tiny\Yvcentermath1$\yng(3)$}& &\text{\tiny\Yvcentermath1$\yng(2,1)$} &&\text{\tiny\Yvcentermath1$\yng(1,1,1)$}
}
\end{matrix}
\end{align}
}

\ignore{
Note that the generators $e_j$ and $t_j$ of the partition algebras are $*$--invariant, 
\begin{align*}
t_{s_j}^*=t_{s_j},\qquad  e_j^*=e_j\quad \text{ for all } j.
\end{align*}
}

\ignore{ 
The maps
\begin{align*}
A_{2i}/(A_{2i}e_{2i-1}A_{2i})&\to H_{2i}, &&\text{and} && & A_{2i+1}/(A_{2i+1}e_{2i}A_{2i+1})&\to H_{2i + 1},\\
t_v+(e_{2i-1})&\mapsto v, &&
&& &t_v+(e_{2i})&\mapsto v, &&\text{for $v\in\mathfrak{S}_i$,}
\end{align*}
are algebra isomorphisms. 
}

It is shown in~\cite[Sect.~5.7]{MR2794027}  that the pair of towers $(A_n)_{n \ge 0}$   and $(H_n)_{n \ge 0}$   satisfy the framework axioms~\eqref{axiom: involution on An}--\eqref{axiom: semisimplicity}  of \hyperref[subsection:  cellularity and jones setting and correction]{Section~\ref*{subsection:  cellularity and jones setting and correction}}.  
Axiom~\eqref{axiom Hn coherent}  holds by \hyperref[corollary:  Hecke algebras strongly coherent]{Corollary~\ref*{corollary:  Hecke algebras strongly coherent}}.   Axiom~\eqref{axiom:  Delta J} hold for the partition algebras,  by the remarks at the end of \hyperref[subsection:  cellularity and jones setting and correction]{Section~\ref*{subsection:  cellularity and jones setting and correction}}.  Finally, Axiom~\eqref{axiom: Hn cyclic cellular}  holds by \hyperref[corollary Hecke algebras cyclic cellular]{Corollary~\ref*{corollary Hecke algebras cyclic cellular}}.  Therefore, by \hyperref[theorem: jones tower coherent tower cyclic cellular]{Theorem~\ref*{theorem: jones tower coherent tower cyclic cellular}}, the tower of algebras  $(A_n)_{n \ge 0}$ 
is a strongly coherent tower of cyclic cellular algebras.

By the discussion in \hyperref[subsection results on cellularity from Goodman Graber]{Section~\ref*{subsection results on cellularity from Goodman Graber}},  the partially ordered set $\hat A_i$ in the cell datum for $A_i$ can be realized as
\begin{align*}
\hat{A}_i=\left\{ 
(\lambda,l)\ \big|\ 
\text{$\lambda\in\widehat{H}_{i-2l}$, for $l=0,1,\ldots,\lfloor i/2\rfloor$}
\right\}
\end{align*}
ordered by  $(\lambda,l)\unrhd(\mu,m)$  if $l>m$, or if $l=m$ and $\lambda\unrhd\mu$ as elements of $\widehat{H}_{i-2l}$.   The branching diagram $\hat A$ for the tower $(A_n)_{n \ge 0}$ is that obtained by reflections from the branching diagram $\widehat H$.  Thus the branching rule is the following:
\begin{enumerate}
\item 
Let $i$ be even and $(\la, \el) \in \hat A_i$.  
\begin{enumerate}
\item For $(\mu, \el) \in \hat A_{i+1}$,  
$(\la, \el) \to (\mu, \el)$ in $\hat A$ if and only if $\la = \mu$.  
\item For 
$(\mu, \el + 1) \in \hat A_{i+1}$,  
$(\la, \el) \to (\mu, \el+ 1)$ in $\hat A$ if and only $\mu \subset \la$. 
\end{enumerate}  
\item Let $i$ be odd and $(\la, \el) \in \hat A_i$.
\begin{enumerate}
\item For $(\mu, \el) \in \hat A_{i+1}$,  
$(\la, \el) \to (\mu, \el)$ in $\hat A$ if and only if $\la  \subset \mu$.  
\item For 
$(\mu, \el + 1) \in \hat A_{i+1}$,  
$(\la, \el) \to (\mu, \el+ 1)$ in $\hat A$ if and only $\la = \mu$. 
\end{enumerate} 
\end{enumerate}
The first few levels of $\hat{A}$ are given in  \hyperref[figure: branching diagram for partition algebras]{Figure~\ref*{figure: branching diagram for partition algebras}}.

  \begin{figure}
  \begin{align*}
\begin{matrix}
\xymatrix{
0&\emptyset\ar[d]        & & & & & &\\
&\emptyset \ar[d]\ar[dr] & & & & & &\\
2&\emptyset\ar[d]&\text{\tiny\Yvcentermath1$\yng(1)$}\ar[d]\ar[dl]& & & & &\\
&\emptyset\ar[d]\ar[dr]&\text{\tiny\Yvcentermath1$\yng(1)$}\ar[d]\ar[dr]\ar[drr] & & & & &\\
4&\emptyset\ar[d]&\text{\tiny\Yvcentermath1$\yng(1)$}\ar[d]\ar[dl] &\text{\tiny\Yvcentermath1$\yng(2)$}\ar[d]\ar[dl] & \text{\tiny\Yvcentermath1$\yng(1,1)$}\ar[d]\ar[dll] & & & \\
&\emptyset \ar[d]\ar[dr] &\text{\tiny\Yvcentermath1$\yng(1)$} \ar[d]\ar[dr]\ar[drr]&\text{\tiny\Yvcentermath1$\yng(2)$} \ar[d]\ar[drr]\ar[drrr]& \text{\tiny\Yvcentermath1$\yng(1,1)$}\ar[d]\ar[drr]\ar[drrr] & & &\\
6& \emptyset & \text{\tiny\Yvcentermath1$\yng(1)$} &\text{\tiny\Yvcentermath1$\yng(2)$} & \text{\tiny\Yvcentermath1$\yng(1,1)$} &\text{\tiny\Yvcentermath1$\yng(3)$} & \text{\tiny\Yvcentermath1$\yng(2,1)$} &\text{\tiny\Yvcentermath1$\yng(1,1,1)$}
}
\end{matrix}
\end{align*}
\caption{Branching diagram for the partition algebras.}
\label{figure: branching diagram for partition algebras}
  \end{figure}
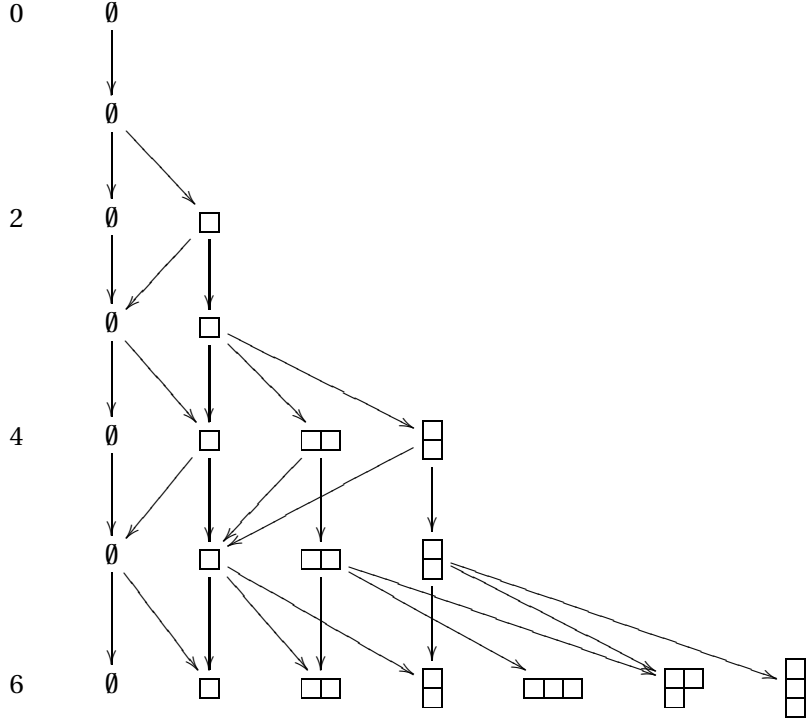

Next, we determine the branching coefficients for the two towers $(H_n)_{n\ge 0}$  and $(A_n)_{n\ge 0}$.
Let  $\lambda\in \widehat{H}_{2i-1}$ and $\mu\in\widehat{H}_{2i}$  with $\la \to \mu$ in $\widehat H$.   If  $\mu=\lambda\cup\{(r,\mu_r)\}$, let $a=\sum_{j=1}^r\mu_j$. Then the branching factors for the inclusion $H_{2i-1}\subseteq H_{2i}$ in the tower $(H_i)_{i\ge0}$ are given by 
\begin{align} \label{branching factors in H even}
d_{\lambda\to\mu}^{(2i)}=s_{a,i} \qquad \text{and} \qquad  u_{\lambda\to\mu}^{(2i)}=s_{i,a}\sum_{k=0}^{\lambda_r}s_{a,a-k},
\end{align}
where the elements $s_{i, j}$ are defined in Equation~\eqref{definition of s i j}.  
The branching factors for the inclusion $H_{2i}\subseteq H_{2i+1}$ in the tower $(H_{i})_{i\ge 0}$ are given by
\begin{align}  \label{branching factors in H odd}
d_{\lambda\to \la}^{(2i+1)}= u_{\lambda\to\la}^{(2i+1)}=1 \qquad \text{if $\lambda\in\widehat{H}_{2i} = \widehat{H}_{2i+1}$}.
\end{align}
For $\la \in \widehat{H}_k$ and $\mu \in \widehat{H}_{k+1}$,  define $\bar{d}\power {k+1}_{\la \to \mu}$ 
and $\bar{u}\power {k+1}_{\la \to \mu}$ by the same formulas, specifying elements of  the subalgebra of $A_{k+1}$ spanned by permutation diagrams; these are liftings in $A_{k+1}$ of the branching factors in   $H_{k+1}$ specified above.

\ignore{
If $\lambda\in\widehat{H}_{2i-1}$ and $\mu\in\widehat{H}_{2i}$, such that $\mu=\lambda\cup\{(j,\mu_j)\}$, let $a=\sum_{r=1}^{j}\mu_j$ and define
\begin{align*}
\bar{u}_{\lambda\to\mu}^{\,(2i)}=t_{i,a}
\sum_{r=0}^{\lambda_j}t_{a,a-r}&&\text{and}&&
\bar{d}_{\lambda\to\mu}^{\,(2i)}=t_{a,i}.
\end{align*}
If $\mu\in\widehat{H}_{2i}$ and $\nu\in\widehat{H}_{2i+1}$ such that $\mu\to\nu$ in $\widehat{H}$, define
\begin{align*}
\bar{u}_{\mu\to\nu}^{\,(2i+1)}=\bar{d}_{\mu\to\nu}^{\,(2i+1)}=1. 
\end{align*}
These elements are liftings in the partition algebras $A_{2i}$ or $A_{2i+ 1}$ of branching factors in the algebras $H_{2i}$ or $H_{2i+ 1}$ which were given in Equations~\eqref{branching factors in H even} and~\eqref{branching factors in H odd}.
}

By \hyperref[theorem:  closed form determination of the branching factors]{Theorem~\ref*{theorem:  closed form determination of the branching factors}}, the branching factors for the tower $(A_n)_{n \ge 0}$ can be chosen as follows:
Let $(\lambda,l)\in\hat{A}_{2i-1}$. If $(\mu,l)\in\hat{A}_{2i}$  and $(\lambda,l)\to(\mu,l)$ in $\hat{A}$, then 
$\la \subset \mu$ and 
\begin{align*}
d_{(\lambda,l)\to(\mu,l)}^{(2i)}=
\bar{d}_{\lambda\to\mu}^{\,(2i-2l)}e_{2i-2}^{(l)},
\end{align*}
and, if $(\mu,l+1)\in\hat{A}_{2i}$  and $(\lambda,l)\to(\mu,l+1)$ in $\hat{A}$, then  $\la = \mu$ and
\begin{align*}
d_{(\lambda,l)\to(\la,l+1)}^{(2i)}=
e_{2i-2}^{(l)}.
\end{align*}
Similarly, if $(\lambda,l)\in\hat{A}_{2i}$ and $(\mu,\el)\in\hat{A}_{2i+1}$ and $(\lambda,l)\to(\mu,l)$, then
$\la = \mu$ and
\begin{align*}
d_{(\lambda,l)\to(\la,l)}^{(2i+1)}=e_{2i-1}^{(l)},
\end{align*}
and, if $(\mu,l+1)\in\hat{A}_{2i+1}$ and $(\lambda,l)\to(\mu,l+1)$, then $\mu \subset \la$ and 
\begin{align*}
d_{(\lambda,l)\to(\mu,l+1)}^{(2i+1)}=
\bar{u}_{\mu\to\lambda}^{\,(2i-2l)}e_{2i-1}^{(l)}. 
\end{align*}
The $u$--coefficients  $\uu {(\la, \el)} {\mu, m)}  {n+1}$  are determined by similar formulas by \hyperref[theorem:  closed form determination of the branching factors]{Theorem~\ref*{theorem:  closed form determination of the branching factors}}.

Fix $n \ge 1$ and
 $(\la, \el) \in \hat A_n$.     For a path $\mft$ in $\hat A_n \power {\la, \el}$,   define $d_\mft$ to be 
the ordered product of the $d$--branching coefficients for the tower $(A_k)$ along the path $\mft$, as in 
Equation~\eqref{equation: ordered product of d coefficients}.   
Define  $c_{(\la,\el)}={c}_{(\la,0)}e_{i-1}^{(\el)}$,  where  $c_{(\la,0)}=\sum_{v\in\mathfrak{S}_\la}  v$, and
and $e_{i-1}^{(\el)}$ is defined in Equation~\eqref{notation e sub i super ell}.
From \hyperref[label:1]{Corollary~\ref*{label:1}} \ignore{and \hyperref[theorem:  closed form determination of the branching factors]{Theorem~\ref*{theorem:  closed form determination of the branching factors}} }
we obtain:

\ignore{
For $\mathfrak{t}=((\lambda^{(0)},l_0),(\lambda^{(1)},l_1),\ldots,(\lambda^{(i)},l_{i}))\in\hat{A}_{i}^{(\lambda,l)}$, let 
\begin{align*}
d_\mathfrak{t}=
d_{(\lambda^{(i-1)},l_{i-1})\to (\lambda^{(i)},l_{i})}^{(i)}
d_{(\lambda^{(i-2)},l_{i-2})\to (\lambda^{(i-1)},l_{i-1})}^{(i-1)}
\cdots 
d_{(\lambda^{(0)},l_{0})\to(\lambda^{(1)},l_{1})}^{(1)}.
\end{align*}
From \hyperref[label:1]{Corollary~\ref*{label:1}} and \hyperref[theorem:  closed form determination of the branching factors]{Theorem~\ref*{theorem:  closed form determination of the branching factors}} we obtain:
}

\begin{proposition}\label{partition algebra basis}  Let $R = \Z[\deltabold^{\pm 1}]$ and let 
$A_n = A_n(R; \deltabold)$  be 
partition algebra defined over $R$ with parameter $\deltabold$.  
For each $n$,    the set
\begin{align*}
\mathscr{A}_n=
\left\{
d_\mathfrak{s}^*c_{(\lambda,l)}d_\mathfrak{t}\ \big|\ 
\text{$\mathfrak{s},\mathfrak{t}\in\hat{A}_n^{(\lambda,l)}$, $(\lambda,l)\in\hat{A}_n$}
\right\},
\end{align*}
is an $R$--basis for $A_n$, and $(A_n,*,\hat{A}_n,\unrhd,\mathscr{A}_n)$ is a cell datum for $A_n$. 
\end{proposition}

\subsubsection{The Murphy basis and the generic ground ring}
It remains to show that the set $\mathscr{A}_n$ is a basis for the partition algebra $A_n(R_0; \deltabold)$ defined over the generic ground ring $R_0 = \Z[\deltabold]$.    Let $\mathcal{B}$ denote the diagram basis for $A_n(R_0; \deltabold)$. 

\begin{definition}
Let $1\le l\le k\le j$. A set partition $\varrho$  of $P = \{\p 1, \dots, \p j, \pbar 1,\dots,  \pbar j\}$
 is said to be of \emph{even type} $(k,l)$ if each element of the set of lower vertices $\lbrace \pbar{k-l+1}, \pbar{k-l+2},\ldots,\pbar k\rbrace$ lives in a block of size one; $\varrho$ is said to be of \emph{odd type} $(k,l)$ if all the  lower vertices in the set $\lbrace \pbar{k-l},\pbar{k-l+1},\ldots,\pbar k\rbrace$ live in the same block of $\varrho$. 
\end{definition}

\begin{lemma}
\label{lemma: partition types}
Let $\varrho$ be a partition.
\begin{enumerate}
\item\label{partition types.1} If $\varrho$ is of odd type $(k+1,m)$ and $ (\lambda,l) \to (\mu,m)$ is an edge from level $ \,2k$ to level $ \, 2k+1$ in $\widehat{A}$, then $\varrho d_{(\lambda,l) \to (\mu,m)}^{(2k+1)}$ is a $\mathbb{Z}$-linear combination of partitions of even type $(k,\el)$. 
\item\label{partition types.2} If $\varrho$ is of even type $(k,m)$ and  $ (\lambda,l) \to (\mu,m)$ is an edge from level $\,2k-1$ to level $\,2k$ in $\widehat{A}$, then $\varrho d_{(\lambda,l) \to (\mu,m)}^{(2k)}$is a $\mathbb{Z}$-linear combination of partitions of odd type $(k,\el)$. 
\end{enumerate}
\end{lemma}

\begin{proof}  
\ignore{  
\begin{proof}  
\begin{figure}[ht]
\begin{tikzpicture}
\foreach \x in {6.4,7.2,...,9.6} 
	\foreach \y in {2.3}
		\draw (\x,\y)  .. controls (\x+0.2,\y+0.3) and (\x+0.6,\y+0.3) ..  (\x+0.8,\y);
\draw (5.6,2.3) .. controls (6.9,2.9) and (9.9,2.9) .. (11.2,2.3);
\foreach \x in {4.0,4.8,5.6,...,11.5} 
	\foreach \y in {2.3,1.8,0.5}	
	\filldraw [black] (\x,\y) circle (1.1pt);
\foreach \x in {4.0,4.8,5.6,...,11.5} 
	\foreach \y in {1.8} 
 	\draw[dotted] (\x,\y) -- (\x,\y+0.5);
\foreach \x in {4.0,4.8,5.6,10.4,11.2}
	\draw (\x,0.5) -- (\x,1.8);
\foreach \x in {4.0, 4.8,10.4,11.2}
	\draw (\x,2.3) -- (\x,2.6);
\foreach \x in {4.0,4.8,10.4,11.2}
 \draw[dashed] (\x,2.5) -- (\x,2.8);
\end{tikzpicture}
\caption{}
\label{figure:  partition odd type case 1}
\end{figure}
}   
\begin{figure}[ht]
\begin{tikzpicture}
\foreach \x in {6.4,7.2,...,9.6} 
	\foreach \y in {2.3}
		\draw (\x,\y)  
		-- (\x+0.8,\y);
\foreach \x in {4.0,4.8,5.6,...,11.5} 
	\foreach \y in {2.3,1.8,0.5}	
	\filldraw [black] (\x,\y) circle (1.1pt);
\foreach \x in {4.0,4.8,5.6,...,11.5} 
	\foreach \y in {1.8} 
 	\draw[dotted] (\x,\y) -- (\x,\y+0.5);
\foreach \x in {4.0,4.8,5.6,10.4,11.2}
	\draw (\x,0.5) -- (\x,1.8);
\foreach \x in {4.0, 4.8,5.6,10.4,11.2}
	\draw (\x,2.3) -- (\x,2.6);
\foreach \x in {4.0,4.8, 5.6, 10.4,11.2}
 \draw[dashed] (\x,2.5) -- (\x,2.8);
 \node at (6.4, 0) {$ k-\el +1$};
 \node at (9.6, 0) {$k$};
\end{tikzpicture}
\caption{}
\label{figure:  partition odd type case 1}
\end{figure}
Assume that $\varrho$ is of odd type $(k+1,m)$ and $ (\lambda,l) \to (\mu,m)$ is an edge from level $ \,2k$ to level $ \, 2k+1$ in $\widehat{A}$.    Thus $\varrho$  has lower vertices
$\pbar{k-m +1},  \dots, \pbar{k+1}$ in one block.
 If $\el = m$,   then $\la = \mu$,  and 
$\dd {(\la, \el)}  {(\mu, \el)}  {2k + 1} = e_{2k-1} \power \el$.    It follows that $\varrho \dd {(\la, \el)}  {(\mu, m)}  {2k + 1} = \varrho  e_{2k-1} \power \el$ is equal to a single partition of even type $(k, \el)$, and that no factor of $\deltabold$ arises in the computation of the product, as shown in \hyperref[figure:  partition odd type case 1]{Figure~\ref*{figure:  partition odd type case 1}}.  
If $m = \el + 1$,  then $\mu \subset \el$  and 
$$\dd {(\la, \el)}  {(\mu, \el+ 1)}  {2k + 1}   = \uubar \mu \la {2k - 2\el}    e_{2k - 1} \power \el,$$
which is a sum of elements of the form $s_{k- \el, j}   e_{2k - 1} \power \el$  with $j \le k - \el$.   It follows that  $\varrho \dd {(\la, \el)}  {(\mu, \el + 1)}  {2k + 1}$ is equal a sum of distinct partitions, each of even type $(k, \el)$, and again no factor of $\deltabold$ appears in the computation of the product, as shown in \hyperref[figure:  partition odd type case 2]{Figure~\ref*{figure:  partition odd type case 2}}.    

\ignore{   
\begin{figure}[ht]
\begin{tikzpicture}
\foreach \x in {6.4,7.2,...,8.8,9.6} 
	\foreach \y in {2.3}
		\draw (\x,\y) .. controls (\x+0.2,\y+0.3) and (\x+0.6,\y+0.3) .. (\x+0.8,\y);
\draw (5.6,2.3) .. controls (6.9,2.9) and (9.9,2.9) .. (11.2,2.3);
\foreach \x in {4.0,4.8,5.6,...,11.5} 
	\foreach \y in {2.3,1.8,0.5}	
	\filldraw [black] (\x,\y) circle (1.1pt);
\foreach \x in {4.0,4.8,5.6,...,11.5} 
	\foreach \y in {1.8} 
 	\draw[dotted] (\x,\y) -- (\x,\y+0.5);
\foreach \x in {10.4,11.2}
	\draw (\x,0.5) -- (\x,1.8);
\foreach \x in {4.0,4.8,5.6}
	\draw (\x+0.8,0.5) -- (\x,1.8);
\draw (6.4,1.8) -- (4.0,0.5);
\foreach \x in {4.0,4.8,5.6,10.4,11.2}
	\draw (\x,2.3) -- (\x,2.6);
\foreach \x in {4.0,4.8,5.6,10.4,11.2}
 \draw[dashed] (\x,2.5) -- (\x,2.8);
\end{tikzpicture}
\caption{}
\label{figure:  partition odd type case 2}
\end{figure}
}    

\begin{figure}[ht]
\begin{tikzpicture}
\foreach \x in {6.4,7.2,...,8.8,9.6} 
	\foreach \y in {2.3}
		\draw (\x,\y) 
		--  (\x+0.8,\y);
\foreach \x in {4.0,4.8,5.6,...,11.5} 
	\foreach \y in {2.3,1.8,0.5}	
	\filldraw [black] (\x,\y) circle (1.1pt);
\foreach \x in {4.0,4.8,5.6,...,11.5} 
	\foreach \y in {1.8} 
 	\draw[dotted] (\x,\y) -- (\x,\y+0.5);
\foreach \x in {10.4,11.2}
	\draw (\x,0.5) -- (\x,1.8);
\foreach \x in {4.0,4.8,5.6}
	\draw (\x+0.8,0.5) -- (\x,1.8);
\draw (6.4,1.8) -- (4.0,0.5);
\foreach \x in {4.0,4.8,5.6,10.4,11.2}
	\draw (\x,2.3) -- (\x,2.6);
\foreach \x in {4.0,4.8,5.6,10.4,11.2}
 \draw[dashed] (\x,2.5) -- (\x,2.8);
   \node at (7.2, 0) {$ k-\el +1$};
 \node at (9.6, 0) {$k$};
\end{tikzpicture}
\caption{}
\label{figure:  partition odd type case 2}
\end{figure}

Assume now that $\varrho$ is of even type $(k,m)$ and  $ (\lambda,l) \to (\mu,m)$ is an edge from level $\,2k-1$ to level $\,2k$ in $\widehat{A}$.   Thus  the lower vertices $\pbar {k-m + 1},  \dots,  \pbar k$ each constitute a block of $\varrho$.  
If $\el = m$,  then $\la \subset \mu$  and
$$\dd {(\la, \el)}  {(\mu, \el)}  {2k }   = \ddbar  \la  \mu {2k - 2\el}    e_{2k - 2} \power \el.$$
But $\ddbar  \la  \mu {2k - 2\el} = s_{j, k-\el}$  for some $j \le k- \el$, and $\varrho' = \varrho s_{j, k-\el}$ is also a partition of even type $(k, \el)$.   Thus, we have to consider
$\varrho \dd {(\la, \el)}  {(\mu, \el)}  {2k + 1}  = \varrho s_{j, k-\el}  e_{2k - 2} \power \el =
\varrho'  e_{2k - 2} \power \el $,  where   $\varrho' $ is a partition of even type $(k, \el)$.   The product
$\varrho'  e_{2k - 2} \power \el $ is a single partition, of odd type $(k, \el)$,  and no power of $\deltabold$ occurs in the computation of the product, as shown in \hyperref[figure:  partition even type case 1]{Figure~\ref*{figure:  partition even type case 1}}.  

\ignore{   
and $\ddbar  \la  \mu {2k - 2\el} = s_{j, k-\el}$  for some $j \le k- \el$.    Now it follows that 
$\varrho \dd {(\la, \el)}  {(\mu, \el)}  {2k + 1}  = \varrho s_{j, k-\el}  e_{2k - 2} \power \el$ is a single partition, of odd type $(k, \el)$,  and no power of $\deltabold$ occurs in the computation of the product, as shown in \hyperref[figure:  partition even type case 1]{Figure~\ref*{figure:  partition even type case 1}}.  
}

\ignore{  
\begin{figure}[ht]
\begin{tikzpicture}
\foreach \x in {7.2,8.0,8.8,9.6} 
	\foreach \y in {0.5}
		\draw (\x,\y) .. controls (\x+0.2,\y+0.3) and (\x+0.6,\y+0.3) .. (\x+0.8,\y);
\foreach \x in {7.2,8.0,8.8,9.6} 
	\foreach \y in {1.8}
		\draw (\x,\y) .. controls (\x+0.2,\y-0.3) and (\x+0.6,\y-0.3) .. (\x+0.8,\y);
\draw (5.6,2.3) .. controls (6.9,2.9) and (9.9,2.9) .. (11.2,2.3);
\foreach \x in {4.8,5.6,...,11.5} 
	\foreach \y in {2.3,1.8,0.5}	
	\filldraw [black] (\x,\y) circle (1.1pt);
\foreach \x in {4.8,5.6,...,11.5} 
	\foreach \y in {1.8} 
 	\draw[dotted] (\x,\y) -- (\x,\y+0.5);
\foreach \x in {10.4,11.2}
	\draw (\x,0.5) -- (\x,1.8);
\foreach \x in {4.8,5.6,6.4,7.2}
	\draw (\x,0.5) -- (\x,1.8);
\foreach \x in {4.8,5.6,6.4,7.2,11.2}
	\draw (\x,2.3) -- (\x,2.6);
\foreach \x in {4.8,5.6,6.4,11.2}
 \draw[dashed] (\x,2.5) -- (\x,2.8);
\end{tikzpicture}
\caption{}
\label{figure:  partition even type case 1}
\end{figure}
}    

\begin{figure}[ht]
\begin{tikzpicture}
\foreach \x in {7.2,8.0,8.8,9.6} 
	\foreach \y in {0.5}
		\draw (\x,\y) 
		-- (\x+0.8,\y);
\foreach \x in {7.2,8.0,8.8,9.6} 
	\foreach \y in {1.8}
		\draw (\x,\y)  
		--(\x+0.8,\y);
\foreach \x in {4.8,5.6,...,11.5} 
	\foreach \y in {2.3,1.8,0.5}	
	\filldraw [black] (\x,\y) circle (1.1pt);
\foreach \x in {4.8,5.6,...,11.5} 
	\foreach \y in {1.8} 
 	\draw[dotted] (\x,\y) -- (\x,\y+0.5);
\foreach \x in {10.4,11.2}
	\draw (\x,0.5) -- (\x,1.8);
\foreach \x in {4.8,5.6,6.4,7.2}
	\draw (\x,0.5) -- (\x,1.8);
\foreach \x in {4.8,5.6,6.4,7.2,11.2}
	\draw (\x,2.3) -- (\x,2.6);
\foreach \x in {4.8,5.6,6.4,7.2, 11.2}
 \draw[dashed] (\x,2.5) -- (\x,2.8);
    \node at (7.2, 0) {$ k-\el $};
 \node at (10.4, 0) {$k$};
\end{tikzpicture}
\caption{}
\label{figure:  partition even type case 1}
\end{figure}

Finally, if $m = \el + 1$,  then $\la = \mu$ and $\dd {(\la, \el)}  {(\mu, \el + 1)} {2k} = e_{2k - 2} \power \el$. 
Again the product  $\varrho \dd {(\la, \el)}  {(\mu, \el)}  {2k + 1}  = \varrho  e_{2k - 2} \power \el$ is a single partition, of odd type $(k, \el)$,  and no power of $\deltabold$ occurs in the computation of the product.
The diagram for this case is similar to 
 \hyperref[figure:  partition even type case 1]{Figure~\ref*{figure:  partition even type case 1}}, except that the lower vertex $\pbar{k - \el}$ of $\varrho$ is now an singleton block of  $\varrho$.
\end{proof}

\begin{proposition}  \label{proposition:  transition matrix  for partition algebra}
Let $(\lambda,l)\in\widehat{A}_n$ and $\mathfrak{s,t}\in\widehat{A}_n^{(\lambda,l)}$. Then $d_\mathfrak{s}^*c_{(\lambda,l)}d_\mathfrak{t}=d_\mathfrak{s}^*c_{(\lambda,0)}e_{n-1}^{(l)}d_\mathfrak{t}$ lies in the $\mathbb{Z}$-span of $\mathcal{B}$. 
\end{proposition}
\begin{proof}   If $n = 2k + 1$ is odd, then $c_{(\lambda,0)}e_{n-1}^{(\el)}$  is a sum of partitions of odd type $(k+1, \el)$.    If $n = 2k$ is even, then $c_{(\lambda,0)}e_{n-1}^{(\el)}$  is a sum of partitions of even type $(k, \el)$.    The argument proceeds as in the proof of \hyperref[proposition:  expansion of GE basis in MW basis]{Proposition~\ref*{proposition:  expansion of GE basis in MW basis}}, with \hyperref[lemma: partition types]{Lemma~\ref*{lemma: partition types}} taking the place of \hyperref[lemma:  expansion of GE basis in MW basis 1]{Lemma~\ref*{lemma:  expansion of GE basis in MW basis 1}}. 
\end{proof}

\begin{theorem}  The set $\mathscr{A}_n = \{
d_\mathfrak{s}^*c_{(\lambda,l)}d_\mathfrak{t}   \mid      
\text{$\mathfrak{s},\mathfrak{t}\in\hat{A}_n^{(\lambda,l)}$, $(\lambda,l)\in\hat{A}_n$}
\},$ 
is a basis for the partition algebra $A_n(R_0; \deltabold)$  over the generic ground ring $R_0 = \Z[\deltabold]$.
\end{theorem}

\begin{proof}  The transition matrix $B$ between the diagram basis of the partition algebra and the set $\mathscr A_n$  has integer entries, according to \hyperref[proposition:  transition matrix  for partition algebra]{Proposition~\ref*{proposition:  transition matrix  for partition algebra}}, and in particular $d = \det(B)$ is an integer.    Since $\mathscr A_n$ is a basis for the partition algebra over $R = \Z[\deltabold^{\pm 1}]$, it follow that $B$ is invertible over $R$, so the integer $d$ is a unit in $R$.  It follows that $d = \pm 1$  and hence $B$ is invertible over $\Z$.  Hence $\mathscr A_n$ is a basis  of the partition algebra over $R_0$.    
\end{proof}

\begin{corollary} \label{corollary:  cell filtrations for partition algebra over generic ground ring}
Let $A_n$ denote the partition algebra over the generic ground ring $R_0 = \Z[\deltabold]$.  
For  $n \ge 0$ and for  $\Delta$ a cell module of  $A_{n+1}$,    the restricted module $\Res^{A_{n+1}}_{A_n}(\Delta)$ has an order preserving cell filtration.
\end{corollary}

\begin{proof}  The proof is the same as that of  \hyperref[corollary:   cell filtrations for BMW over generic ground ring]{Corollary~\ref*{corollary:   cell filtrations for BMW over generic ground ring}}.
\end{proof}

\appendix 

\section{A formula for Murphy basis elements}

In this appendix, we give an alternative formula for the Murphy basis of the Iwahori--Hecke algebra $\mathcal{H}_n(q^2)$ and for the Murphy--type bases of the various algebras treated in \hyperref[section: applications]{Section~\ref*{section: applications}}.  The formula was pointed out to us by Chris Bowman for the case of the Hecke algebra, and Bowman posed the question whether an analogous formula holds also for the BMW algebras, etc.

Consider a tower $(H_n)_{n \ge 0}$ of cyclic cellular algebras satisfying the hypotheses of \hyperref[subsection: bases of cell modules for coherent cyclic towers]{Section~\ref*{subsection: bases of cell modules for coherent cyclic towers}}.  As we will show, in all the examples  of such towers treated in this paper, the elements $c_\la$ and the branching factors $d\power {n+1}_{\mu \to \nu}$ and 
$u\power {n+1}_{\mu \to \nu}$ can be chosen to satisfy
\begin{equation} \label{intertwining formula 1}
c_\mu u\power {n+1}_{\mu \to \nu} = (d\power {n+1}_{\mu \to \nu})^*  c_\nu
\end{equation}
for all $n \ge 0$  and all $\mu \in \widehat H_n$ and $\nu \in \widehat H_{n+1}$  with $\mu \to \nu$.  

We define an ordered product of $u$--coeffcients along paths, analogous to the elements $d_\mft$ defined in Equation~\eqref{equation: ordered product of d coefficients}. 

Fix $n \ge 1$ and $\la \in \widehat H_n$.  For each path
$\mft = (\emptyset = \la \power 0, \la \power 1, \dots, \la \power n = \la) \in \widehat H_n^\la$,  define
\begin{equation} \label{equation: ordered product of u coefficients}
u_\mft =    {\uu {\emptyset} {\la \power 1} {1}}
{\uu {\la \power {1}} {\la \power 2} {2}} \,
\cdots
\, {\uu {\la \power {n-1}} {\la \power n} {n}} 
.\end{equation}

\begin{lemma}  \label{lemma: intertwining formula 2}  
Let $(H_n)_{n \ge 0}$ be a tower of cyclic cellular algebras satisfying the hypotheses of \hyperref[subsection: bases of cell modules for coherent cyclic towers]{Section~\ref*{subsection: bases of cell modules for coherent cyclic towers}}.   Suppose that
  Equation~\eqref{intertwining formula 1} holds for all $n \ge 0$  and all $\mu \in \widehat H_n$ and $\nu \in \widehat H_{n+1}$  with $\mu \to \nu$.   Then for all $n \ge 0$,  all  $\la \in \widehat H_n$ and all
  $\mft \in \widehat H_n^\la$,   one has
  \begin{equation}  \label{equation: intertwining formula 2}  
  d^*_\mft c_\la =  u_t.
 \end{equation}
 Consequently, the cellular basis of $H_n$ given in \hyperref[label:1]{Corollary~\ref*{label:1}} can be written as 
 $$
 \leftbrace u_\mfs  d_\mft \ \big \vert \   \la \in \widehat H_n  \text{ and }  \mfs, \mft \in  \widehat H_n^\la \rightbrace.
 $$
\end{lemma}

\begin{proof}  The formula~\eqref{equation: intertwining formula 2} is evident for $n = 0, 1$.  Fix $n \ge 1$ and suppose that~\eqref{equation: intertwining formula 2} holds for all $\la \in \widehat H_n$ and all
  $\mft \in  \widehat H_n^\la$.    Let $\nu \in \widehat H_{n+1}$  and
  $\mft = (\emptyset, \la\power 1, \dots, \la \power n = \mu, \la \power {n+1} = \nu)$ be an element of
  $\widehat H_{n+1}^\nu$.   Write $\mft' = \mft_{[0, n]}$.    Then, using the induction hypothesis as well as Equation~\eqref{intertwining formula 1}, we have
  $$
  \begin{aligned}
  u_\mft &= u_{\mft'}  \uu \mu \nu {n+1} 
  =  d_{\mft'}^* c_\mu \uu \mu \nu {n+1} 
  =  d_{\mft'}^*  (d\power {n+1}_{\mu \to \nu})^*  c_\nu  = d_\mft^*  c_\nu.
  \end{aligned}
  $$
 The  statement now follows by induction. 
\end{proof}

\begin{lemma}  \label{lemma:  intertwining formula for Hecke algebras}
The branching factors $\dd \mu \nu {n+1}$  and $\uu \mu \nu {n+1}$ for the tower of  Iwahori--Hecke algebras of the symmetric groups,  as determined in \hyperref[corollary:  d branching factors for the Hecke algebra]{Corollary~\ref*{corollary:  d branching factors for the Hecke algebra}} and \hyperref[corollary: u branching factors for the Hecke algebra]{Corollary~\ref*{corollary: u branching factors for the Hecke algebra}}, satisfy
$$
m_\mu \uu \mu \nu {n+1} = (\dd \mu \nu {n+1})^*  m_\nu,
$$
for all $n \ge 0$ and all partitions $\mu$ of size $n$ and $\nu$ of size $n+1$ with $\mu \to \nu$.  
\end{lemma}

\begin{proof} This is immediate from \hyperref[lemma:   explicit isomorphism induced perm module mod N with induced cell module]{Lemma~\ref*{lemma:   explicit isomorphism induced perm module mod N with induced cell module}},  part (1).
\end{proof}

\begin{corollary}  The Murphy basis of the Iwahori--Hecke algebra $\mathcal H_n(q^2)$   is given by
$$
m^\lambda_{\mfs \mft}  = u_\mfs d_\mft
$$
for $\la$ a partition of $n$  and $\mfs, \mft$ standard $\la$--tableaux.
\end{corollary}

\begin{proof} This follows from Equation~\eqref{equation:  formula for Murphy basis 1} and  \hyperref[lemma: intertwining formula 2]{Lemma~\ref*{lemma: intertwining formula 2}} and \hyperref[lemma:  intertwining formula for Hecke algebras]{Lemma~\ref*{lemma:  intertwining formula for Hecke algebras}}.
\end{proof}

\ignore{
Our next goal is to obtain similar formulas for the Murphy type bases of various algebras treated in \hyperref[section: applications]{Section~\ref*{section: applications}}.  To this end,  consider a pair of towers of cyclic cellular algebras $(H_n)_{n \ge 0}$ and $(A_n)_{n \ge 0}$  satisfying axioms (1)--(9) of \hyperref[subsection axioms]{Section~\ref*{subsection axioms}}.    Suppose that elements $c_\la$ and branching factors $\dd \mu \nu {n+1}$  and $\uu \mu \nu {n+1}$ have been chosen for the tower $(H_n)_{n \ge 0}$.  For $\la \in \widehat H_n$,  let $c_{(\la, 0)}$ be an element of $\pi_n\inv(c_\la)$, and for $(\la, \el) \in \hat A_n$,  define $c_{(\la, \el)} =  c_{(\la, 0)}  e_{n-1} \power \el$.   Let
$\ddbar \mu \nu {n+1}$ be a lift of $\dd \mu \nu {n+1}$ in $A_{n+1}$  and likewise 
$\uubar \mu \nu {n+1}$ a lift of $\uu \mu \nu {n+1}$ in $A_{n+1}$, and determine a system of branching factors for all edges in $\hat A$  by the formulas of \hyperref[theorem:  closed form determination of the branching factors]{Theorem~\ref*{theorem:  closed form determination of the branching factors}}.
}

Our next goal is to obtain similar formulas for the Murphy type bases of the various algebras treated in \hyperref[ection: applications]{Section~\ref*{section: applications}}.

\begin{proposition}   Let $A_n$ denote the $n$--th  BMW, Brauer, partition or Jones--Temperley--Lieb algebra.  The Murphy type basis of $A_n$ established in \hyperref[section: applications]{Section~\ref*{section: applications}} can be written in the form
$$
\leftbrace u_\mfs d_\mft \ \big | \  \la \in \hat A_n \text{ and }  \mfs, \mft \in \hat A_n^\la \rightbrace.
$$
\end{proposition}

\noindent{\em Sketch of proof.}   We need to show that if
$x \in \hat A_n$ and $y \in \hat A_{n+1}$  with $x \to y$ in the branching diagram $\hat A$,  then
\begin{equation} \label{equation: intertwining 3}
c_x \uu x y {n+1} = (\dd x y {n+1})^*  c_y,
\end{equation} 
where the elements $c_x \in A_n$  and $c_y,  \uu x y {n+1}, \dd x y {n+1} \in A_{n+1}$  are as specified in Section 6.  The result will then follow from \hyperref[lemma: intertwining formula 2]{Lemma~\ref*{lemma: intertwining formula 2}}.  For the Temperely--Lieb algebras, Equation~\eqref{equation: intertwining 3} is evident from the formulas in \hyperref[subsection: temperley lieb algebras]{Section~\ref*{subsection: temperley lieb algebras}} for the elements $c_x$ and for the branching factors.

For the BMW, Brauer and partition algebras,~\eqref{equation: intertwining 3} can be established in two steps.  The first step is to show that~\eqref{equation: intertwining 3} holds when $x = (\la, 0) \in \hat A_n$  and $y = (\hat \mu, 0) \in \hat A_{n+1}$.   For the Brauer and partition algebras this special case of~\eqref{equation: intertwining 3}  follows from \hyperref[lemma:   explicit isomorphism induced perm module mod N with induced cell module]{Lemma~\ref*{lemma:   explicit isomorphism induced perm module mod N with induced cell module}},  part (1), as all the elements involved  lie in a copy of the symmetric group algebra contained in $A_{n+1}$.   For the BMW algebras, it is necessary to establish an analogue of \hyperref[lemma:   explicit isomorphism induced perm module mod N with induced cell module]{Lemma~\ref*{lemma:   explicit isomorphism induced perm module mod N with induced cell module}},  part (1) which is valid in the algebra of the braid group.

The second step in the proof of~\eqref{equation: intertwining 3} is to establish the general case from the special case.  This involves a straightforward computation using the formulas of \hyperref[theorem:  closed form determination of the branching factors]{Theorem~\ref*{theorem:  closed form determination of the branching factors}}. 
\qed


\def\Dbar{\leavevmode\lower.6ex\hbox to 0pt{\hskip-.23ex \accent"16\hss}D}
\providecommand{\bysame}{\leavevmode\hbox to3em{\hrulefill}\thinspace}
\providecommand{\MR}{\relax\ifhmode\unskip\space\fi MR }
\providecommand{\MRhref}[2]{%
  \href{http://www.ams.org/mathscinet-getitem?mr=#1}{#2}
}
\providecommand{\href}[2]{#2}



\begin{thebibliography}{10}

\bibitem{MR1750939}
Susumu Ariki and Andrew Mathas, \emph{The number of simple modules of the
  {H}ecke algebras of type {$G(r,1,n)$}}, Math. Z. \textbf{233} (2000), no.~3,
  601--623. \MR{1750939}

\bibitem{MR1503378}
Richard Brauer, \emph{On algebras which are connected with the semisimple
  continuous groups}, Ann. of Math. (2) \textbf{38} (1937), no.~4, 857--872.
  \MR{1503378}

\bibitem{MR812444}
Richard Dipper and Gordon James, \emph{Representations of {H}ecke algebras of
  general linear groups}, Proc. London Math. Soc. (3) \textbf{52} (1986),
  no.~1, 20--52. \MR{812444}

\bibitem{MR2811310}
James East, \emph{Generators and relations for partition monoids and algebras},
  J. Algebra \textbf{339} (2011), 1--26. \MR{2811310}

\bibitem{MR2348099}
John Enyang, \emph{Specht modules and semisimplicity criteria for {B}rauer and
  {B}irman-{M}urakami-{W}enzl algebras}, J. Algebraic Combin. \textbf{26}
  (2007), no.~3, 291--341. \MR{2348099}

\bibitem{MR3092697}
\bysame, \emph{A seminormal form for partition algebras}, J. Combin. Theory
  Ser. A \textbf{120} (2013), no.~7, 1737--1785. \MR{3092697}

\bibitem{MR3065998}
T.~Geetha and Frederick~M. Goodman, \emph{Cellularity of wreath product
  algebras and {$A$}-{B}rauer algebras}, J. Algebra \textbf{389} (2013),
  151--190. \MR{3065998}

\bibitem{MR2510050}
Frederick~M. Goodman, \emph{Cellularity of cyclotomic
  {B}irman-{W}enzl-{M}urakami algebras}, J. Algebra \textbf{321} (2009),
  no.~11, 3299--3320. \MR{2510050}

\bibitem{MR2794027}
Frederick~M. Goodman and John Graber, \emph{Cellularity and the {J}ones basic
  construction}, Adv. in Appl. Math. \textbf{46} (2011), no.~1-4, 312--362.
  \MR{2794027}

\bibitem{MR2774622}
\bysame, \emph{On cellular algebras with {J}ucys {M}urphy elements}, J. Algebra
  \textbf{330} (2011), 147--176. \MR{2774622}

\bibitem{GKT-2014}
Frederick~M. Goodman, Ross Kilgore, and Nicholas Teff, \emph{Restrictions of
  cell modules of the {H}ecke algebra of the symmetric group}, preprint (2014).

\bibitem{MR1376244}
J.~J. Graham and G.~I. Lehrer, \emph{Cellular algebras}, Invent. Math.
  \textbf{123} (1996), no.~1, 1--34. \MR{1376244}

\bibitem{MR2143201}
Tom Halverson and Arun Ram, \emph{Partition algebras}, European J. Combin.
  \textbf{26} (2005), no.~6, 869--921. \MR{2143201}

\bibitem{MR1680384}
Martin H{\"a}rterich, \emph{Murphy bases of generalized {T}emperley-{L}ieb
  algebras}, Arch. Math. (Basel) \textbf{72} (1999), no.~5, 337--345.
  \MR{1680384}

\bibitem{MR696688}
V.~F.~R. Jones, \emph{Index for subfactors}, Invent. Math. \textbf{72} (1983),
  no.~1, 1--25. \MR{696688}

\bibitem{MR1317365}
\bysame, \emph{The {P}otts model and the symmetric group}, Subfactors
  ({K}yuzeso, 1993), World Sci. Publ., River Edge, NJ, 1994, pp.~259--267.
  \MR{1317365 (97b:82023)}

\bibitem{MR766964}
Vaughan F.~R. Jones, \emph{A polynomial invariant for knots via von {N}eumann
  algebras}, Bull. Amer. Math. Soc. (N.S.) \textbf{12} (1985), no.~1, 103--111.
  \MR{766964}

\bibitem{MR1461487}
Thomas Jost, \emph{Morita equivalence for blocks of {H}ecke algebras of
  symmetric groups}, J. Algebra \textbf{194} (1997), no.~1, 201--223.
  \MR{1461487}

\bibitem{MR1648638}
Steffen K{\"o}nig and Changchang Xi, \emph{On the structure of cellular
  algebras}, Algebras and modules, {II} ({G}eiranger, 1996), CMS Conf. Proc.,
  vol.~24, Amer. Math. Soc., Providence, RI, 1998, pp.~365--386. \MR{1648638}

\bibitem{MR1753809}
\bysame, \emph{Cellular algebras: inflations and {M}orita equivalences}, J.
  London Math. Soc. (2) \textbf{60} (1999), no.~3, 700--722. \MR{1753809}

\bibitem{MR1768036}
P.~P. Martin, \emph{The partition algebra and the {P}otts model transfer matrix
  spectrum in high dimensions}, J. Phys. A \textbf{33} (2000), no.~19,
  3669--3695. \MR{1768036}

\bibitem{MR1103994}
Paul Martin, \emph{Potts models and related problems in statistical mechanics},
  Series on Advances in Statistical Mechanics, vol.~5, World Scientific
  Publishing Co. Inc., Teaneck, NJ, 1991. \MR{1103994}

\bibitem{MR1265453}
\bysame, \emph{Temperley-{L}ieb algebras for nonplanar statistical
  mechanics---the partition algebra construction}, J. Knot Theory Ramifications
  \textbf{3} (1994), no.~1, 51--82. \MR{1265453}

\bibitem{MR1711316}
Andrew Mathas, \emph{Iwahori-{H}ecke algebras and {S}chur algebras of the
  symmetric group}, University Lecture Series, vol.~15, American Mathematical
  Society, Providence, RI, 1999. \MR{1711316}

\bibitem{MR2414949}
\bysame, \emph{Seminormal forms and {G}ram determinants for cellular algebras},
  J. Reine Angew. Math. \textbf{619} (2008), 141--173, With an appendix by
  Marcos Soriano. \MR{2414949 (2009e:16059)}

\bibitem{MR2531227}
\bysame, \emph{A {S}pecht filtration of an induced {S}pecht module}, J. Algebra
  \textbf{322} (2009), no.~3, 893--902. \MR{2531227}

\bibitem{Morton-Wassermann}
Hugh~R. Morton and Anthony~J. Wassermann, \emph{A basis for the
  {B}irman-{W}enzl algebra}, Unpublished manuscript (1989, revised 2000),
  1--29, placed on the arXiv by H.R. Morton in 2010. arXiv 1012.3116, 2010.

\bibitem{MR1327362}
G.~E. Murphy, \emph{The representations of {H}ecke algebras of type {$A\sb
  n$}}, J. Algebra \textbf{173} (1995), no.~1, 97--121. \MR{1327362}

\bibitem{MR2542221}
Hebing Rui and Mei Si, \emph{Gram determinants and semisimplicity criteria for
  {B}irman-{W}enzl algebras}, J. Reine Angew. Math. \textbf{631} (2009),
  153--179. \MR{2542221 (2010h:16075)}

\bibitem{MR2912472}
\bysame, \emph{The representations of cyclotomic {BMW} algebras, {II}}, Algebr.
  Represent. Theory \textbf{15} (2012), no.~3, 551--579. \MR{2912472}

\bibitem{MR951511}
Hans Wenzl, \emph{On the structure of {B}rauer's centralizer algebras}, Ann. of
  Math. (2) \textbf{128} (1988), no.~1, 173--193. \MR{951511}

\bibitem{MR1711582}
Changchang Xi, \emph{Partition algebras are cellular}, Compositio Math.
  \textbf{119} (1999), no.~1, 99--109. \MR{1711582}

\bibitem{MR1784677}
\bysame, \emph{On the quasi-heredity of {B}irman-{W}enzl algebras}, Adv. Math.
  \textbf{154} (2000), no.~2, 280--298. \MR{1784677}

\end{thebibliography}
\end{document}